\documentclass[a4paper,11pt]{article}
\usepackage{pdfsync}

\usepackage[plainpages=false]{hyperref}
\hypersetup{pdftitle={On the rigidity of Ricci shrinkers},pdfauthor={Yu Li and Wenjia Zhang},pdfsubject={On the rigidity of Ricci shrinkers}}
\usepackage{amsfonts,amsmath,amssymb,amsthm,enumerate}
\usepackage{latexsym,lscape,rawfonts,mathrsfs,mathtools}
\usepackage{enumitem,pifont}
\usepackage{cases}
\usepackage[savepos]{zref}

\usepackage[all]{xy}
\usepackage{eufrak}
\usepackage{makeidx}
\usepackage{graphicx,psfrag}
\usepackage{pstool}
\usepackage{tikz-cd}
\usepackage{smartdiagram}
\usepackage{caption}
\usepackage{array,tabularx,tabu}

\usepackage{setspace}

\usepackage[titletoc,title]{appendix}

\usepackage{txfonts}

\newcommand{\IF}{\ensuremath{\mathbb{F}}}

\newcommand{\N}{\ensuremath{\mathbb{N}}}

\newcommand{\R}{\ensuremath{\mathbb{R}}}
\newcommand{\RR}{\ensuremath{\mathcal{R}}}
\newcommand{\MS}{\ensuremath{\mathcal{S}}}

\newcommand{\Z}{\ensuremath{\mathbb{Z}}}
\newcommand{\CP}{\ensuremath{\mathbb{CP}}}

\newcommand{\LL}{\mathcal{L}}
\newcommand{\MM}{\mathcal{M}}
\newcommand{\NN}{\mathcal{N}}
\newcommand{\PP}{\mathcal{P}}
\newcommand{\KK}{\mathcal{K}}

\newcommand{\ba}{\begin{align*}}
\newcommand{\ea}{\end{align*}}
\newcommand{\na}{\nabla}
\newcommand{\la}{\langle}
\newcommand{\ra}{\rangle}
\newcommand{\lc}{\left(}
\newcommand{\rc}{\right)}

\newcommand{\ep}{\epsilon}

\newcommand{\ka}{K\"ahler\,}

\newcommand{\im}{\ensuremath{\mathrm{Im}}}
\newcommand{\tr}{\ensuremath{\mathrm{Tr}}}
\newcommand{\Rm}{\ensuremath{\mathrm{Rm}}}
\newcommand{\Rc}{\ensuremath{\mathrm{Ric}}}

\newcommand{\bphi}{\ensuremath{\boldsymbol{\Phi}}}

\newcommand{\mmu}{\ensuremath{\boldsymbol{\mu}}}
\newcommand{\nnu}{\ensuremath{\boldsymbol{\nu}}}

\newcommand{\isd}{\mathrm{ISD}}

\newcommand{\df}{\delta_{f}}
\newcommand{\dbf}{\delta_{\bar{f}}}
\newcommand{\ddf}{\Delta_{\bar f}}

\newcommand{\vf}{\mathrm{d}V_{\bar f}}
\newcommand{\XX}{\mathcal{X}}

\newcommand{\Var}{\textrm{Var}}
\newcommand{\CF}{\mathfrak{C}}

\makeatletter
\newcommand*\owedge{\mathpalette\@owedge\relax}
\newcommand*\@owedge[1]{%
\mathbin{%
\ooalign{%
$#1\m@th\bigcirc$\cr
\hidewidth$#1\m@th\wedge$\hidewidth\cr
}%
}%
}
\makeatother

\makeatletter
\def\ExtendSymbol#1#2#3#4#5{\ext@arrow 0099{\arrowfill@#1#2#3}{#4}{#5}}

\makeatother

\makeatletter
\def\ExtendSymbol#1#2#3#4#5{\ext@arrow 0099{\arrowfill@#1#2#3}{#4}{#5}}
\newcommand\longright[2][]{\ExtendSymbol{-}{-}{\rightarrow}{#1}{#2}}
\makeatother

\def\XXint#1#2#3{{\setbox0=\hbox{$#1{#2#3}{\int}$ }
\vcenter{\hbox{$#2#3$ }}\kern-.55\wd0}}

\numberwithin{equation}{section}

\newtheorem{thm}{Theorem}[section]
\newtheorem{cor}[thm]{Corollary}
\newtheorem{prop}[thm]{Proposition}
\newtheorem{lem}[thm]{Lemma}

\newtheorem{rem}[thm]{Remark}

\newtheorem{defn}[thm]{Definition}

\newtheorem{exmp}[thm]{Example}

\title{On the rigidity of Ricci shrinkers}
\author{Yu Li \quad and \quad Wenjia Zhang}
\date{\today}

\begin{document}
\maketitle

\begin{abstract}
In this paper, we establish the rigidity of the generalized cylinder $N^n \times \R^{m-n}$, or a quotient thereof, in the space of Ricci shrinkers equipped with the pointed Gromov--Hausdorff topology. Here, $N$ is an $H$-stable Einstein manifold that has an obstruction of order $3$. The proof is based on a quantitative characterization of the rigidity of compact Ricci shrinkers, a rigidity inequality of mixed orders on generalized cylinders, and the method of contraction and extension. As an application, we prove the uniqueness of the tangent flow for general compact Ricci flows under the assumption that one tangent flow is such a generalized cylinder or a finite smooth quotient thereof.
\end{abstract}
\tableofcontents
\section{Introduction}
A Ricci shrinker $(M, g, f)$ is a complete Riemannian manifold $(M,g)$ together with a smooth function $f: M \to \mathbb{R}$ such that
\begin{equation} 
\Rc+\na^2 f=\frac{1}{2}g,
\label{E100}
\end{equation}
where the potential function $f$ is normalized such that
\begin{align} 
\mathrm{R}+|\nabla f|^2&=f \label{E101}.
\end{align}

As the critical metrics of Perelman's $\mmu$-functional, Ricci shrinkers play a crucial role in analyzing the singularity formation of the Ricci flow. Enders--M\"uller--Topping \cite{EMT11} proved that parabolic rescalings about a type-I singular point subconverge smoothly to a nontrivial Ricci shrinker. Furthermore, Bamler \cite{Bam20c} has shown that the finite-time singularities of general compact Ricci flows are modeled on Ricci shrinkers with a possibly singular underlying space, by using the theory of $\IF$-convergence developed in \cite{Bam20a, Bam20b, Bam20c}.

In dimensions $2$ and $3$, all Ricci shrinkers are completely classified (cf.~\cite{Ha95,Naber,NW,CCZ}, etc). We know that $\R^2,S^2,\R^3,S^3,S^2 \times \R$ and their quotients form the complete list. In higher dimensions, the classification is far from complete. We apply the strategy from \cite{LLW21} to consider all Ricci shrinkers of the same dimension $n$ as a moduli space $\MM_n$ equipped with the pointed Gromov--Hausdorff distance. For pointed convergence in $\MM_n$, we choose basepoints at minima of the normalized potential $f$. In \cite{LLW21} and \cite{HLW21}, the weak-compactness theory is established for the moduli subspace $\MM_n(A)$, which consists of all Ricci shrinkers with $\mmu \ge -A$. More precisely, any sequence of Ricci shrinkers in $\MM_n(A)$ subconverges to a Ricci shrinker conifold, and the convergence is smooth away from the singularities. In particular, if the limiting Ricci shrinker conifold is smooth, the convergence is improved to be pointed Cheeger--Gromov.

The question of rigidity in geometry pertains to whether a particular geometric object is uniquely determined within its moduli space, i.e., whether it admits any non-trivial deformations. In the context of Ricci shrinkers, the rigidity problem concerns whether a given Ricci shrinker is isolated in the moduli space $\MM_n$. The analogous question was studied by Koiso for compact Einstein manifolds \cite{Ko78,Ko80,Ko82}. Later, Podest\`a-Spiro \cite{PS15} and Kr\"oncke \cite{Kr16} established the deformation theory for compact Ricci shrinkers. Kr\"oncke developed finite-order obstructions to the integrability of infinitesimal solitonic deformations (ISDs; see Definitions~\ref{def:ISD} and~\ref{D201}).

For example, the standard sphere $S^n$ and its quotients are rigid since $\textrm{ISD}=0$. Kr\"oncke \cite{Kr16} proved that $\CP^{2m}$ with the standard metric is also rigid by showing that every nonzero $h \in \isd$ is not integrable up to order $2$. More recently, the rigidity of $S^2 \times S^2$ and $\CP^{2m-1}$ was established by Sun--Zhu \cite{SZ21} and Li--Zhang \cite{LZ22}, respectively, using third-order deformation calculations.

The first main result of this paper is the following new characterization of the rigidity of a compact Ricci shrinker.

\begin{thm} \label{thm:001}
A compact Ricci shrinker $(N^n,\bar g,\bar f)$ is rigid if and only if it satisfies the rigidity inequality of some order $k$, i.e., there exist positive constants $\epsilon$ and $C$ such that the following property holds:

For any $g \in C^{2,\alpha}(S^2(N))$ with $\|g- \bar g \|_{C^{2,\alpha}} < \epsilon$, there exists a $C^{3,\alpha}$ self-diffeomorphism $\varphi$ of $N$ such that $\|\varphi- \mathrm{Id} \|_{C^{3,\alpha}} \le C \|\dbf(g- \bar{g}) \|_{C^{1,\alpha}}$ and
\begin{align}\label{E001a}
\|\varphi^{*} g - \bar g \|^{k}_{C^{2,\alpha}} \leq C \| \Phi(g) \|_{C^{\alpha}}.
\end{align}
\end{thm}
In Theorem \ref{thm:001}, $\alpha\in(0,1)$ is fixed and $\Phi$ is the Ricci shrinker operator in Definition \ref{def:op1}. The diffeomorphism fixes the gauge by imposing $\dbf(\varphi^*g-\bar g)=0$. We define the obstruction of order $k$ through the induced sequence in Definition \ref{def:ind}, and prove that it is equivalent to the rigidity inequality of the same order (Theorem \ref{thm:hold} and Figure \ref{dia:1}). Rigidity implies an obstruction of finite order by analytic finite-dimensional reduction and the \L{}ojasiewicz inequality.

The question of rigidity is considerably more complex for non-compact Ricci shrinkers. The primary challenge arises because the concept of closeness under the pointed Gromov--Hausdorff or pointed Cheeger--Gromov metric is not global. By definition, a non-compact Ricci shrinker $(\bar M^m, \bar g, \bar f)$ is said to be rigid if any Ricci shrinker $(M^m,g,f)$ that is sufficiently close to $(\bar M^m, \bar g, \bar f)$ on a sufficiently large compact set must be isometric to it. This means that uniqueness radiates out from a compact set, a principle known as the shrinker principle, which was originally discovered in the context of mean curvature flow \cite{CIM15,CM20} and subsequently applied to Ricci flow by Colding--Minicozzi \cite{CM21b}.

The Gaussian soliton $(\R^m,g_E)$ is the first non-compact example of a rigid Ricci shrinker. This rigidity has been independently proved by Yokota \cite{Yo09, Yo12} and Li--Wang \cite{LW20} by using an entropy-gap argument. Recently, Li--Wang \cite{LW21} have also shown that the standard round cylinder $S^{n-1} \times \R$ is rigid. Independently, Colding--Minicozzi \cite{CM21b} have proved that any standard $S^{n-k} \times \R^k$ is also rigid.

In this paper, we focus on the rigidity of the generalized cylinder $(\bar M^m, \bar g, \bar f)=(N^n \times \R^{m-n},g_N \times g_E, |z|^2/4+n/2)$, where $(N,g_N)$ is an $H$-stable Einstein manifold (cf. Definition \ref{def:Hstable}) that has an obstruction of order $3$, and $z$ is the coordinate of $\R^{m-n}$. Rigidity of the cylinder requires rigidity of its compact factor. Under these assumptions, the infinitesimal solitonic deformation space splits into the deformations induced from $N$ and the tensors $ug_N$, where $u$ is a quadratic Hermite polynomial on $\R^{m-n}$. The main result of this paper is to confirm the rigidity of generalized cylinders of this type and their quotients.

\begin{thm}[Main Theorem] \label{thm:002}
Suppose $(N^n,g_N)$ is an $H$-stable Einstein manifold with $\Rc(g_N)=g_N/2$ that has an obstruction of order $3$, and let $(\bar M^m, \bar g, \bar f)$ denote either the generalized cylinder $(N^n \times \R^{m-n},g_N \times g_E, |z|^2/4+n/2)$ or a finite smooth quotient thereof. Then any $m$-dimensional Ricci shrinker $(M,g,f,p)$ that is close to $(\bar M,\bar g,\bar f,\bar p)$ in the pointed Gromov--Hausdorff sense must be isometric to $(\bar M,\bar g)$.
\end{thm}

We explain the main ingredients of Theorem \ref{thm:002} for the model
\[
 (\bar M^m,\bar g,\bar f)=(N^n\times\R^{m-n},g_N\times g_E,|z|^2/4+n/2).
\]

(I) Rigidity inequality on generalized cylinders.

On a non-compact manifold, a minimizer of $\mmu(\cdot,1)$ need not exist, so we treat the metric and potential as an independent pair $(g,f)$ near $(\bar g,\bar f)$. Weighted elliptic estimates for the stability operator $\LL$ (Theorem \ref{thm:est1}) provide the linear control needed for a rigidity inequality.

The deformation space splits as $\isd_N\oplus\KK_1$, where $\KK_1$ consists of tensors $ug_N$ with $u$ a quadratic Hermite polynomial (Proposition \ref{PX303}). The compact modes have an obstruction of order $3$ by hypothesis, while the Euclidean modes have an obstruction of order $2$, as in \cite{CM21b}. We estimate their mixed terms to obtain a rigidity inequality with these two orders (Theorem \ref{thm:sta1}).

(II) The method of contraction and extension.

The rigidity inequality concerns deformations on the fixed model manifold, whereas pointed convergence initially supplies comparison maps only on compact regions. We use contraction and extension, following \cite{CIM15,CM21b}, to construct comparison embeddings on successively larger model domains. Contraction improves the estimates on a slightly smaller domain; extension enlarges the domain with a controlled loss in the estimates. Iteration yields a global isometry.

(IIa) The contraction process.

For a cutoff deformation $\{h,\chi\}=\{g-\bar g,f-\bar f\}$, solve the linear gauge equation $2\PP w=\dbf h$ and apply the flow of a cutoff of $w$. The weighted inverse estimate (Proposition \ref{P301}) and the pointwise growth estimate (Theorem \ref{thm:diff2}) control this flow. A finite sequence of gauge corrections reduces the weighted divergence and center to the Gaussian cutoff scale (Lemma \ref{lem:weighted-gauge-iteration}). The Gaussian weight makes the cutoff error exponentially small. The mixed rigidity inequality and local interpolation then improve the metric and potential estimates on a slightly smaller model domain.

(IIb) The extension process.

Starting from a comparison embedding, pseudolocality and the gradient flow of the potential give a regular collar around its image. Radial coordinates extend the comparison map across this collar. We interpolate the pulled-back metric and potential with the model outside an inner region, obtaining a complete pair $(\bar M,g',f')$ whose shrinker defect is uniformly bounded and supported in a Gaussian tail.

The first $m-n$ nonconstant eigenfunctions of $\Delta_{f'}$ have eigenvalues close to $1/2$. Appendix \ref{app:A} gives polynomial growth of these functions and their gradients. The finite-radius frequency estimate in Lemma \ref{lem:finite-frequency} controls their Hessians in the exact-shrinker region inside the gluing annulus, using the polynomial bounds to control the outer boundary term. The resulting almost-parallel gradients produce a product chart on an enlarged domain.

The constructions are equivariant and descend to finite smooth quotients (Theorem \ref{thm:rigid2}).

The rigidity of $(\CP^n,g_{FS})$ was established in \cite{Kr16,LZ22}. Proposition \ref{P202} supplies the quantitative obstruction of order $3$ required in Theorem \ref{thm:002}. On the other hand, it is well-known that $(\CP^n,g_{FS})$ is $H$-stable. As a direct corollary of Theorem \ref{thm:002}, we obtain the following result. The stability properties of further compact symmetric spaces are listed in \cite[Table 1, Table 2]{CH15}.

\begin{cor}
The Ricci shrinker $(\CP^{n} \times \R^{m-2n},g_{FS} \times g_E,|z|^2/4+n)$ or the quotient thereof is rigid.
\end{cor}

Munteanu-Wang \cite[Corollary 4]{MW17} have proven that any Ricci shrinker with nonnegative curvature operator must be locally symmetric. Specifically in dimension $4$, $\R^4,S^4,\CP^2, S^2 \times S^2, S^2 \times \R^2, S^3 \times \R$ and their quotients are all possibilities. By combining Theorem \ref{thm:002} with the results from \cite{Yo09,Yo12,LW20,Kr16,SZ21,LW21} and \cite{CM21b}, we can deduce the following:
\begin{cor} 
Any four-dimensional Ricci shrinker with nonnegative curvature operator is rigid.
\end{cor}

Another application of Theorem \ref{thm:002} and related techniques is the proof of the uniqueness of the tangent flow for a compact Ricci flow, provided that one tangent flow is $(\bar M, \bar g, \bar f)$. In the case of the mean curvature flow, the uniqueness of the tangent flow at each cylindrical singular point is proven in \cite{CM15,CM19}. Moreover, under the type-I assumption on curvature, the corresponding uniqueness of tangent flow in the setting of the Ricci flow is also established; see \cite[Theorem 10.8]{CM21b}. In fact, the proof of \cite[Theorem 10.8]{CM21b} can be extended to any rigid Ricci shrinkers, since the convergence is always smooth under the type-I assumption. In general, we consider the tangent flows at a particular conjugate heat flow $(\mu_t)$ with controlled variance, which can be regarded as a ``point'' at the singular time (cf. Definition \ref{def:tangent}).

\begin{thm}[Uniqueness of the tangent flow]\label{thm:003}
Let $(M^m,g(t))_{t \in [-T,0)}$ be a compact Ricci flow solution with singular time $0$, and let $(\mu_t)_{t \in [-T,0)}$ be a conjugate heat flow satisfying $\emph{\Var}_t(\mu_t) \le H_m |t|$ for any $t \in [-T,0)$, where $H_m:=(m-1)\pi^2/2+4$. If one tangent flow at $(\mu_t)$ is $(\bar M, \bar g,\bar f)$, which satisfies the assumption of Theorem \ref{thm:002}, then any tangent flow at $(\mu_t)$ is.
\end{thm}

The proof of Theorem \ref{thm:003} uses Bamler's theory of $\IF$-convergence \cite{Bam20a,Bam20b,Bam20c}. We first obtain smooth convergence to the model on each fixed regular region. We then apply contraction and extension to embeddings of model domains into the regular part of a limiting soliton. Lemma \ref{lem:soliton-forward-continuation} gives the continuation of the Ricci-flow worldlines used in the extension. A boundary-distance estimate then shows that the resulting compact images exhaust the completed time slice. The slice is therefore smooth and complete, so Theorem \ref{thm:002} applies. The complete-approximant version yields the corresponding uniqueness result at infinity (Theorem \ref{thm:602}).

This paper is organized as follows. Section $2$ proves the equivalence between an obstruction of order $k$ and a rigidity inequality of order $k$, deduces Theorem \ref{thm:001}, and establishes the order-three obstruction for complex projective space. Section $3$ develops weighted elliptic estimates on non-compact Ricci shrinkers. Section $4$ proves the rigidity inequality of mixed orders for generalized cylinders. Section $5$ applies contraction and extension to prove Theorem \ref{thm:002}. Section $6$ discusses $\IF$-convergence and tangent flows and proves Theorem \ref{thm:003}. The Appendix establishes the growth estimates used for the scalar eigenfunctions and their gradients on the glued manifold, together with the finite-radius estimate for their Hessians.
\\
\\
\\
{\bf Acknowledgements}: 
Yu Li is supported by National Key R\&D Program of China 2025YFA1018200, NSFC12522105, YSBR-001 and research funds from the University of Science and Technology of China
and the Chinese Academy of Sciences. Both authors would like to thank Prof. Bing Wang for his interest in this work.
\section{Rigidity of the compact Ricci shrinkers}
For any compact Riemannian manifold $(N^n,g)$, we recall that Perelman's celebrated entropy $\boldsymbol{\mu}(g,1)$ \cite{Pe1} is defined as
\begin{align} \label{EX201}
\boldsymbol{\mu}(g,1) := \inf_\rho \lc (4 \pi)^{-\frac{n}{2}}\int_N \lc |\nabla \rho|^2+\mathrm{R} +\rho-n \rc e^{-\rho}\,\mathrm{d}V_g \rc,
\end{align} 
where the infimum is taken for all smooth functions $\rho$ such that
\begin{align*}
(4 \pi)^{-\frac{n}{2}} \int_N e^{-\rho} \,\mathrm{d}V_g=1.
\end{align*} 
Every minimizer $\rho$ of \eqref{EX201} satisfies the Euler--Lagrange equation:
\begin{align} \label{EX201a}
2 \Delta \rho - |\nabla \rho |^2 + \mathrm{R}+\rho- n = \boldsymbol{\mu}(g,1). 
\end{align}

Let $(N^n,g,f)$ be a compact Ricci shrinker with normalization \eqref{E101}. Then $g$ is a critical point of $\boldsymbol{\mu}(\cdot,1)$. We write $C^{\infty}(S^2(N))$ for the space of smooth symmetric $2$-tensors on $N$. Fix $\alpha\in(0,1)$. We use the usual H\"older spaces $C^{k,\alpha}$, with norms taken relative to a fixed smooth background metric.

There is a neighborhood $\mathcal U$ of $g$ in $C^{2,\alpha}(S^2(N))$ on which the minimizer $\rho_{\tilde g}$ of $\boldsymbol{\mu}(\tilde g,1)$ is unique and depends analytically on $\tilde g$. This is the fixed-scale version of the analytic construction in \cite[Lemma 2.2]{SW15} and \cite[Lemma 4.1 and Remark 4.2]{Kr14}. We define
\begin{align} \label{EX201b}
f_{\tilde g}:=\rho_{\tilde g}-\boldsymbol{\mu}(\tilde g,1).
\end{align}
In particular, it follows from \cite[Lemma 3.4]{SW15} (see also \cite[Proposition 3]{LW20}) that $f=f_{g}$. It is immediate from \eqref{EX201a} that
\begin{align} \label{EX201c}
2 \Delta f_{\tilde g} - |\nabla f_{\tilde g} |^2 + \mathrm{R}(\tilde g)+f_{\tilde g}- n =0,
\end{align}
where the underlying metric is $\tilde g$.

We define the following Ricci shrinker operator on the neighborhood $\mathcal{U}$ of $g$.
\begin{defn} \label{def:op1}
For any $\tilde g \in \mathcal U$, we define
\begin{align}\label{EX201cc}
\Phi(\tilde g) := \frac{1}{2} \tilde g - \Rc(\tilde g) - \nabla_{\tilde g}^2 f_{\tilde g},
\end{align}
where $f_{\tilde g}$ is given by \eqref{EX201b}. In particular, $\Phi(\tilde g) = 0$ if and only if $(N,\tilde g,f_{\tilde g})$ is a Ricci shrinker with normalization \eqref{E101}.
\end{defn}

Before we continue the discussion, we recall some definitions and notations for later use. 

\begin{defn} \label{def:not}
On $(N^{n}, g, f)$, we have the following definitions. 
\begin{enumerate}
\item The $L^2$-norm for $(r,s)$-tensor fields, denoted by $T^{r,s}N$, is defined with respect to the volume form $\mathrm{d}V_f: =e^{-f} \mathrm{d}V_{g}$, i.e., for any $T \in L^2(T^{r,s}N)$, $\|T\|_{L^2} = (\int |T|^2 \,\mathrm{d}V_f)^{\frac 1 2}$. The Sobolev spaces $H^{k}(T^{r,s}N)$ are defined with respect to $\mathrm{d}V_f$ as well.

\item The weighted divergence $\df: C^{\infty}(T^{r,s}N) \rightarrow C^{\infty}(T^{r,s-1}N)$ is defined as 
\begin{align*}
(\df T)^{i_1,\cdots,i_r}_{j_1,\dots,j_{s-1}}:=\na_i T^{i_1,\cdots,i_r}_{i,j_1,\dots,j_{s-1}}-T^{i_1,\cdots,i_r}_{i,j_1,\dots,j_{s-1}} f_i
\end{align*}
for any $T \in C^{\infty}(T^{r,s}N)$. We denote by $\df^*$ the adjoint operator of $\df$ with respect to the $L^2$-norm. When acting on $1$-forms, $\df^*$ denotes the adjoint of the restriction $\df:C^\infty(S^2(N))\to C^\infty(T^*N)$, so
\[
 (\df^*w)_{ij}=-\frac12(\na_iw_j+\na_jw_i).
\] 

\item The weighted Laplacian operator $\Delta_f$ is defined as 
$\Delta_f=\Delta-\la \na \cdot, \na f \ra$ on $C^{\infty}(T^{r,s}N)$. We set $\LL$ to be the stability operator $\Delta_f + 2\Rm $ on $C^{\infty}(S^2(N))$, that is, for any $h \in C^{\infty}(S^2(N))$,
\begin{align*}
\LL(h)_{ij}: =\Delta_f h_{ij} +2 R_{ikjl} h_{kl}.
\end{align*}
\item The second-order elliptic operator $\mathcal{P}$ is defined as
\begin{align*}
\mathcal{P} w = \df \df^* w
\end{align*}
for any $1$-form $w \in C^{\infty}(T^* N)$. It is clear by the definition that the kernel of $\mathcal P$, denoted by $\KK_{\PP}$, consists of all Killing forms (i.e., the dual forms of Killing fields) on $N$.
\end{enumerate}
\end{defn}

Next, we prove the following orthogonal decomposition; see also \cite[Proposition 2.2]{PS15}.

\begin{prop} \label{prop:decom1}
For any $k \ge 1$,
\begin{align} \label{EX202a}
C^{k,\alpha}(S^2(N))=\ker \lc \df\vert_{C^{k,\alpha}(S^2(N))} \rc \oplus \df^*\lc C^{k+1,\alpha}(T^* N)\rc.
\end{align}
In particular, we have
\begin{align} \label{EX202aa}
C^{\infty}(S^2(N))=\ker(\df)\oplus \im(\df^*),
\end{align}
where $\ker(\df):=\ker \lc \df\vert_{C^{\infty}(S^2(N))} \rc$ and $\im(\df^*):=\df^*\lc C^{\infty}(T^* N)\rc$. Furthermore, the decompositions are orthogonal with respect to $L^2$.
\end{prop}
\begin{proof}
For any $h \in C^{k,\alpha}(S^2(N))$, we solve the following equation for $w \in C^{k+1,\alpha}(T^* N)$
\begin{align} \label{EX202ab}
\mathcal{P} w = \df h.
\end{align}
By standard elliptic theory (cf. \cite[Theorem 27]{Be87}), such a $w$ exists because $\df h$ is orthogonal to all Killing forms. Therefore, it follows from \eqref{EX202ab} that $h-\df^*w \in \ker \lc \df\vert_{C^{k,\alpha}(S^2(N))} \rc$. Clearly, the decompositions \eqref{EX202a} and \eqref{EX202aa} are orthogonal, and the proof is complete.
\end{proof}

The following quantitative slice construction will be used to fix the gauge. All norms and orthogonal complements in its proof are taken with respect to the fixed pair $(g,f)$.

\begin{thm}\label{TH202}
For any $k\ge1$, there exist positive constants $\epsilon$ and $C$ such that, if $\|h\|_{C^{k,\alpha}}<\epsilon$, there is a $C^{k+1,\alpha}$ self-diffeomorphism $\varphi$ of $N$ satisfying
\begin{align*}
\delta_f\bigl(\varphi^*(g+h)-g\bigr)=0,
\qquad
\|\varphi-\mathrm{Id}\|_{C^{k+1,\alpha}}
\le C\|\delta_f h\|_{C^{k-1,\alpha}}.
\end{align*}
\end{thm}
\begin{proof}
It is convenient to solve for the inverse diffeomorphism. For a small $1$-form $w\in C^{k+1,\alpha}(T^*N)$ perpendicular to $\mathcal K_{\mathcal P}$, let
\[
\psi_w(x)=\exp_x^g(w^\sharp(x)),\qquad
g_w=\psi_w^*g,\qquad f_w=f\circ\psi_w.
\]
Let $\Pi_0$ be the fixed $L^2$-projection onto $\mathcal K_{\mathcal P}^{\perp}$, and consider
\[
\mathcal G(w,h)=\Pi_0\delta_{g_w,f_w}(g+h-g_w).
\]
This is a smooth map from a neighborhood of $(0,0)$ in
\[
\bigl(C^{k+1,\alpha}(T^*N)\cap\mathcal K_{\mathcal P}^{\perp}\bigr)
\times C^{k,\alpha}(S^2(N))
\quad\hbox{to}\quad
C^{k-1,\alpha}(T^*N)\cap\mathcal K_{\mathcal P}^{\perp}.
\]
At $(0,0)$,
\[
\mathrm{D}_w\mathcal G(0,0)w
=-\delta_f(L_{w^\sharp}g)=2\mathcal P w,
\]
which is an isomorphism on the indicated complements. The implicit function theorem gives a small solution $w=w(h)$. Since $\mathcal G(0,h)=\delta_fh$, the uniform inverse estimate and absorption of the nonlinear remainder give
\begin{align}\label{EX203ba}
\|w(h)\|_{C^{k+1,\alpha}}
\le C\|\delta_fh\|_{C^{k-1,\alpha}}.
\end{align}
To remove the projection $\Pi_0$, put
\[
\beta=\delta_{g_w,f_w}(g+h-g_w).
\]
The equation $\mathcal G(w,h)=0$ says that $\beta$ belongs to the fixed finite-dimensional space $\mathcal K_{\mathcal P}$. On the other hand, $\beta$ is orthogonal, in $L^2(e^{-f_w}\mathrm{d}V_{g_w})$, to every Killing form of $g_w$, including $\psi_w^*\kappa$ for $\kappa\in\mathcal K_{\mathcal P}$. The resulting pairing matrix between a fixed basis of $\mathcal K_{\mathcal P}$ and its pullbacks is invertible when $w$ is small, since at $w=0$ it is the positive definite Gram matrix. Hence $\beta=0$. Naturality of the weighted divergence, with $\varphi=\psi_w^{-1}$, now gives
\begin{align}\label{EX203b}
\delta_f\bigl(\varphi^*(g+h)-g\bigr)=0.
\end{align}
Finally, inversion of a diffeomorphism sufficiently close to the identity preserves the stated norm bound, so \eqref{EX203ba} proves the theorem.
\end{proof}

Next, we recall the following definition.
\begin{defn} \label{def:ISD}
The infinitesimal solitonic deformation space with respect to $(N^{n},g, f)$ is defined as
\begin{align*}
\mathrm{ISD}:=\{h \in C^{\infty}(S^2(N)) \mid h \in \ker(\Phi') \cap \ker(\df) \},
\end{align*}
where $\Phi'$ is the first variation of $\Phi$ in \eqref{EX201cc} at $g$.
\end{defn}

Notice that $\Phi'$ can be explicitly calculated by the following lemma, whose proof is similar to \cite[Lemma 2.5]{LZ22}; see also \cite[Theorem 1.1]{CZ12}.

\begin{lem} \label{lem:ephi}
For any $h \in C^{2,\alpha}(S^2(N))$, we have
\begin{align*}
\Phi'(h) =\frac{1}{2}\LL(h) + \df^{*} \df h +\nabla^2 (\tr_{g}(h)/2- f'),
\end{align*}
where $f'$ is determined by
\begin{align} \label{E217}
(2 \Delta_f+ 1)(\tr_{g}(h)/2 - f') = \df^2 h.
\end{align}
In particular, if $h \in \ker(\df)$, then $f' = \tr_{g}(h)/2$ and we have
\begin{align*}
\Phi'(h) =\frac{1}{2} \LL(h).
\end{align*}
Moreover, $\Phi'$ is a self-adjoint operator on $C^{2,\alpha}(S^2(N))$ and preserves the decomposition \eqref{EX202a}.
\end{lem}

Lemma \ref{lem:ephi} and Definition \ref{def:ISD} imply that ISD is finite-dimensional.

\begin{lem} \label{lem:kphi}
We have the following orthogonal decomposition
\begin{align*}
\ker(\Phi')=\mathrm{ISD} \oplus \mathrm{Im}(\df^{*}).
\end{align*}
\end{lem}
\begin{proof}
Since $\Phi'$ preserves the decomposition $\ker(\df) \oplus \mathrm{Im}(\df^{*})$, we only need to prove that $\mathrm{Im}(\df^{*}) \subset \ker(\Phi')$, which follows from the fact that $\Phi(\varphi^* g)=0$ for any self-diffeomorphism $\varphi$ of $N$. 
\end{proof}

We next prove that $\LL$ preserves the decomposition \eqref{EX202aa}, thanks to the following result, which can be checked by direct computation (cf. \cite[Theorem 2.6]{CM21b}).

\begin{lem}\label{lem:pre}
For any $h \in C^{\infty}(S^2(N))$ and $w \in C^{\infty}(T^*N)$, we have
\begin{align*}
\df (\LL h)=(\Delta_f+1/2) \df h \quad \text{and} \quad\LL(\df^* w)=\df^* ( \Delta_f w+w/2 ).
\end{align*}
\end{lem}

Since $\LL$ preserves the decomposition \eqref{EX202aa}, one can consider the spectrum of $\LL$ restricted on $\ker(\df)$. Notice that by direct computation, $\df(\Rc)=0$ and $\LL(\Rc)-\Rc=0$ (cf. \cite[Lemma 2.1]{PW10} and \cite[Lemma 3.2]{CZ12}). For $A=\Delta,\Delta_f,\LL$, we use the eigenvalue convention $Au+\lambda u=0$. Therefore, $-1 \in \mathrm{spec}(\LL \vert_{\ker(\df)})$ for any compact Ricci shrinker. Next, we recall the following definition of a linearly stable Ricci shrinker (cf. \cite{CHI04,CZ12}); see also \cite{CM12}.

\begin{defn} \label{def:stable1}
A compact Ricci shrinker $(N^n,g, f)$ is called \textbf{linearly stable} if $-1$ is the only negative eigenvalue of $\LL$ on $\ker(\df)$, and the corresponding multiplicity is one. $(N^n,g, f)$ is called \textbf{strictly linearly stable} if furthermore $\isd=0$.
\end{defn}

A compact Ricci shrinker is linearly stable if and only if the second derivative of Perelman's $\nnu$-functional is nonpositive (cf. \cite{CHI04,CZ12}). For instance, the standard $S^n$ and $\CP^n$ are linearly stable Ricci shrinkers (cf. \cite{CHI04}). The standard $\CP^n$, $n\ge2$, is linearly stable but is not a local maximizer of $\nnu$ (cf. \cite{Kr20}). Unstable examples include all compact \ka-Ricci shrinkers with Hodge number $h^{1,1}>1$ (cf. \cite{CHI04,HaMu11}). In particular, all compact \ka-Ricci shrinker surfaces are linearly unstable except for $\CP^2$. In addition, the linear stability of every simply-connected irreducible compact symmetric space is known; see \cite[Table 1, Table 2]{CH15}. If $(N,g,f)$ is strictly linearly stable, then every volume-normalized Ricci flow starting sufficiently close to $g$ exists for all positive time and converges, modulo diffeomorphisms, to a constant multiple of $g$ \cite[Theorems 1.2 and 1.4]{Kr14}.

A linearly stable compact Ricci shrinker cannot split globally as a product of two positive-dimensional compact Ricci shrinkers, as the next lemma shows.

\begin{lem} \label{lem:prounst}
Let $(N,g,f)=(N_1 \times N_2,g_1\times g_2,f_1+f_2)$ be the product of two compact Ricci shrinkers. Then $(N,g,f)$ is linearly unstable.
\end{lem}

\begin{proof}
Since $\Rc(g_1)$ and $\Rc(g_2)$ are two linearly independent eigentensors of $\LL$ corresponding to $-1$, $(N,g,f)$ is linearly unstable by definition.
\end{proof}

In the special case of Einstein manifolds, we have the following definition of the $H$-stability (cf. \cite{Ko80}).

\begin{defn}\label{def:Hstable}
A compact Einstein manifold $(N^n,g)$ is called $H$-stable if the eigenvalues of $\LL$ on the TT (traceless-transverse) space are nonnegative.
\end{defn}

For a compact Einstein manifold $(N^n, g)$, $\LL\vert_{TT} \le 0$ if and only if the second derivative of the renormalized Einstein-Hilbert functional is nonpositive for all volume-preserving deformations modulo the conformal variations (cf. \cite{Ko80}). The following characterization of the $H$-stability is well-known; see also \cite[Lemma 3.5]{CH15}.

\begin{lem}
A compact Einstein manifold $(N^n, g)$ with $\Rc(g)=g/2$ is linearly stable if and only if either it is $H$-stable such that all nonzero eigenvalues of $\Delta$ on functions are no smaller than $1$, or isometric to the standard $S^n$.
\end{lem} 

\begin{proof}
We first assume $(N^n, g)$ is linearly stable. Then it is clear by definition that the $\LL\vert_{TT} \le 0$. If $\Delta u+\mu u=0$ for some constant $\mu \in (0,1)$, it follows from direct computation that $h:=(\mu-1/2) u g+\na^2 u$ satisfies
\begin{align*}
\delta h=0 \quad \text{and} \quad \LL h+(\mu-1)h=0.
\end{align*} 
Therefore, it follows from Definition \ref{def:stable1} that $h=0$ and hence $(\mu-1/2) u g+\na^2 u=0$. By taking the trace, we obtain
\begin{align*}
\Delta u+n(\mu-1/2)u=0.
\end{align*} 
Consequently, either $u=0$ and in this case $\mu$ is not an eigenvalue, or $\mu=\frac{n}{2(n-1)}$ and hence by Obata's theorem \cite{Oba62}, $(N,g)$ is isometric to the standard $S^n$.

Conversely, if $(N,g)$ is isometric to the standard $S^n$, it is obviously linearly stable. Now, we assume 
\begin{align} \label{EZ201}
\delta h=0 \quad \text{and} \quad \LL h+(\mu-1) h=0.
\end{align} 
By taking the trace in \eqref{EZ201}, we obtain $\Delta H+\mu H=0$, where $H:=\tr_{g}(h)$. If $H\equiv0$, then $h$ is TT, and $H$-stability gives $\mu-1\ge0$. Otherwise, $\mu$ is an eigenvalue of the scalar Laplacian and hence $\mu\ge0$. If $\mu=0$, then $H$ is a constant and hence $h':=h-H g/n$ satisfies
\begin{align*}
\delta h'=\tr_{g}(h')=0 \quad \text{and} \quad \LL h'-h'=0.
\end{align*} 
By $H$-stability, $h'=0$, and hence $h$ is a multiple of $g$. If $H\not\equiv0$ and $\mu>0$, then $\mu\ge1$ by the assumed scalar spectral bound, so the corresponding eigenvalue $\mu-1$ of $h$ is nonnegative. Thus $-1$ is the only negative eigenvalue, with eigenspace spanned by $g$.

In sum, the proof is complete.
\end{proof}

For the rest of this section, fix a compact Ricci shrinker $(N^n,\bar g,\bar f)$. Unless otherwise stated, all norms, inner products and projections refer to $(\bar g,\bar f)$. The slice condition for a curve $g(t)$ through $\bar g$ is
\[
\delta_{\bar f}(g(t)-\bar g)=0.
\]
If $g(t)$ is a smooth curve of Ricci shrinker metrics in this slice and $h^i=\left.\frac{\mathrm{d}^i}{\mathrm{d}t^i}\right|_{t=0}g(t)$, differentiating $\Phi(g(t))=0$ gives
\begin{align}\label{E201}
\Phi'(h^k)+
\sum_{l=2}^k\ \sum_{\substack{1\le k_1\le\cdots\le k_l\\ k_1+\cdots+k_l=k}}
C(k,l,k_1,\ldots,k_l)\Phi^{(l)}(h^{k_1},\ldots,h^{k_l})=0.
\end{align}
Here the coefficients are the usual constants in the chain rule, and $h^1\in\isd$. This motivates the following notation and the usual notion of finite-order integrability.

\begin{defn}
We define $M_1=0$ and for $k \geq 2$,
\begin{align}
\label{EX204}
M_{k}(h_1 , \cdots, h_{k-1}) : = \sum_{l = 2}^{k} \sum_{1 \leq k_1 \leq \cdots \leq k_l, k_1 + \cdots + k_l = k} C(k,l,k_1,\cdots,k_l) \Phi^{(l)} (h_{k_1} , \cdots, h_{k_l}),
\end{align}
for any $h_1,\cdots,h_{k-1} \in C^{\infty}(S^{2}(N))$.
\end{defn}
For instance, $M_2(h_1)=\Phi^{(2)}(h_1,h_1)$ and $M_3(h_1,h_2)=3\Phi^{(2)}(h_1,h_2)+\Phi^{(3)}(h_1,h_1,h_1)$. With the above notation, \eqref{E201} can be written as 
\begin{align*}
\Phi'(h^{k}) + M_k (h^1,\cdots,h^{k-1}) = 0.
\end{align*}

\begin{defn}\label{D201}
For $k\ge2$ and a nonzero $h_1\in\isd$, we say that $h_1$ is \textbf{integrable up to order $k$} if there are smooth symmetric $2$-tensors $h_2,\ldots,h_k$ such that
\begin{align}\label{EX205a}
\Phi'(h_l)+M_l(h_1,\ldots,h_{l-1})=0,
\qquad 2\le l\le k.
\end{align}
\end{defn}

\begin{lem} \label{lem:perp1}
Suppose $h_1 \in \mathrm{ISD}$ and $(h_2,\cdots,h_{k})$ solves \eqref{EX205a}. Then
\begin{align*}
M_{k+1}(h_1,\cdots,h_{k}) \perp \mathrm{Im}(\dbf^*).
\end{align*}
\end{lem}
\begin{proof}
For any $h =L_{X} \bar g \in \mathrm{Im}(\dbf^*)$, we assume $\varphi_t$ is a family of diffeomorphisms generated by $X$. We define
\begin{align*}
g=g(s,t) := \varphi_t^{*} \lc \bar g+ \sum_{i=1}^{k} \frac{s^i}{i!} h_i \rc \quad \text{and} \quad \mu (s,t) = \boldsymbol{\mu}(g(s,t),1).
\end{align*} 
Notice that by the invariance of $\boldsymbol{\mu}$ up to a diffeomorphism, $ \mu (s,t)$ depends only on $s$. The first variational formula of $\boldsymbol{\mu}$ \cite{Pe1} yields
\begin{align*}
\partial_t \mu(s,t) = \frac{1}{(4 \pi)^{\frac{n}{2}}} \int_{N} \langle \Phi(g),\partial_t g \rangle e^{-\rho_g} \,\mathrm{d}V_g,
\end{align*}
where $\rho_g$ is the minimizer of $\boldsymbol{\mu}(g,1)$. Therefore, it follows from our assumption and definitions that
\begin{align} \label{EX206a}
0=\left. \partial_s^{k+1} \partial_t \mu(s,t) \right\vert_{s=t=0}=\frac{1}{(4 \pi)^{\frac{n}{2}}} \int_{N} \langle M_{k+1}(h_1,\cdots,h_{k}) ,h \rangle e^{-\rho_{\bar g}} \,\mathrm{d}V_{\bar g}.
\end{align}
Here, we have used the fact that $\left. \partial_s^{k+1} g(s,t) \right\vert_{s=t=0}=0$. From \eqref{EX206a}, we conclude 
\begin{align*}
\int_{N} \langle M_{k+1}(h_1,\cdots,h_{k}), h \rangle \,\vf=0
\end{align*}
and hence the proof is complete.
\end{proof}

\begin{defn} \label{def:ind}
Given $h_1 \in \isd$, we say the sequence of symmetric $2$-tensors $\{h_i\}$ is \textbf{induced} by $h_1$ if $h_i \in (\ker{\Phi'})^{\perp}$ for $i\ge2$ and for any $l \ge 2$,
\begin{align}\label{E203}
\Phi'(h_{l}) + \pi_{(\ker{\Phi'})^{\perp}}(M_l(h_{1},\cdots,h_{l-1})) = 0,
\end{align}
where $\pi_{(\ker{\Phi'})^{\perp}}$ is the projection onto $(\ker{\Phi'})^{\perp}$ in $L^2$.
\end{defn}

It follows from Lemma \ref{lem:kphi} that $(\ker{\Phi'})^{\perp}=\ker(\dbf) \cap \isd^{\perp}$. The restriction of $\Phi'$ to this space is invertible, so \eqref{E203} determines a unique induced sequence. Since each $M_l$ is a sum of multilinear expressions in the preceding terms, induction shows that $h_l$ is a homogeneous polynomial of degree $l$ in $h_1$. In particular, $\{\lambda^i h_i\}$ is the sequence induced by $\lambda h_1$ for every constant $\lambda$.

For $k \ge 2$, we set $I_k$ to be the linear space spanned by all possible $h_k$ as $h_1$ varies in $\isd$. Next, we prove
\begin{lem}
For any $k \ge 2$, $I_k$ is finite-dimensional.

\end{lem}
\begin{proof}
The coefficients of the homogeneous polynomial $h_1\mapsto h_k$ on the finite-dimensional space $\isd$ span a finite-dimensional space containing $I_k$.
\end{proof}
In particular, all norms on $I_i$ are equivalent. Next, we prove
\begin{prop}
\label{P01}
Suppose $h_1 \in \isd$ and $\{h_i\}$ is the sequence induced by $h_1$. Then for every fixed $i \ge 1$ there is a constant $C_i$ such that
\begin{align} \label{EX207a}
\|h_i\|_{C^{2,\alpha}} \leq C_i \|h_1\|^i_{C^{\alpha}}.
\end{align}
\end{prop}
\begin{proof}
For each fixed $i$, the map $h_1\mapsto h_i$ is a continuous homogeneous polynomial of degree $i$ on the finite-dimensional space $\isd$. Its $C^{2,\alpha}$-norm is bounded on the compact sphere $\|h_1\|_{C^\alpha}=1$. Homogeneity gives \eqref{EX207a}.
\end{proof}

\begin{defn} \label{def:kob}
We say $(N^{n},\bar g, \bar f)$ has an \textbf{obstruction of order $k$} if there exists $\epsilon>0$ such that for any $h_1\in \isd$ with $\|h_1\|_{L^2} < \epsilon$, 
\begin{align} \label{E205}
\| \pi_{\isd}( \sum _{l=2}^{k} \frac{1}{l!} M_l (h_{1},\cdots,h_{l-1}))\|_{L^2} \geq \ep \| h_1 \|_{L^2}^{k},
\end{align}
where $\{h_i\}$ is the sequence induced by $h_1$. $(N^{n},\bar g, \bar f)$ has an obstruction of order $1$ if $\isd=0$.
\end{defn}

For instance, $(N^{n},\bar g, \bar f)$ has an obstruction of order $3$ if for all $h_1 \in \isd$ small enough,
\begin{align} \label{E205a}
\| \pi_{\isd} (\frac{1}{6}\Phi^{(3)} (h_1 ,h_1 ,h_1 ) + \frac{1}{2} \Phi^{(2)} (h_1 ,h_2) + \frac{1}{2} \Phi^{(2)} (h_1 ,h_1 )) \|_{L^2} \geq \ep \| h_1 \|_{L^2}^3.
\end{align}

\begin{lem}\label{lem:kob_monotone}
Suppose $(N^{n},\bar g, \bar f)$ has an obstruction of order $k_0$, then $(N^{n},\bar g, \bar f)$ has an obstruction of order $k$ for any $k>k_0$.
\end{lem}

\begin{proof}
By our assumption, there exists $\ep>0$ such that for any $h_1\in \isd$ with $\|h_1\|_{L^2} < \epsilon$, 
\begin{align} \label{EX208a}
\| \pi_{\isd}( \sum _{l=2}^{k_0} \frac{1}{l!} M_l (h_{1},\cdots,h_{l-1}))\|_{L^2} \geq \ep \| h_1 \|_{L^2}^{k_0},
\end{align}
where $\{h_i\}$ is the sequence induced by $h_1$. From \eqref{EX207a} and the definition of $M_l$, it is clear that
\begin{align}\label{EX208b}
\| \pi_{\isd}( \sum _{l=k_0+1}^{k} \frac{1}{l!} M_l (h_{1},\cdots,h_{l-1}))\|_{L^2} \le \| \sum _{l=k_0+1}^{k} \frac{1}{l!} M_l (h_{1},\cdots,h_{l-1})\|_{L^2} \le C \|h_1\|^{k_0+1}_{L^2}.
\end{align}
Combining \eqref{EX208a} and \eqref{EX208b}, we obtain
\begin{align*}
\| \pi_{\isd}( \sum _{l=2}^{k} \frac{1}{l!} M_l (h_{1},\cdots,h_{l-1}))\|_{L^2} \geq \ep \| h_1 \|_{L^2}^{k_0}-C \|h_1\|^{k_0+1}_{L^2} \ge 
\frac{\ep}{2} \| h_1 \|_{L^2}^{k_0} \ge \frac{\ep}{2} \| h_1 \|_{L^2}^{k},
\end{align*}
if $ \| h_1 \|_{L^2}$ is sufficiently small. In particular, it implies that $(N^{n},\bar g, \bar f)$ has an obstruction of order $k$.
\end{proof}

\begin{defn} \label{def:hol}
We say $(N^{n},\bar g, \bar f)$ satisfies the \textbf{rigidity inequality of order $k$} if there exist positive constants $\ep$ and $C$ satisfying the following property.

For any $g \in C^{2,\alpha}(S^2(N))$ with $\|g- \bar{g} \|_{C^{2,\alpha}} < \epsilon$, there exists a $C^{3,\alpha}$ self-diffeomorphism $\varphi$ of $N$ such that $\|\varphi- \mathrm{Id} \|_{C^{3,\alpha}} \le C \|\dbf(g- \bar{g}) \|_{C^{1,\alpha}}$ and
\begin{align}\label{EX210a}
\|\varphi^{*} g - \bar{g} \|^{k}_{C^{2,\alpha}} \leq C \| \Phi(g) \|_{C^{\alpha}}.
\end{align}
\end{defn}

\begin{prop}\label{P202}
Let $N\ge2$. The Fubini--Study metric on $\CP^{N-1}$ with $\Rc(g_{FS})=g_{FS}/2$ has an obstruction of order $3$ in the sense of Definition \ref{def:kob}. More precisely, it has an obstruction of order $1$ when $N=2$, and an obstruction of order $2$ when $N$ is odd.
\end{prop}
\begin{proof}
Write $g=g_{FS}$ and $\bar f=N-1$. The case $N=2$ follows from $\isd=0$ on $S^2$ and Lemma \ref{lem:kob_monotone}. Henceforth, assume $N\ge3$, and put $n=2(N-1)$.

For $h\in\isd$, let $h_2(h)$ be the second tensor induced by $h$ in Definition \ref{def:ind}, and define
\begin{align*}
Q_2(h)&:=\frac12\pi_{\isd}\bigl(\Phi^{(2)}(h,h)\bigr),\\
Q_3(h)&:=\frac16\pi_{\isd}\bigl(\Phi^{(3)}(h,h,h)+3\Phi^{(2)}(h,h_2(h))\bigr).
\end{align*}
These are homogeneous polynomial maps of degrees $2$ and $3$, respectively. We will prove that there exist $c,\epsilon>0$ such that
\begin{align}\label{ECP:obstruction}
\|Q_2(h)+Q_3(h)\|_{L^2}\ge c\|h\|_{L^2}^3
\end{align}
whenever $h\in\isd$ and $\|h\|_{L^2}<\epsilon$.

Let $V$ be the real vector space of traceless Hermitian $N\times N$ matrices. For $A\in V$, set
\begin{align*}
u_A([z])=\frac{z^*Az}{|z|^2},\qquad H(A)=u_Ag+2\nabla^2u_A.
\end{align*}
The description of infinitesimal solitonic deformations of $\CP^{N-1}$ gives a linear isomorphism $H:V\to\isd$; see \cite[Section 2]{LZ22}. In particular, $\Delta u_A=-u_A$. Integration by parts and $\Rc=g/2$ give
\begin{align}\label{ECP:inner}
\langle H(A),H(B)\rangle_{L^2}
=(n-2)e^{-(N-1)}\int_{\CP^{N-1}}u_Au_B\,\mathrm{d}V_g
=\kappa\operatorname{tr}(AB),
\end{align}
where $\kappa>0$.
We claim that
\begin{align}\label{ECP:quadratic}
Q_2(H(A))=\frac1{N+2}H\left(A^2-\frac{\operatorname{tr}(A^2)}N\mathrm{Id}\right).
\end{align}

To verify this, first note that $\Phi'(u_Ag)=0$. Diffeomorphism invariance of the entropy, differentiated twice in the direction $u_Ag$, shows that $\Phi^{(2)}(u_Ag,u_Ag)$ is orthogonal to $\im(\dbf^*)$. If $\varphi_t$ is the flow of $\nabla u_A$, then the first derivative of $\varphi_t^*((1+tu_A)g)$ at zero is $H(A)$. Diffeomorphism covariance of $\Phi$ therefore implies
\begin{align*}
\pi_{\isd}\Phi^{(2)}(H(A),H(A))
=\pi_{\isd}\Phi^{(2)}(u_Ag,u_Ag).
\end{align*}
The conformal second-variation formula in \cite[Lemma A.3]{LZ22}, integrated against $u_B$, yields
\begin{align*}
\int_{\CP^{N-1}}\langle\Phi^{(2)}(u_Ag,u_Ag),u_Bg\rangle\,\mathrm{d}V_g
=(n-2)\int_{\CP^{N-1}}u_A^2u_B\,\mathrm{d}V_g.
\end{align*}
Here one uses $\Delta u_B=-u_B$ and
$\int u_B|\nabla u_A|^2\,\mathrm{d}V_g=\frac12\int u_A^2u_B\,\mathrm{d}V_g$.
Consequently,
\begin{align*}
\langle Q_2(H(A)),H(B)\rangle_{L^2}
=\frac{n-2}{2}e^{-(N-1)}\int_{\CP^{N-1}}u_A^2u_B\,\mathrm{d}V_g.
\end{align*}
The standard sphere moment formulas, applied to the homogeneous coordinates of $\CP^{N-1}$, give
\begin{align*}
\frac1{\operatorname{Vol}(\CP^{N-1},g)}\int_{\CP^{N-1}}u_Au_B\,\mathrm{d}V_g
&=\frac{\operatorname{tr}(AB)}{N(N+1)},\\
\frac1{\operatorname{Vol}(\CP^{N-1},g)}\int_{\CP^{N-1}}u_A^2u_B\,\mathrm{d}V_g
&=\frac{2\operatorname{tr}(A^2B)}{N(N+1)(N+2)}.
\end{align*}
Together with \eqref{ECP:inner}, these identities prove \eqref{ECP:quadratic}.

If $N$ is odd, a nonzero traceless Hermitian matrix cannot satisfy $A^2=a\mathrm{Id}$: its eigenvalues would be $\sqrt a$ and $-\sqrt a$ with equal multiplicities. Thus $Q_2$ has no nonzero zero. Compactness of the unit sphere in $\isd$ gives $\|Q_2(h)\|_{L^2}\ge c\|h\|_{L^2}^2$, proving the order-$2$ assertion. Lemma \ref{lem:kob_monotone} then gives an obstruction of order $3$.

It remains to consider even $N\ge4$. Equip $V$ with the norm $|A|=(\operatorname{tr}(A^2))^{1/2}$, and define
\begin{align*}
Z:=\{A\in V:|A|=1,\ A^2=N^{-1}\mathrm{Id}\}.
\end{align*}
This is the compact unitary orbit of $N^{-1/2}\operatorname{diag}(\mathrm{Id}_{N/2},-\mathrm{Id}_{N/2})$. Moreover, whenever $|A|=1$ and $|A^2-N^{-1}\mathrm{Id}|$ is sufficiently small,
\begin{align}\label{ECP:distance}
\operatorname{dist}(A,Z)\le C|A^2-N^{-1}\mathrm{Id}|.
\end{align}
Diagonalizing $A$ and replacing its eigenvalues by the nearest values $\pm N^{-1/2}$ proves \eqref{ECP:distance}; the trace-free condition forces equal multiplicities when the error is small.

We next record the cubic pairing on the zero set of $Q_2$. If $A\ne0$ and $A^2=\lambda^2\mathrm{Id}$, then
\begin{align}\label{ECP:cubic}
\langle Q_3(H(A)),H(A)\rangle_{L^2}
=c_N\lambda^4>0,
\end{align}
where $c_N>0$. This follows from \cite[Theorem 3.15]{LZ22}. Set $h=H(A)$ and $h_2=h_2(h)$. By \eqref{ECP:quadratic}, $Q_2(h)=0$. Lemma \ref{lem:perp1} and the defining equation for $h_2$ therefore give
\begin{align*}
\Phi'(h_2)+\Phi^{(2)}(h,h)=0.
\end{align*}
Thus the first two derivatives of $\Phi(g(t))$ vanish for $g(t)=g+th+t^2h_2/2$. Pull back $g(t)$ by the flow of $-\nabla u_A$. The resulting curve has first derivative $u_Ag$ and a second derivative $k$ satisfying
\begin{align*}
\Phi'(k)+\Phi^{(2)}(u_Ag,u_Ag)=0.
\end{align*}
Since the lower-order residuals vanish, covariance of $\Phi$ shows that its third derivative is unchanged by this pullback. The third derivative of the pulled-back metric contributes a term in $\im(\Phi')$, which is orthogonal to $u_Ag$. In addition, Lemma \ref{lem:perp1} shows that $M_3(h,h_2)$ is orthogonal to $\im(\dbf^*)$. Hence
\begin{align*}
\langle M_3(h,h_2),h\rangle_{L^2}
=e^{-(N-1)}\int_{\CP^{N-1}}
\langle\Phi^{(3)}(u_Ag,u_Ag,u_Ag)+3\Phi^{(2)}(u_Ag,k),u_Ag\rangle\,\mathrm{d}V_g.
\end{align*}
After a unitary change of coordinates, $u_A$ is the function with parameter $\lambda$ in \cite[Lemma 2.10]{LZ22}. Theorem 3.15 there shows that the last integral is a positive constant times $\lambda^4$ for every second-order solution $k$. Division by $6$ proves \eqref{ECP:cubic}.

Finally, suppose \eqref{ECP:obstruction} fails. By \eqref{ECP:inner} and homogeneity, there exist $t_j\downarrow0$ and $A_j\in V$ with $|A_j|=1$ such that
\begin{align}\label{ECP:contradiction}
\|t_j^{-1}Q_2(H(A_j))+Q_3(H(A_j))\|_{L^2}\longrightarrow0.
\end{align}
Since $Q_3$ is bounded on the unit sphere, $Q_2(H(A_j))=O(t_j)$. Equations \eqref{ECP:quadratic} and \eqref{ECP:distance} provide $B_j\in Z$ with $A_j=B_j+E_j$ and $|E_j|=O(t_j)$. Since $B_j^2=N^{-1}\mathrm{Id}$ and $\operatorname{tr}(E_j)=0$,
\begin{align*}
\operatorname{tr}(A_j^2B_j)
=\operatorname{tr}(B_j^3)+2\operatorname{tr}(B_j^2E_j)+\operatorname{tr}(E_j^2B_j)
=O(t_j^2).
\end{align*}
It follows from \eqref{ECP:inner} and \eqref{ECP:quadratic} that
\begin{align*}
\langle Q_2(H(A_j)),H(B_j)\rangle_{L^2}=O(t_j^2).
\end{align*}
Pass to a subsequence so that $B_j\to B_\infty\in Z$; then also $A_j\to B_\infty$. Pairing \eqref{ECP:contradiction} with $H(B_j)$ yields
\begin{align*}
0=\lim_{j\to\infty}\langle Q_3(H(A_j)),H(B_j)\rangle_{L^2}
=\langle Q_3(H(B_\infty)),H(B_\infty)\rangle_{L^2},
\end{align*}
contradicting \eqref{ECP:cubic}. This proves \eqref{ECP:obstruction} and the proposition.
\end{proof}

We next compare the obstruction polynomial with the full nonlinear equation. Put
\[
K=\isd,\qquad S=\ker\delta_{\bar f},\qquad X=S\cap K^\perp,
\qquad A=\Phi'_{\bar g}.
\]
The orthogonal projections $\Pi_K$ and $\Pi_X$ are bounded on the H\"older and Sobolev spaces used below. By Lemmas \ref{lem:ephi} and \ref{lem:pre}, the restriction
\[
A:X\cap C^{2,\alpha}\longrightarrow X\cap C^\alpha
\]
is an isomorphism, and it also satisfies
\begin{align}\label{eq:compact_inverse}
\|v\|_{H^{2}}\le C\|Av\|_{L^2},\qquad
\|v\|_{C^{2,\alpha}}\le C\|Av\|_{C^\alpha},\qquad v\in X\cap C^{2,\alpha}.
\end{align}

\begin{lem}\label{lem:compact_tame}
For $g'$ in a sufficiently small $C^{2,\alpha}$-neighborhood of $\bar g$ and every symmetric $2$-tensor $v\in C^{2,\alpha}$,
\begin{align}\label{eq:compact_tame}
\|(\mathrm{D}\Phi_{g'}-A)v\|_{L^2}
&\le C\|g'-\bar g\|_{C^{2,\alpha}}\|v\|_{H^{2}},\\
\|(\mathrm{D}\Phi_{g'}-A)v\|_{C^\alpha}
&\le C\|g'-\bar g\|_{C^{2,\alpha}}\|v\|_{C^{2,\alpha}}.\notag
\end{align}
Moreover, the following identity holds distributionally:
\begin{align}\label{eq:compact_bianchi}
\delta_{g',f_{g'}}\Phi(g')=0.
\end{align}
\end{lem}
\begin{proof}
Write $\mathscr E(g',f)=2\Delta_{g'}f-|\nabla f|_{g'}^2+\mathrm{R}(g')+f-n$. The function $f_{g'}$ satisfies $\mathscr E(g',f_{g'})=0$, and differentiation with respect to the metric in the direction $v$ gives
\[
(2\Delta_{g',f_{g'}}+1)\,\mathrm{D}f_{g'}[v]
=-\mathrm{D}_g\mathscr E(g',f_{g'})[v].
\]
Uniform invertibility of $2\Delta_{g',f_{g'}}+1$, elliptic estimates, and subtraction of the background equation give
\[
\|\mathrm{D}f_{g'}[v]\|_{H^{2}}\le C\|v\|_{H^{2}},\qquad
\|(\mathrm{D}f_{g'}-\mathrm{D}f_{\bar g})[v]\|_{H^{2}}
\le C\|g'-\bar g\|_{C^{2,\alpha}}\|v\|_{H^{2}}.
\]
The corresponding Schauder estimates are analogous. Differentiating
$\Phi(g')=g'/2-\mathrm{Ric}(g')-\nabla_{g'}^2f_{g'}$
now proves \eqref{eq:compact_tame}.

Finally, the contracted Bianchi identity and the commutation formula for the Hessian give
\[
\delta_{g',f_{g'}}\Phi(g')
=-\frac12\mathrm{d}\bigl(2\Delta_{g'}f_{g'}-|\nabla f_{g'}|_{g'}^2+\mathrm{R}(g')+f_{g'}-n\bigr)=0.
\]
\end{proof}

\begin{prop}[Kuranishi map]\label{prop:compact_reduction}
There is an analytic map $\psi$ from a neighborhood of $0$ in $K$ to $X\cap C^{2,\alpha}$, with $\psi(0)=\mathrm{D}\psi(0)=0$, such that
\begin{align}\label{eq:compact_graph}
\Pi_X\Phi(\bar g+x+\psi(x))=0.
\end{align}
Define
\begin{align}\label{eq:compact_reduced}
B(x)=\Pi_K\Phi(\bar g+x+\psi(x)).
\end{align}
Then $B$ is analytic near $0$ and, for every fixed $k\ge2$,
\begin{align}\label{eq:compact_reduced_taylor}
B(x)=Q_k(x)+O(\|x\|_{L^2}^{k+1}),\qquad
Q_k(x)=\Pi_K\sum_{l=2}^k\frac1{l!}M_l(x,h_2(x),\ldots,h_{l-1}(x)),
\end{align}
where $\{h_i(x)\}$ is the sequence induced by $x$. Moreover,
\begin{align}\label{eq:compact_graph_control}
\|\Phi(\bar g+x+\psi(x))\|_{C^\alpha}\le C\|B(x)\|_{L^2}.
\end{align}
If $h=x+y\in S\cap C^{2,\alpha}$ is sufficiently small, with $x=\Pi_Kh$ and $y\in X$, then
\begin{align}\label{eq:compact_transverse}
\|y-\psi(x)\|_{H^{2}}
&\le C\|\Phi(\bar g+h)\|_{L^2},\\
\|y-\psi(x)\|_{C^{2,\alpha}}
&\le C\|\Phi(\bar g+h)\|_{C^\alpha},\notag\\
\|B(x)\|_{L^2}
&\le C\|\Phi(\bar g+h)\|_{L^2}.\notag
\end{align}
In particular, the nearby zeros of $\Phi$ in the slice are exactly the metrics $\bar g+x+\psi(x)$ with $B(x)=0$.
\end{prop}
\begin{proof}
The derivative in $y$ of the analytic map
$(x,y)\mapsto\Pi_X\Phi(\bar g+x+y)$ at $(0,0)$ is $A|_X$. Thus \eqref{eq:compact_graph} follows from the analytic implicit function theorem. Differentiating this equation along $tx$, and using \eqref{E203}, gives inductively
\[
\left.\frac{\mathrm{d}^l}{\mathrm{d}t^l}\right|_{t=0}\psi(tx)=h_l(x),\qquad l\ge2.
\]
Projecting the corresponding derivatives of $\Phi$ onto $K$ proves \eqref{eq:compact_reduced_taylor}.

Put $g_x=\bar g+x+\psi(x)$. By \eqref{eq:compact_graph} and the fixed orthogonal decomposition, we can write
\[
\Phi(g_x)=B(x)+\delta_{\bar f}^*w,
\qquad w\perp\mathcal K_{\mathcal P}.
\]
The identity \eqref{eq:compact_bianchi} implies
\[
\delta_{g_x,f_{g_x}}\delta_{\bar f}^*w
=-\delta_{g_x,f_{g_x}}B(x).
\]
This is a uniformly elliptic equation whose operator is a small perturbation of $\mathcal P$. Since $B(x)\in K$, we have $\delta_{\bar f}B(x)=0$, and all norms on $K$ are equivalent. The inverse estimate for $\mathcal P$ on the complement of its kernel gives, with $\epsilon=\|g_x-\bar g\|_{C^{2,\alpha}}$,
\[
\|w\|_{C^{2,\alpha}}
\le C\epsilon\|w\|_{C^{2,\alpha}}+C\epsilon\|B(x)\|_{L^2}.
\]
Absorbing the first term proves \eqref{eq:compact_graph_control}. In particular, $B(x)=0$ implies $\Phi(g_x)=0$.

For the remaining assertions set $\eta=y-\psi(x)\in X$. By \eqref{eq:compact_graph},
\[
\Pi_X\Phi(\bar g+h)
=A\eta+\int_0^1\Pi_X(\mathrm{D}\Phi_{g_x+t\eta}-A)\eta\,\mathrm{d}t.
\]
Combining \eqref{eq:compact_inverse} and \eqref{eq:compact_tame}, and absorbing the small multiple of $\|\eta\|_{H^{2}}$ or $\|\eta\|_{C^{2,\alpha}}$, proves the first two estimates in \eqref{eq:compact_transverse}. Since $\Pi_K A\eta=0$, the same fundamental theorem of calculus gives
\[
\left\|B(x)-\Pi_K\Phi(\bar g+h)\right\|_{L^2}
\le C\epsilon\|\eta\|_{H^{2}}
\le C\epsilon\|\Phi(\bar g+h)\|_{L^2},
\]
which proves its last estimate. If $\Phi(\bar g+h)=0$, these estimates give $y=\psi(x)$ and $B(x)=0$, completing the proof.
\end{proof}

\begin{thm}\label{thm:hold}
$(N^n,\bar g,\bar f)$ has an obstruction of order $k$ if and only if it satisfies the rigidity inequality of order $k$.
\end{thm}
\begin{proof}
First let $k\ge2$. Suppose the obstruction inequality holds. By \eqref{eq:compact_reduced_taylor}, after decreasing the neighborhood if necessary,
\[
\|B(x)\|_{L^2}\ge\|Q_k(x)\|_{L^2}-C\|x\|_{L^2}^{k+1}
\ge c\|x\|_{L^2}^k.
\]
For a small $h=x+y\in S$, Proposition \ref{prop:compact_reduction} therefore gives
\[
\|x\|_{L^2}^k\le C\|\Phi(\bar g+h)\|_{L^2}.
\]
Since $\psi(x)=O(\|x\|_{L^2}^2)$ in $C^{2,\alpha}$, the estimates for $y-\psi(x)$ imply
\begin{align}\label{EX217a}
\|x\|_{L^2}^k+\|h\|_{H^{2}}^k
\le C\|\Phi(\bar g+h)\|_{L^2},
\end{align}
and
\[
\|h\|_{C^{2,\alpha}}^k\le C\|\Phi(\bar g+h)\|_{C^\alpha}.
\]
For a general nearby metric, apply Theorem \ref{TH202}. The covariance of $\Phi$ under diffeomorphisms and the uniform bounds for the resulting small diffeomorphism transfer the last estimate to \eqref{EX210a}.

Conversely, assume the rigidity inequality of order $k$. Apply it to $g_x=\bar g+x+\psi(x)$. This metric already satisfies the slice condition. The bound on the diffeomorphism in Definition \ref{def:hol} thus forces that diffeomorphism to be the identity. Since $x$ is the $K$-projection of $g_x-\bar g$, \eqref{eq:compact_graph_control} yields
\[
\|x\|_{L^2}^k
\le C\|g_x-\bar g\|_{C^{2,\alpha}}^k
\le C\|\Phi(g_x)\|_{C^\alpha}
\le C\|B(x)\|_{L^2}.
\]
Taylor expansion \eqref{eq:compact_reduced_taylor} and absorption of its order-$(k+1)$ remainder now give \eqref{E205}.

For $k=1$, if $K=0$, the inverse estimate and Lemma \ref{lem:compact_tame} give the rigidity inequality directly. If $K\ne0$, choose $0\ne x\in K$ and apply the assumed order-one rigidity inequality to $\bar g+tx$. It is in the slice and $\Phi(\bar g+tx)=O(t^2)$, contradicting the order-one estimate. This proves the assertion also for $k=1$.
\end{proof}

\begin{rem}\label{rem:compact_sobolev}
The proof also gives the following Sobolev form of the rigidity inequality. If $(N^n,\bar g,\bar f)$ has an obstruction of order $k$, then every sufficiently small $C^{2,\alpha}$ perturbation $g$ admits a diffeomorphism as in Theorem \ref{TH202} such that
\begin{align}\label{EX217d}
\|\varphi^*g-\bar g\|_{H^{2}}^k
\le C\|\Phi(g)\|_{L^2}.
\end{align}
This follows from \eqref{EX217a} and Theorem \ref{TH202}; for $k=1$, use the inverse estimate.
\end{rem}

\begin{cor}\label{cor:rigid1}
If $(N^n,\bar g,\bar f)$ is rigid, then it has an obstruction of some finite order.
\end{cor}
\begin{proof}
By the local slice theorem \cite{EB70,PS15}, rigidity makes $\bar g$ an isolated zero of $\Phi$ in the divergence-free slice. Proposition \ref{prop:compact_reduction} therefore makes $0$ an isolated zero of $B$.

If $K=0$, there is an obstruction of order one. Otherwise the finite-dimensional \L{}ojasiewicz distance inequality, applied to $\|B(x)\|_{L^2}^2$ near its isolated zero, gives constants $c>0$ and $\theta>0$ such that
\[
\|B(x)\|_{L^2}\ge c\|x\|_{L^2}^{\theta};
\]
see \cite[Theorem 6.4 and Remark 6.5]{BM88}. Choose an integer $k\ge\max\{2,\theta\}$. By \eqref{eq:compact_reduced_taylor},
\[
\|Q_k(x)\|_{L^2}
\ge c\|x\|_{L^2}^k-C\|x\|_{L^2}^{k+1}
\ge\frac c2\|x\|_{L^2}^k
\]
for sufficiently small $x$, proving the claim.
\end{proof}

\begin{cor}\label{cor:rigid2}
If $(N^n,\bar g,\bar f)$ has an obstruction of some finite order, then it is rigid.
\end{cor}
\begin{proof}
By \cite[Proposition 5.8 and Theorem 1.1]{LLW21}, convergence to the smooth compact shrinker in the pointed Gromov--Hausdorff topology improves to smooth convergence on the same underlying manifold; in dimensions two and three, this follows from the classification. If $g$ is nearby and $\Phi(g)=0$, Theorem \ref{thm:hold} gives a diffeomorphism $\varphi$ with $\varphi^*g=\bar g$.
\end{proof}

Theorem \ref{thm:hold} and Corollaries \ref{cor:rigid1} and \ref{cor:rigid2} prove Theorem \ref{thm:001}. Figures \ref{dia:1} and \ref{dia:2} summarize the equivalences.

\begin{center}
\begin{minipage}{\linewidth}\centering
\captionof{figure}{Two equivalent quantitative conditions}
\vspace*{5mm}
\smartdiagramset{uniform color list=gray!20 for 2 items, arrow style=<->, back arrow disabled, arrow color=gray, uniform arrow color=true, text width=5.8cm, module x sep=7cm}
\smartdiagram[flow diagram:horizontal]{Rigidity inequality of order $k$, Obstruction of order $k$}
\label{dia:1}
\end{minipage}
\end{center}

\begin{center}
\captionof{figure}{Characterization of rigidity}
\vspace*{5mm}
\smartdiagramset{uniform color list=gray!20 for 2 items, arrow style=<->, back arrow disabled, arrow color=gray, uniform arrow color=true, text width=5.8cm, module x sep=7cm}
\smartdiagram[flow diagram:horizontal]{Rigidity, Rigidity inequality of some order $k$}
\label{dia:2}
\end{center}

For later applications, we define another auxiliary sequence similar to Definition \ref{def:ind}.

For any curve $g_t$ starting from $\bar g$, we set $h^{i}: =\left.\frac{\mathrm{d}^i}{\mathrm{d}t^i}\right|_{t=0} g_t$, $f_t:=f_{g_t}$ from \eqref{EX201b} and $f^{i}: =\left.\frac{\mathrm{d}^i}{\mathrm{d}t^i}\right|_{t=0} f_t$. Then it follows from \eqref{EX201c} that
\begin{align} \label{EX218a}
2 \Delta f_t - |\nabla f_t |^2 + \mathrm{R}+f_t- n =0,
\end{align}
where the metric is taken for $g_t$. If we take $k$-th derivative at $t=0$ for \eqref{EX218a}, we obtain
\begin{align*}
(2 \ddf +1) f^k+L_k(h^1,\cdots, h^{k},f^1,\cdots, f^{k-1})=0,
\end{align*}
where $L_k$ is an operator. For instance, if $k=1$, then it follows from \eqref{E217} that
\begin{align*}
L_1(h^1)=\dbf^2 h^1-(\ddf+1/2)\tr_{\bar g}(h^1).
\end{align*}
In particular, $f^1=\tr_{\bar g}(h^1)/2$ if $h^1 \in \ker(\dbf)$. Next, we define

\begin{defn} \label{def:ind1}
Given $h_1 \in \isd$, we say the sequence of functions $\{f_i\}$ is \textbf{induced} by $h_1$ if $f_1=\tr_{\bar g}(h_1)/2$ and for any $l \ge 2$,
\begin{align}\label{EX218d}
(2 \ddf +1) f_l+L_l(h_1,\cdots, h_{l},f_1,\cdots, f_{l-1})=0,
\end{align}
where $\{h_i\}$ is the sequence of $2$-tensors induced by $h_1$.
\end{defn}
Equation \eqref{EX218d} determines $f_l$ uniquely as a homogeneous polynomial of degree $l$ in $h_1$. As in Proposition \ref{P01}, we obtain the following estimate.

\begin{prop}
\label{P01X}
Suppose $h_1 \in \isd$ and $\{f_i\}$ is the sequence of functions induced by $h_1$. Then for every fixed $i \ge 1$ there is a constant $C_i$ such that
\begin{align*}
\|f_i\|_{C^{2,\alpha}} \leq C_i \|h_1\|^i_{C^{\alpha}}.
\end{align*}
\end{prop}

Given $h_1 \in \isd$ with induced sequences $\{h_i\}$ and $\{f_i\}$, define
\begin{align*}
g(t):=\bar g+\sum_{l=1}^k \frac{t^l}{l!}h_l \quad \text{and} \quad f(t):=f_{g(t)}.
\end{align*}
Then the uniqueness of the solution $f_l$ in \eqref{EX218d} gives
\begin{align} \label{EX219c}
\left.\frac{\mathrm{d}^i}{\mathrm{d}t^i}\right|_{t=0} f(t)=f_i
\end{align}
for any $1 \le i \le k$. In other words, $f(t)$ can be approximated by $\bar f+\sum_{l=1}^k \dfrac{t^l}{l!}f_l$.


\section{Elliptic estimates on non-compact Ricci shrinkers}

In this section, we consider an $m$-dimensional non-compact Ricci shrinker $(M^m,g,f)$ with normalization \eqref{E101}. The notations and operators in Section 2 (cf. Definition \ref{def:not}) can be defined similarly on $(M,g, f)$.

If $(M,g)$ has bounded curvature, then $|\na^k \Rm|$ is also bounded by Shi's local estimates \cite{Shi89A}. Moreover, it follows from the no-local-collapsing theorem (see \cite[Theorem 4.2]{Naber} and \cite[Theorem 23]{LW20}) that each unit ball of $M$ has a uniformly positive volume lower bound. In particular, $(M,g)$ has bounded geometry. 

First, we recall the following well-known estimate of $f$.

\begin{lem}[\cite{CZ10,HM11}]
\label{lem:pot}
There exists a point $p \in M$ where $f$ attains its infimum and $f$ satisfies the quadratic growth estimate
\begin{align*}
\frac{1}{4}\left(d(x,p)-5 m \right)^2_+ \le f(x) \le \frac{1}{4} \left(d(x,p)+\sqrt{2m} \right)^2
\end{align*}
for all $x\in M$, where $a_+ :=\max\{0,a\}$.
\end{lem}

Lemma \ref{lem:pot} indicates $f$ grows like $d^2(p ,\cdot)/4$. For this reason, we define $b:=2\sqrt f$, which can be regarded as a replacement for the distance function. Moreover, we have the following volume estimate.
\begin{lem}[\cite{CZ10,HM11}]
\label{lem:volume}
There exists a constant $C=C(m)>0$ such that for any $r >0$,
\begin{align*}
|B(p,r)| \le Cr^m.
\end{align*}
\end{lem}

In particular, Lemma \ref{lem:pot} and Lemma \ref{lem:volume} imply that any tensor $T$ with polynomial growth belongs to $L^2$.

To deal with the integration by parts on non-compact Ricci shrinkers, we recall that there exists a family of cutoff functions constructed in \cite[Section 3]{LW20}. Specifically, we define
\begin{align*}
\phi^r:=\eta \lc \frac{f}{r} \rc,
\end{align*}
where $\eta$ is a fixed smooth nonincreasing function on $\R$ such that $0\le\eta\le1$, $\eta=1$ on $(-\infty,1]$, and $\eta=0$ on $[2,\infty)$. In particular, $\phi^r$ is identically $1$ if $f \le r$ and supported on $f \le 2r$. Moreover, we have
\begin{align} \label{eq:cutoff}
|\na \phi^r|^2 \le C r^{-1} \quad \text{and} \quad |\Delta_f \phi^r| \le C,\qquad r\ge1,
\end{align}
for a constant $C=C(m)$. Here, the identity $\Delta_f f=m/2-f$ is crucial.

If we further assume $(M,g)$ has bounded curvature, it follows from the bound of $|\na^i \Rm|$ and the Ricci shrinker equation \eqref{E100} that for any $i \ge 2$,
\begin{align} \label{eq:cutoffaa}
|\na^i \phi^r| \le C_i r^{-1}, \qquad r\ge1.
\end{align}

We first derive the following formula for integration by parts for $(r,s)$-tensor fields, denoted by $T^{r,s}M$.

\begin{lem}[Integration by parts] \label{lem:ibp}
Let $S$ and $T$ be tensors of type $(r,s)$. If $S,T,\Delta_f S,\Delta_f T\in L^2$, then $\nabla S,\nabla T\in L^2$ and 
\begin{align} \label{EY301}
\int \la \Delta_f S,T \ra \,\mathrm{d}V_f=\int \la S, \Delta_f T \ra \,\mathrm{d}V_f=-\int \la \na S, \na T \ra \,\mathrm{d}V_f.
\end{align} 
\end{lem}

\begin{proof}
We first prove that $|\na T| \in L^2$. 
Applying the integration by parts, we obtain
\begin{align} \label{ex301a}
2\int (\phi^r)^2 (\la \Delta_f T,T \ra+|\na T|^2) \,\mathrm{d}V_f= \int (\phi^r)^2 \Delta_f |T|^2\,\mathrm{d}V_f=-2 \int \la \na \phi^r,\na |T|^2 \ra \phi^r \,\mathrm{d}V_f.
\end{align}
By using Cauchy–Schwarz inequality, it follows from \eqref{ex301a} that
\begin{align*}
\int (\phi^r)^2 (\la \Delta_f T,T \ra+|\na T|^2) \,\mathrm{d}V_f \le \frac{1}{2} \int (\phi^r)^2|\na T|^2\,\mathrm{d}V_f+2 \int |\na \phi^r|^2 |T|^2\,\mathrm{d}V_f
\end{align*}
and hence
\begin{align} \label{ex301b}
\int (\phi^r)^2 |\na T|^2\,\mathrm{d}V_f \le -2\int (\phi^r)^2 \la \Delta_f T,T \ra \,\mathrm{d}V_f+4 \int |\na \phi^r|^2 |T|^2\,\mathrm{d}V_f.
\end{align}
Since $T,\Delta_f T \in L^2$, it is clear from \eqref{ex301b} by taking $r \to \infty$ that $|\na T| \in L^2$.

Applying the integration by parts again, we have
\begin{align} \label{ex301c}
2\int \phi^r (\la \Delta_f T,T \ra+|\na T|^2) \,\mathrm{d}V_f =\int \phi^r \Delta_f |T|^2\,\mathrm{d}V_f= \int |T|^2 \Delta_f \phi^r \,\mathrm{d}V_f.
\end{align}

By using \eqref{eq:cutoff}, we immediately conclude from \eqref{ex301c} that
\begin{align} \label{ex301d}
\int |\na T|^2+\la \Delta_f T,T \ra \,\mathrm{d}V_f=0.
\end{align}

Applying integration by parts with the cutoff $\phi^r$ and letting $r\to\infty$ gives the equality with the gradient pairing in \eqref{EY301}, since the boundary term is bounded by $\|\nabla\phi^r\|_{C^0}\|\nabla S\|_{L^2}\|T\|_{L^2}\to0$. Interchanging $S$ and $T$ gives the remaining equality.
\end{proof}

\begin{rem} \label{rem:appr}
It is a standard fact that the space of smooth tensors with compact support is dense in $H^{1}$. Indeed, it follows from \eqref{eq:cutoff} that any $T \in H^{1}$ can be approximated by $\phi^r T$. Moreover, $\phi^r T$ can be further approximated by smooth tensors with compact support through a standard convolution argument. Furthermore, if we assume $|\Rm|$ is bounded, then the space of smooth tensors with compact support is dense in $H^{k}$ for any $k \ge 2$ by using \eqref{eq:cutoffaa}.
\end{rem}

Applying \eqref{EY301}, we immediately obtain
\begin{cor}\label{cor:w1}
For any $\ep>0$ and tensor field $T \in L^2$ with $\Delta_f T \in L^2$,
\begin{align} \label{EY303a}
\|T \|^2_{H^{1}} \le \ep \|\Delta_f T\|^2_{L^2}+\lc 1+\frac{1}{4\ep} \rc\|T\|_{L^2}^2.
\end{align} 
\end{cor}

Next, we have the following elementary estimate (cf. \cite[Lemma 2.10]{CM21b}).

\begin{lem} \label{lem:esf}
For any $k \ge 0$, there exists a constant $C_k=C_k(m)$ such that for any tensor $T \in H^{k}(T^{r,s}M)$,
\begin{align} \label{EX302aa}
\int f^k |T|^2 \,\mathrm{d}V_f \le C_k \|T\|_{H^{k}}^2.
\end{align} 
\end{lem}

\begin{proof}
The case $k=0$ is obvious. Suppose \eqref{EX302aa} holds for $k-1$.

We compute from $\Delta_f f=m/2-f$ that
\begin{align}\label{EY302a}
\Delta_f f^k=kf^{k-1} \Delta_f f+k(k-1)f^{k-2} |\na f|^2= \frac{km}{2}f^{k-1}-kf^k+k(k-1)f^{k-2} |\na f|^2.
\end{align}

Integration by parts with the cutoff $\phi^R$ gives
\begin{align}
&\int (\phi^R)^2(\Delta_f f^k)|T|^2\,\mathrm{d}V_f \notag \\
=&-\int(\phi^R)^2\langle\nabla f^k,\nabla|T|^2\rangle\,\mathrm{d}V_f
-\frac{2k}{R}\int\phi^R\eta'(f/R)f^{k-1}|\nabla f|^2|T|^2\,\mathrm{d}V_f \notag \\
\ge&-2k\int(\phi^R)^2f^{k-\frac12}|T||\nabla T|\,\mathrm{d}V_f \notag \\
\ge&-\frac{k}{2}\int(\phi^R)^2f^k|T|^2\,\mathrm{d}V_f
-2k\int(\phi^R)^2f^{k-1}|\nabla T|^2\,\mathrm{d}V_f, \label{EY302b}
\end{align}
where the cutoff term is nonnegative because $\eta'\le0$, and we have used $|\nabla f|^2\le f$.

Combining \eqref{EY302a} and \eqref{EY302b}, we have
\[
\frac12\int(\phi^R)^2f^k|T|^2\,\mathrm{d}V_f\le (m/2+k-1)\int f^{k-1}|T|^2\,\mathrm{d}V_f+2\int f^{k-1}|\nabla T|^2\,\mathrm{d}V_f\le C\|T\|_{H^{k}}^2,
\]
where we have used the inductive assumption for $T$ and $\nabla T$. Letting $R\to\infty$ and applying Fatou's lemma proves \eqref{EX302aa}.
\end{proof}

Next, we prove the following elliptic estimate for $\Delta_f$ on tensors.
\begin{prop}
\label{prop:ellip1}
Suppose $(M,g)$ has bounded curvature. For any $k \ge 2$, there exists a constant $C$ so that if $T \in L^2(T^{r,s}M)$ and $\Delta_f T \in H^{k-2}(T^{r,s}M)$, then $T \in H^{k}(T^{r,s}M)$ and
\begin{align} \label{EY303}
\|T \|_{H^{k}} \le C \lc \|\Delta_f T\|_{H^{k-2}}+\|T\|_{L^2} \rc.
\end{align} 
\end{prop} 
\begin{proof}
We first prove \eqref{EY303} for smooth $T$ with compact support.

For any $k \ge 1$, it is clear by direct computation that
\begin{align}
\Delta \na^k T=&\na^k \Delta T+\sum_{i=0}^{k} \na^i \Rm * \na^{k-i} T, \label{EY303b} \\
f_i \na_i \na^k T=&f_i \na^k \na_i T+\na f *\lc \sum_{i=0}^{k-1}\na^i \Rm * \na^{k-1-i} T \rc, \label{EY303c}
\end{align} 
where $A*B$ denotes a linear combination of contractions of the tensor product of $A$ and $B$.

In addition, we compute
\begin{align} 
\na^k (f_i \na_i T)=& f_i \na^k \na_i T+\sum_{i=2}^{k+1} \na^i f *\na^{k+2-i} T=f_i \na^k \na_i T+1*\na^k T+\sum_{i=0}^{k-1} \na^i \Rm* \na^{k-i} T, \label{EY303d}
\end{align} 
where we have used $\na^2f=g/2-\Rc$. Combining \eqref{EY303b}, \eqref{EY303c} and \eqref{EY303d}, we obtain
\begin{align} \label{EY303e}
\Delta_f \na^k T= \na^k \Delta_f T+1*\na^k T+\sum_{i=0}^{k} \na^i \Rm * \na^{k-i} T+\na f *\lc \sum_{i=0}^{k-1}\na^i \Rm * \na^{k-1-i} T \rc.
\end{align} 

Based on \eqref{EY303e} and the integration by parts, we have
\begin{align} 
&\int |\na (\na^k T)|^2\,\mathrm{d}V_f=-\int \la \Delta_f \na^k T,\na^k T \ra \,\mathrm{d}V_f \notag \\
\le&-\int\langle\nabla^k(\Delta_fT),\nabla^kT\rangle\,\mathrm{d}V_f+C\int|\nabla^kT|^2\,\mathrm{d}V_f\notag\\
&+C\int\sum_{i=0}^k|\nabla^i\Rm|\,|\nabla^{k-i}T|\,|\nabla^kT|\,\mathrm{d}V_f\notag\\
&+C\int\sum_{i=0}^{k-1}|\nabla f|\,|\nabla^i\Rm|\,|\nabla^{k-1-i}T|\,|\nabla^kT|\,\mathrm{d}V_f\notag\\
\le& -\int \la \na^k (\Delta_f T), \na^k T \ra \,\mathrm{d}V_f+C\int \sum_{i=0}^{k}|\na^i T|^2+\sum_{i=0}^{k-1} f | \na^{i} T|^2\,\mathrm{d}V_f \notag \\
\le& -\int \la \na^k (\Delta_f T), \na^k T \ra \,\mathrm{d}V_f+C \| T \|_{H^{k}}^2, \label{EY303f}
\end{align}
where we have used Lemma \ref{lem:esf} for the last inequality. Similarly, we compute
\begin{align} 
&-\int \la \na^k (\Delta_f T), \na^k T \ra \,\mathrm{d}V_f=\int \la \na^{k-1} (\Delta_f T), \Delta_f \na^{k-1} T \ra \,\mathrm{d}V_f \notag \\
\le& \int |\na^{k-1} (\Delta_f T)|^2 \,\mathrm{d}V_f+C\int |\na^{k-1} (\Delta_f T)| \lc \sum_{i=0}^{k-1} |\na^{i} T|+\sum_{i=0}^{k-2} |\na f || \na^{i} T| \rc \,\mathrm{d}V_f \notag \\
\le& C \int |\na^{k-1} (\Delta_f T)|^2 \,\mathrm{d}V_f+C \| T \|_{H^{k-1}}^2 \le C\lc \|\Delta_f T\|_{H^{k-1}}^2+ \| T \|_{H^{k-1}}^2 \rc. \label{EY303g}
\end{align} 

Combining \eqref{EY303f} and \eqref{EY303g}, we conclude that
\begin{align}
\int |\na^{k+1} T|^2\,\mathrm{d}V_f \le C\lc \|\Delta_f T\|_{H^{k-1}}^2+ \| T \|_{H^{k}}^2 \rc. \label{EY303h}
\end{align} 
For compactly supported smooth $T$, the conclusion follows from \eqref{EY303a}, \eqref{EY303h} and induction.

For general $T$, local elliptic regularity and Lemma \ref{lem:ibp} give $T\in H^{k}_{\mathrm{loc}}\cap H^1$. Inductively assuming $T\in H^{k-1}$, the cutoff bounds and
\[
\Delta_f(\phi^rT)=\phi^r\Delta_fT+2\nabla\phi^r\cdot\nabla T+(\Delta_f\phi^r)T
\]
give a uniform $H^{k-2}$ bound for $\Delta_f(\phi^rT)$. The compact-support estimate and weak compactness then give $T\in H^k$ and \eqref{EY303}.
\end{proof} 

We investigate the spectral properties of $\Delta_f$ acting on $T^{r,s}M$.

\begin{prop}
\label{prop:dis}
$-\Delta_f$ has discrete eigenvalues $0\le \mu_0\le \mu_1\le \cdots \to +\infty$ with finite-dimensional eigenspaces $E_{\mu_i} \subset H^{1}(T^{r,s}M)$. We write $E_0=\ker\Delta_f$. Moreover, there exists a constant $C$ such that for any $T \in H^{1}(T^{r,s}M)$ with $\Delta_fT\in L^2$, we have
\begin{align} \label{EX304aa}
\|T-\pi_{E_0} (T)\|_{H^{1}} \le C\| \Delta_f T\|_{L^2}.
\end{align}
If we assume $(M,g)$ has bounded curvature, then for any $T \in H^{k}(T^{r,s}M)$ with $k \ge 2$, we have
\begin{align} \label{EX304aax}
\|T-\pi_{E_0} (T)\|_{H^{k}} \le C\| \Delta_f T\|_{H^{k-2}}.
\end{align}
\end{prop} 
\begin{proof}

By Lemma \ref{lem:esf}, a bounded set in $H^{1}$ satisfies a uniform bound on $\int f|T|^2\,\mathrm{d}V_f$. Since $f$ is proper, this gives uniform $L^2$ tails outside its compact sublevel sets. Combining these tails with the usual local Rellich theorem proves that $H^{1}(T^{r,s}M)\hookrightarrow L^2(T^{r,s}M)$ is compact. The Friedrichs realization of $-\Delta_f$ therefore has compact resolvent.

The estimate $\|T-\pi_{E_0} (T)\|_{L^2} \le C\| \Delta_f T\|_{L^2}$ follows from a standard argument by using the decomposition of $T$ into the sum of eigentensors. In general, \eqref{EX304aa} and \eqref{EX304aax} follow from Corollary \ref{cor:w1} and Proposition \ref{prop:ellip1}, respectively.
\end{proof} 

As a corollary of Proposition \ref{prop:dis}, we have the following characterization of the Sobolev spaces, which is well-known for compact manifolds (cf. \cite[Chapter 5]{Tay11}).

\begin{cor}\label{cor:sobo}
Suppose $(M,g)$ has bounded curvature. Let $\{U_i\}$ be an orthonormal eigenbasis of $L^2(T^{r,s}M)$, with $-\Delta_fU_i=\mu_iU_i$, and put $x_i=\langle T,U_i\rangle_{L^2}$. For every integer $k\ge0$, $T\in H^{k}$ if and only if
\[
 \|T\|_{D^k}^2:=\sum_i(1+\mu_i)^k|x_i|^2<\infty.
\]
This norm is equivalent to $\|T\|_{H^{k}}$.
\end{cor}
\begin{proof}
For $k=0$ this is Parseval's identity. For $k=1$, integration by parts and density in the form domain give
\[
 \sum_i(1+\mu_i)|x_i|^2=\|T\|_{L^2}^2+\|\nabla T\|_{L^2}^2.
\]
For $k\ge2$, Lemma \ref{lem:esf}, the bounded derivatives of $\nabla^2f$, and the product rule imply that $\Delta_f:H^{k}\to H^{k-2}$ is bounded. Proposition \ref{prop:ellip1} gives the reverse graph-norm estimate. The equivalence follows by induction using
\begin{equation}\label{EY303ea}
 \sum_i(1+\mu_i)^k|x_i|^2
 \asymp \|T\|_{L^2}^2+
 \sum_i(1+\mu_i)^{k-2}\mu_i^2|x_i|^2
\end{equation}
for $k\ge2$.
\end{proof}


For the rest of the section, we assume $|\Rm|$ is bounded and focus on operators acting on symmetric $2$-tensors or $1$-forms.

We first consider the operator $\PP=\df \df^*$ on $1$-forms on $M$. This is a self-adjoint elliptic operator (cf. \cite[Lemma 3.1]{CM21b}). By \cite[Lemma 5.7]{CM21b}, $\PP$ has a discrete spectrum with finite-dimensional eigenspaces generated by eigenvectors of $\Delta_f$. Its $L^2$ kernel $\KK_{\PP}$ is therefore finite-dimensional, and each element lies in $H^{k}\cap C^{\infty}(T^*M)$ for every $k$. Integration by parts identifies $\KK_{\PP}$ with the space of $L^2$ Killing forms.

Next, we prove the following existence result, which is similar to \cite[Theorem 5.5]{CM21b}.
\begin{prop}
\label{P301}
For any $k \ge 1$, there exists a constant $C$ so that for any $h \in H^{k}(S^2(M))$, there exists a unique $1$-form $w \in H^{k+1}(T^*M)$ with $\PP w=\df h$ that is $L^2$-orthogonal to $\KK_{\PP}$ and satisfies 
\begin{align*}
\| w \|_{H^{k+1}} \leq C \| \df h \|_{H^{k-1}}.
\end{align*} 
\end{prop} 
\begin{proof}
Put $F=\delta_fh$. Integration by parts shows $F\perp\mathcal K_{\mathcal P}$. By \cite[Lemma 5.7]{CM21b}, the orthogonal complement of this kernel has an orthonormal basis $\{U_i\}_{i\ge1}$ such that
\[
 \mathcal PU_i=\lambda_iU_i,\qquad
 \Delta_fU_i=-\tau_iU_i,\qquad
 0<\lambda_1\le\lambda_2\le\cdots,\qquad
 0\le\tau_i\le\tfrac12+2\lambda_i.
\]
Writing $F=\sum_{i\ge1}x_iU_i$, set $w=\sum_{i\ge1}(x_i/\lambda_i)U_i$. Then
\[
 \|w\|_{D^{k+1}}^2
 =\sum_{i\ge1}\frac{(1+\tau_i)^2}{\lambda_i^2}
                  (1+\tau_i)^{k-1}|x_i|^2
 \le C\|F\|_{D^{k-1}}^2,
\]
since the displayed ratios are uniformly bounded. Corollary \ref{cor:sobo} proves the required estimate and equation.
\end{proof}

\begin{rem}
The estimate in \cite[Theorem 5.5]{CM21b} applies without a curvature assumption. For smooth compactly supported $h$, it gives a $1$-form $w \in H^{1}(T^*M)$ with $\PP w=\df h$ that is $L^2$-orthogonal to $\KK_{\PP}$ and satisfies
\begin{align*}
\| w \|_{H^{1}} \leq C \| \df h \|_{L^2}
\end{align*} 
The conclusion extends to $h\in H^{1}(S^2(M))$ by density, using Remark \ref{rem:appr} and Lemma \ref{lem:esf}.
\end{rem}

Proposition \ref{P301} gives the following $L^2$-orthogonal decomposition, a non-compact version of Proposition \ref{prop:decom1}.

\begin{prop} \label{prop:decom2}
For any $k \ge 1$, we have the $L^2$-orthogonal decomposition
\begin{align*}
H^{k}(S^2(M))=\ker \lc \df \vert_{H^{k}(S^2(M))} \rc \oplus \df^*\lc H^{k+1}(T^* M)\rc.
\end{align*}
\end{prop}

We next give a localized quantitative gauge construction.

\begin{thm}
\label{thm:diff2}
Fix $\alpha\in(0,1)$. There are positive constants $C_k$ with the following property.

Let $L\gg1$ and let $h$ be a smooth symmetric $2$-tensor supported in $\{b\le L\}$. Then there exists a smooth $1$-form $w$ supported in $\{b \le L-1/2 \}$ such that for any $k \ge 4$,
\begin{align} \label{eq:esw1}
\|w\|_{C^{k,\alpha}} \le C_k L^{m+k+8} \sup_{b \le L}\| \df h\|_{C^{k-2,\alpha}}
\end{align}
and for $k=2,3$,
\begin{align} \label{eq:esw2}
\|w\|_{C^{k,\alpha}} \le C_k L^{m+k+8} \sup_{b \le L}\| \df h\|_{C^{2}}
\end{align}

Moreover, $2\PP w=\df h$ on $\{b\le L-1\}$. For the following nonlinear estimate fix $k\ge2$ and additionally assume
\[
 L^{m+k+11}\|\delta_fh\|_{C^{k+1,\alpha}}\le\epsilon_k,
\]
where $\epsilon_k>0$ is sufficiently small. Then
\begin{align} \label{eq:esd}
&\sup_{b \le L-1} \|\df(\varphi_w^*(g+h)-g)\|_{C^{k,\alpha}} \notag \\
\le& C_k L^{m+k+11} \sup_{b \le L} \lc \|h\|_{C^{k+2,\alpha}} \|\df h\|_{C^{k,\alpha}}+L^{m+k+12}(1+\|h\|_{C^{k+3,\alpha}})\| \df h\|^2_{C^{k+1,\alpha}} \rc. 
\end{align}
\end{thm}

\begin{proof}
In the proof, all constants $C$ depend on $k$ and the given Ricci shrinker, which may differ line by line.

Since $h$ is compactly supported, $\delta_fh$ is perpendicular to all Killing forms. Proposition \ref{P301} therefore gives a unique smooth $1$-form $w_1\perp\mathcal K_{\mathcal P}$ such that
\begin{align} \label{ey301a}
2\PP w_1=\df h
\end{align}
and 
\begin{align} \label{ey301b}
\|w_1\|^2_{H^{2}} \le C \|\df h\|^2_{L^2}.
\end{align}

Choose a fixed $L_0$ sufficiently large in terms of the background shrinker, and set $L_2:=4L_0$. We claim that for every $r\ge L_2$,
\begin{align} \label{ey301c}
I_{w_1}(r) \le C r^3 \lc \| \df h\|^2_{L^2}+ \int b^{2-m} \lc |\df h|^2+|\na \df (\df h)|^2 \rc\,\mathrm{d}V_g \rc,
\end{align}
where $I_{T}(r):=r^{1-m} \int_{b=r}|T|^2|\na b|$ for any tensor $T$. Indeed, \eqref{ey301a} and \cite[Theorem 5.4]{CM21b}, with its parameters $\lambda=0$, $\beta=1/16$, and $\delta=1/2$, give for any $r_2>r_1\ge L_0$,
\begin{align*}
I_{w_1}(r_2) \le C \lc \frac{r_2}{r_1} \rc^3 \lc I_{w_1}(r_1)+I_{\na \df w_1}(r_1)+ \int b^{2-m} \lc |\df h|^2+|\na \df (\df h)|^2 \rc\,\mathrm{d}V_g \rc.
\end{align*}
Choose $L_0$ so that $|\nabla b|$ is bounded above and below on $\{b\ge L_0\}$. Integrating $r_1$ over $[L_0,2L_0]$ and using the coarea formula on this fixed annulus, we conclude from \eqref{ey301b} that \eqref{ey301c} holds for every $r\ge L_2$. It follows that for any $r \in [L_2,L]$,
\begin{align}\label{ey301d}
I_{w_1}(r) \le C L^7 \sup_{b \le L}\| \df h\|_{C^2}^2
\end{align}
since
\begin{align*}
\int b^{2-m} \lc | \df h|^2+|\na \df (\df h)|^2 \rc\,\mathrm{d}V_g \le C \sup_{b \le L} \| \df h\|_{C^2}^2 \lc 1+\int_{1\le b \le L} b^{4-m}\,\mathrm{d}V_g\rc \le C L^4 \sup_{b \le L}\| \df h\|_{C^2}^2,
\end{align*}
where we have used Lemma \ref{lem:volume}.

For any $x$ with $2L_2 \le b(x) \le L-1/2$, we consider the ball $B(x,2L^{-1})$, which has bounded geometry. The radius is chosen to be $2L^{-1}$ so that the coefficients of the elliptic system $\PP$ have uniformly bounded H\"older norms after rescaling. By the Schauder estimate (cf. \cite{DN55,GT01}), we obtain
\begin{align} \label{ey301e}
\sup_{B(x,L^{-1})} \|w_1\|^2_{C^{k+2,\alpha}} \le C L^{2k+4+2\alpha} \lc L^{-4}\sup_{B(x,2L^{-1})} \|\PP w_1\|^2_{C^{k,\alpha}}+L^m \int_{B(x,2L^{-1})} |w_1|^2 \,\mathrm{d}V_g \rc
\end{align}

Since we assume the curvature is bounded, $|\na b|$ is almost equal to $1$ by \eqref{E101}, if $L_2$ is sufficiently large. Therefore, it follows from \eqref{ey301d} and the coarea formula that
\begin{align} \label{ey301f}
\int_{B(x,2L^{-1})} |w_1|^2 \,\mathrm{d}V_g \le C L^{m-2} \sup_{r\in[b(x)-3L^{-1},b(x)+3L^{-1}]} I_{w_1}(r) \le CL^{m+5} \sup_{b \le L}\| \df h\|_{C^2}^2.
\end{align}

If $k \ge 2$, it follows from \eqref{ey301e} and \eqref{ey301f} that
\begin{align*}
\sup_{2L_2 \le b \le L-1/2} \|w_1\|_{C^{k+2,\alpha}} \le C L^{m+k+\alpha+9/2} \sup_{b \le L}\| \df h\|_{C^{k,\alpha}} \le CL^{m+k+10} \sup_{b \le L} \| \df h\|_{C^{k,\alpha}}.
\end{align*}
By using the Schauder estimate again for $b \le 2L_2$, we obtain
\begin{align} \label{ey302a}
\sup_{b \le L-1/2} \|w_1\|_{C^{k+2,\alpha}} \le CL^{m+k+10} \sup_{b \le L} \| \df h\|_{C^{k,\alpha}}.
\end{align}
We choose another cutoff function $\rho$ which is identically $1$ if $b \le L-1$ and supported in $\{b \le L-1/2\}$ and set $w:=\rho w_1$. It is clear from \eqref{ey302a} that 
\begin{align} \label{ey302b}
\|w\|_{C^{k+2,\alpha}} \le CL^{m+k+10} \sup_{b \le L} \| \df h\|_{C^{k,\alpha}}.
\end{align}
Moreover, $2\PP w=\df h$ if $b \le L-1$. Thus, \eqref{eq:esw1} is proved, and \eqref{eq:esw2} can be proved similarly.

We set $\varphi^t$ to be the family of diffeomorphisms generated by $w$ and $g(t):=(\varphi^t)^*(g+h)$ and consider the identity
\begin{align} \label{ey302c}
g(1)=g(0)+g'(0)+\int_0^1 \int_0^t g''(s)\,\mathrm{d}s\,\mathrm{d}t.
\end{align}
The smallness assumption and \eqref{ey302b} give uniform pullback bounds. Since $g''(s)=(\varphi^s)^*L_{w^\sharp}^2(g+h)$, the product rule gives
\begin{align} \label{ey302d}
\sup_{b \le L-1/2} \|g''(s)\|_{C^{k,\alpha}} \le C (1+ \sup_{b \le L-1/2} \|h\|_{C^{k+2,\alpha}}) \|w\|^2_{C^{k+2,\alpha}}.
\end{align}

Since $\df(g'(0))+2\PP w=\df(L_{w^\sharp}h)$, the product rule and $|\nabla f|\le L/2$ on $\{b\le L-1/2\}$ give
\begin{align} \label{ey302e}
\|\df(g'(0))+2\PP w\|_{C^{k,\alpha}} \le C L \sup_{b \le L-1/2} \|h\|_{C^{k+2,\alpha}} \|w\|_{C^{k+2,\alpha}}.
\end{align}

By taking $\df$ of \eqref{ey302c}, it follows from \eqref{ey302d}, \eqref{ey302e} and \eqref{ey302b} that if $k \ge 2$
\begin{align*} 
&\sup_{b \le L-1} \|\df (\varphi_w^{*}(g+h) -g) \|_{C^{k,\alpha}} \notag \\
=& \sup_{b \le L-1} \|\df(g(1)-g(0)-g'(0))+\df h+\df (g'(0)) \|_{C^{k,\alpha}} \notag \\
\le & CL \lc \sup_{b \le L} \|h\|_{C^{k+2,\alpha}} \|w\|_{C^{k+2,\alpha}}+(1+ \sup_{b \le L} \|h\|_{C^{k+3,\alpha}}) \|w\|^2_{C^{k+3,\alpha}} \rc \notag\\
\le & C L \sup_{b \le L} \lc L^{m+k+10} \|h\|_{C^{k+2,\alpha}} \|\df h\|_{C^{k,\alpha}}+L^{2(m+k+11)}(1+\|h\|_{C^{k+3,\alpha}})\| \df h\|^2_{C^{k+1,\alpha}} \rc.
\end{align*}
Therefore, \eqref{eq:esd} is proved.
\end{proof}

Next, we investigate the operator $\LL=\Delta_f+2\Rm$ on symmetric $2$-tensors on $M$. It is clear that $\LL$ is a self-adjoint elliptic operator. Since $|\Rm|$ and its higher derivatives are bounded, it is easy to derive the following result from Proposition \ref{prop:ellip1}.

\begin{prop}
\label{prop:esti2}
For any $k \ge 2$, there exists a constant $C$ so that if $h \in L^2(S^2(M))$ and $\LL h \in H^{k-2}(S^2(M))$, then $h \in H^{k}(S^2(M))$ and
\begin{align*} 
\|h \|_{H^{k}} \le C \lc \|\LL h\|_{H^{k-2}}+\|h\|_{L^2} \rc.
\end{align*} 
\end{prop} 

From the same proof as Proposition \ref{prop:dis} and Corollary \ref{cor:sobo}, we have

\begin{prop}
\label{PX302}
$-\LL$ has discrete eigenvalues $\mu_0\le \mu_1\le \cdots \to +\infty$ with finite-dimensional eigenspaces $E_{\mu_i} \subset H^{k}(S^2(M))$ for every $k\ge0$. We set $E_0=\ker\LL$. Moreover, for $k \ge 2$ and any $h \in H^{k}(S^2(M))$, we have
\begin{align*}
\|h-\pi_{E_0} (h)\|_{H^{k}} \le C\| \LL h\|_{H^{k-2}}.
\end{align*}
Let $\{U_i\}$ be an orthonormal basis of $L^2(S^2(M))$ such that $\LL U_i+\mu_i U_i=0$. For any $h \in L^2(S^2(M))$, we set $x_i:=\la h,U_i\ra_{L^2}$. For a fixed $a>1+2\sup_M|\Rm|$ and any $k\ge0$, $h\in H^k$ if and only if
\begin{align*}
\|h\|^2_{D^k}:=\sum_{i} (a+\mu_i)^k x_i^2<\infty. 
\end{align*} 
The norms $\|\cdot\|_{D^k}$ and $\|\cdot\|_{H^{k}}$ are comparable.
\end{prop} 

\begin{exmp} \label{ex:sr}
We consider $(M,g, f)=(\R^m,g_E,f_E)$, where $g_E$ is the flat metric and $f_E=|z|^2/4$ with coordinates $z$ on $\R^m$. In this case, $\LL=\Delta_{f_E}$ on symmetric $2$-tensors on $\R^m$. If $h$ is an eigentensor of $\Delta_{f_E}$, so is each of its components. It is well-known that all possible eigenvalues are $0,1/2,1,3/2,\cdots$, and the eigenspace with respect to $k/2$ consists of all $k$-th Hermite polynomials. Therefore, $E_{k/2}$ consists of all symmetric $m \times m$ matrices with $k$-th Hermite polynomials as components.
\end{exmp}

Notice that $\df$ can be defined on $L^2(S^2(M))$ in the distribution sense. Let $\ker(\df)$ be the kernel space of $\df$ in $L^2$, which is obviously a closed subspace. Based on Proposition \ref{prop:decom2} and Lemma \ref{lem:pre}, we have

\begin{lem} \label{lem:hperp}
For any $k \ge 0$, there exists a constant $C>1$ such that
\begin{align*} 
C^{-1} \|u+v\|^2_{H^{k}} \le \|u\|_{H^{k}}^2+\|v\|^2_{H^{k}} \le C\|u+v\|^2_{H^{k}}
\end{align*}
for any $u \in H^{k}(S^2(M)) \cap \ker(\df)$ and $v \in H^{k}(S^2(M)) \cap \ker(\df)^{\perp}$.
\end{lem}

\begin{proof}
We claim that there exists an orthonormal basis $\{u_i,v_i\}$ of $L^2(S^2(M))$ such that $u_i \in \ker(\df)$ and $v_i \in \ker(\df)^{\perp}$. Moreover, $\LL u_i+\lambda_i u_i=0$ and $\LL v_i+\lambda'_i v_i=0$ for $\lambda_i,\lambda_i' \to +\infty$. Indeed, we consider an orthonormal basis $\{U_i\}$ such that $\LL U_i+\mu_i U_i=0$. Thanks to Proposition \ref{prop:decom2}, we may decompose 
\begin{align*} 
U_i=U_i^0+U_i^1,
\end{align*}
where $U_i^0 \in \ker(\df)$ and $U_i^1 \in \ker(\df)^{\perp} \cap \im(\df^*)$. From Lemma \ref{lem:pre}, we easily conclude that
\begin{align*} 
\LL U^0_i+\mu_i U^0_i=0 \quad \text{and} \quad \LL U^1_i+\mu_i U^1_i=0.
\end{align*}
The projected eigentensors span $\ker(\df)$ and $\ker(\df)^{\perp}$, respectively. Orthonormalizing within each eigenspace gives the required bases.

For any $u \in H^{k}(S^2(M)) \cap \ker(\df)$ and $v \in H^{k}(S^2(M)) \cap \ker(\df)^{\perp}$, we set
\begin{align*} 
u=\sum x_i u_i \quad \text{and} \quad v=\sum y_i v_i.
\end{align*}
Therefore, we conclude from Proposition \ref{PX302} that
\begin{align*} 
\|u+v\|^2_{H^{k}}
&\le C\|u+v\|^2_{D^k}
 =C\left(\sum_i(a+\lambda_i)^k x_i^2+\sum_i(a+\lambda_i')^k y_i^2\right)\\
&=C\bigl(\|u\|_{D^k}^2+\|v\|_{D^k}^2\bigr)
 \le C\bigl(\|u\|_{H^{k}}^2+\|v\|_{H^{k}}^2\bigr).
\end{align*}
The other inequality can be proved similarly.
\end{proof}

From Lemma \ref{lem:pre} and Lemma \ref{lem:hperp}, it is easy to see for any $k \ge 2$ if $u \in \ker(\df) \cap H^{k}$ and $v \in \ker(\df)^{\perp} \cap H^{k}$, then $\LL u \in \ker(\df)$, $\LL v \in \ker(\df)^{\perp}$ and $\|\LL u\|_{H^{k-2}}^2+\|\LL v\|_{H^{k-2}}^2 \approx \|\LL(u+v)\|_{H^{k-2}}^2$.

We have the following definition similar to Definition \ref{def:ISD}.

\begin{defn}
The infinitesimal solitonic deformation space with respect to $(M, g, f)$ is defined as
\begin{align*}
\mathrm{ISD}:= \ker(\LL) \cap \ker(\df).
\end{align*}
\end{defn}
In particular, $\isd$ is finite-dimensional and each element belongs to $C^{\infty}(S^2(M)) \cap H^{k}(S^2(M))$ for any $k \ge 1$. 

Next, we prove the following estimate, which will play an essential role in deriving the rigidity inequality; see also \cite[Theorem 6.3]{CM21b} for the special case of cylinders.
\begin{thm}
\label{thm:est1}
For any $k \ge 2$, there exists a constant $C$ such that for any $h \in H^{k}(S^2 (M))$, we have
\begin{align} \label{EX306}
\| h - \pi_{\isd}(h) \|_{H^{k}} \leq C\lc \| \LL h\|_{H^{k-2}} + \| \df h \|_{L^2} \rc.
\end{align}
\end{thm} 

\begin{proof}
Without loss of generality, we assume additionally that $h$ is smooth with compact support. 

It follows from Proposition \ref{P301} that there exists a $1$-form $w\in H^{k+1}(T^*M)$ with $\PP w=\df h$ satisfying
\begin{align} \label{EX306a}
\| w \|_{H^{2}} \leq C \| \df h \|_{L^2}.
\end{align} 
Moreover, we set $h^0:=h-\df^* w$ and $h^1:=\df^* w$. It is clear from the definition that $h^0 \in \ker(\df)$ and $h^1 \in \ker(\df)^{\perp}$. In addition, it follows from Lemma \ref{lem:pre} and Lemma \ref{lem:hperp} that
\begin{align} \label{EX306b}
\| \LL h^0 \|_{H^{k-2}}^2+\| \LL h^1 \|_{H^{k-2}}^2 \le C \| \LL h \|_{H^{k-2}}^2.
\end{align} 

From Proposition \ref{prop:esti2}, \eqref{EX306a} and \eqref{EX306b} we obtain
\begin{align} \label{EX306f}
\|h^1\|_{H^{k}} \le C \lc \|h^1\|_{L^2}+\| \LL h^1 \|_{H^{k-2}} \rc \le C \lc \|\df h\|_{L^2}+\| \LL h \|_{H^{k-2}} \rc.
\end{align}

On the other hand, it is easy to see
\begin{align} \label{EX306g}
\pi_{\isd}(h)=\pi_{\isd}(h^0)+\pi_{\isd}(h^1)=\pi_{\isd}(h^0)=\pi_{E_0}(h^0).
\end{align} 
Therefore, it follows from Proposition \ref{PX302}, \eqref{EX306g} and \eqref{EX306b} that
\begin{align} \label{EX306h}
\|h^0-\pi_{\isd} (h)\|_{H^{k}} = \|h^0-\pi_{E_0} (h^0)\|_{H^{k}} \le C\|\LL h^0\|_{H^{k-2}} \le C\|\LL h\|_{H^{k-2}}
\end{align} 
Combining \eqref{EX306f} and \eqref{EX306h}, the proof of \eqref{EX306} is complete.
\end{proof}

For later applications, we consider $\Delta_f$ on $1$-forms and denote its $L^2$-kernel by $\KK_0$. Integration by parts shows that every $w\in\KK_0$ is parallel. Since $\Rc(w^\sharp,\cdot)=0$, the shrinker equation gives $\mathrm{d}(2\langle w,\mathrm{d}f\rangle)=w$. Thus a nonzero $w$ has a global affine coordinate, and completeness yields a global line factor. Consequently, if $d=\text{dim}(\KK_0)$, then $(M,g)$ splits off $\R^d$, and its coordinate $1$-forms $(\mathrm{d}z_1,\cdots,\mathrm{d}z_d)$ form a basis of $\KK_0$. The following result follows from the spectral gap (cf. \cite[Lemma 9.8]{CM21b}).

\begin{lem}[Poincar\'e inequality]
\label{lem:p1}
For any $w \in H^{1}(T^* M)$, we have
\begin{align*}
\|w-\pi_{\KK_0}(w)\|_{L^2} \le C \| \na w\|_{L^2}.
\end{align*}
\end{lem} 

For any Ricci shrinker $(M,g,f)$, we may decompose $(M,g,f)=(M' \times \R^d,g' \times g_E,f' + f_E)$, where $d=\text{dim}(\KK_0)$ and hence $(M',g',f')$ does not split off any line. Next, we obtain

\begin{prop} \label{prop:iddd0}
Any non-flat, non-compact Ricci shrinker $(M,g,f)$ with $\isd=0$ admits no nontrivial global Riemannian product decomposition.
\end{prop}

\begin{proof}
We first show $\KK_0=0$. Suppose otherwise. We assume $(M,g,f)=(M' \times \R^d,g' \times g_E,f' + f_E)$ for $d=\text{dim}(\KK_0) \ge 1$. By our assumption, $(M',g')$ is non-flat. Therefore, it is clear that $u \Rc(g')$ lies in $\isd$, where $u$ is any quadratic Hermite polynomial on $\R^d$. In sum, we obtain a contradiction and hence $\KK_0=0$. Next, if we assume $(M,g,f)=(M^{m_1}_1 \times M^{m_2}_2, g_1\times g_2,f_1+f_2)$, then both $(M_1,g_1)$ and $(M_2,g_2)$ are non-flat. By direct computation, $(f_1-m_1/2) \Rc(g_2)$ and $(f_2-m_2/2) \Rc(g_1)$ lie in $\isd$. Therefore, $f_1=m_1/2$ and $f_2=m_2/2$. In other words, both $(M_1,g_1)$ and $(M_2,g_2)$ are compact Einstein manifolds. However, it implies that $M$ is compact, and we obtain a contradiction.
\end{proof}

\section{Rigidity inequality for generalized cylinders}

In this section, we focus on generalized cylinders. That is, we assume $(\bar M,\bar g)=(N^n \times \R^{m-n},g_N \times g_E )$, where $(N^n,g_N)$ is a closed Einstein manifold with $\Rc(g_N)=g_N/2$. It is clear that $(\bar M,\bar g, \bar f)$ is a Ricci shrinker with normalization \eqref{E101} if $\bar f:=|z|^2/4+n/2$, where $z$ is the coordinate on $\R^{m-n}$. It is easy to see all possible choices of $\bar f$ with normalization \eqref{E101} are $|z-a|^2/4+n/2$ for $a \in \R^{m-n}$. Therefore, we may assume $\bar f$ is uniquely determined after a translation. We add the subscript $N$ (resp. $E$) to denote notations and operators on $(N,g_N)$ (resp. $(\R^{m-n},g_E,f_E)$).

In addition, we assume $(N^n,g_N)$ satisfies the following conditions: 
\begin{align}
\begin{cases}
\text{$(N^n,g_N)$ has an obstruction of order $3$}; \tag{$\ast$} \label{EX301}\\
-\dfrac{k}{2} \notin \text{spec}(\LL_N \vert_{\text{TT}}), \quad \forall k \in \N^+.
\end{cases}
\end{align}

The second condition in \eqref{EX301} follows from linear stability or $H$-stability (cf. Definition \ref{def:stable1} and Definition \ref{def:Hstable}). Separation of variables gives the following lemma.

\begin{lem}
The spectrum of $\LL$ satisfies
\begin{align} \label{EX304a}
\mathrm{spec}(\LL) = (\mathrm{spec}(\LL_N)+\mathrm{spec}(\Delta^0_{f_E})) \cup(\mathrm{spec}(\LL_E)+\mathrm{spec}(\Delta^0_{N})) \cup (\mathrm{spec}(\Delta^1_{N})+\mathrm{spec}(\Delta^1_{f_E})).
\end{align}
Moreover, $\LL$ has an eigenbasis consisting of products of the corresponding factor eigenvectors. Here, the superscript $0$ (resp. $1$) means that the corresponding operator acts on functions (resp. $1$-forms).
\end{lem} 
\begin{proof}
Thanks to Proposition \ref{PX302}, the conclusion follows similarly to \cite[Proposition 4.1]{Kr15}; see also \cite[Lemma 3.1]{AM11}.
\end{proof}

\begin{prop} \label{PX303}
We have the orthogonal decomposition
\begin{align} \label{EX304b}
\isd= \isd_N \oplus \KK_1,
\end{align}
where $\KK_1 := \{ u g_{N} \mid u \in H^{1}(\R^{m-n}), \, \Delta_{f_E} u + u =0 \}$.
\end{prop}
\begin{proof}
For $h\in\isd$, we obtain a decomposition according to \eqref{EX304a}
$$
h=h^1+h^2+h^3,
$$
where each $h^i$ is a sum of products of the corresponding factor eigenvectors.
Notice that $\mathrm{spec}(\Delta^{1}_{N}) \ge 0$ and $0 \notin \mathrm{spec}(\Delta^{1}_{N})$ since otherwise the universal cover of $N$ splits off a line, which is impossible. Therefore, $h^3 \equiv 0$.
Since $\text{spec}(\LL_E)=\text{spec}(\Delta_{f_E}) \ge 0$ and $\text{spec}(\Delta^0_N) \ge 0$, $h^2=a_{ij}\mathrm{d}z^i \otimes \mathrm{d}z^j$ for some constants $a_{ij}$. From the types specified, it is clear that
\begin{align*}
\dbf(h^1)=\delta_N (h^1) \quad \text{and} \quad \dbf(h^2)=-h^2(\na f_E).
\end{align*}
Since $\dbf(h)=0$ and $\delta_N (h^1) \perp h^2(\na f_E)$, we conclude that
\begin{align} \label{EX305a}
\delta_N (h^1)=h^2(\na f_E)=0
\end{align}
and consequently $h^2 \equiv 0$, as each component of $h^2(\na f_E)$ is $a_{ij}z_j/2$.
Expanding $h^1$ in a Hermite eigenbasis on the Euclidean factor, it suffices to consider a term $u h_N$, where $u \in H^{1}(\R^{m-n})$ and $h_N \in C^{\infty}(S^2(N))$ satisfy
\begin{align} \label{EX305aa}
\Delta_{f_E} u+\lambda u=0 \quad \text{and} \quad \LL_N(h_N)-\lambda h_N=0
\end{align}
for $\lambda \ge 0$. By Example \ref{ex:sr}, $\lambda=k/2$ for some $k \in \N$. Moreover, it follows from \eqref{EX305a} that $\delta_N (h_N)=0$.
If we set $H:=\tr_{g_N}(h_N)$, by taking the trace of \eqref{EX305aa}, we obtain
\begin{align} \label{EX305b}
\Delta_N H+(1-\lambda)H=0.
\end{align}
From \eqref{EX305b}, we conclude that $H=0$ if $1-\lambda<0$. In this case, $h_N$ belongs to the TT subspace and $\lambda=k/2$ for $k \ge 3$. However, this contradicts \eqref{EX301} by \eqref{EX305aa}.

If $1-\lambda \ge 0$, then $\lambda=0,1/2$ or $1$. If $\lambda=0$, then $u$ must be constant by \eqref{EX305aa}. In this case, $u h_N \in \isd_N$. If $\lambda=1/2$, we obtain from \eqref{EX305b} that $\Delta_N H+H/2=0$. Since $1/2$ is not an eigenvalue of $\Delta_N$, we have $H=0$. However, this contradicts the fact that $-1/2 \notin \text{spec}(\LL_N \vert_{\text{TT}})$ from \eqref{EX301}.

If $\lambda=1$, it follows from \eqref{EX305b} that $H \equiv c$ for some constant $c$. Therefore, we obtain
\begin{align*}
\LL_N(h_N-cg_N/n)-(h_N-cg_N/n)=0,
\end{align*}
which implies that $h_N=cg_N/n$ since $-1 \notin \text{spec}(\LL_N \vert_{\text{TT}})$ by \eqref{EX301}. Consequently, $u h_N \in \KK_1$.

In sum, we obtain the decomposition \eqref{EX304b}. The decomposition is orthogonal since the integral of $u$ vanishes.
\end{proof}

We consider the deformation on the Ricci shrinker $(\bar M,\bar g, \bar f)$. Unlike in the compact case, the deformation of $\bar f$ is not a priori determined by the deformation of $\bar g$, since for a complete non-compact Riemannian manifold (even with bounded geometry), the minimizer of $\boldsymbol{\mu}(\cdot,1)$ may not exist (cf. \cite{Zhang12}). For this reason, we consider a pair $\{g,f\}$ which is in a small neighborhood of $\{\bar g ,\bar f\}$ in $C^2(S^2(\bar M)) \times C^2(\bar M)$. Moreover, we define
\begin{defn} 
For any $\{g,f\}$ close to $\{\bar g, \bar f\}$, we define
\begin{align*}
\bphi(\{g,f\})&= \frac{g}{2} - \Rc(g) - \nabla_g^2 f.
\end{align*}
In particular, $\bphi(\{g,f\})=0$ if and only if $(g,f)$ satisfies the Ricci shrinker equation \eqref{E100}.
\end{defn} 

Similar to Lemma \ref{lem:ephi}, we have the following variational formula of $\bphi$.
\begin{lem} \label{lem:ephi2}
For any $\{h,\chi\} \in C^{2}(S^2(\bar M)) \times C^{2}(\bar M)$, we have
\begin{align*}
\bphi'(\{h,\chi\}) =\frac{1}{2}\LL(h) + \dbf^{*} \dbf h +\nabla^2 (\mathrm{Tr}_{\bar g}(h)/2-\chi).
\end{align*}
\end{lem} 

For any $\{g,f\}$ nearby $\{\bar g,\bar f\}$, we set $\{h,\chi\}:=\{g-\bar g, f-\bar f\}$. For the rest of the section, we assume that for a given integer $k \ge 2$,
\begin{align} \label{EX401a}
\delta:=\|h \|_{C^k}+\| \chi\|_{C^k}
\end{align}
for a sufficiently small parameter $\delta>0$. Moreover, we denote by $C$ a constant which depends only on the Ricci shrinker $(\bar M,\bar g, \bar f)$ and $k$, unless otherwise stated. In the following, $C$ may be different line by line.

From Proposition \ref{PX303}, we set $\pi_{\isd}(h)=h_1+ug_N$, where $h_1 \in \isd_{N}$ and $u$ is a quadratic Hermite polynomial on $\R^{m-n}$. Let $\{h_i\}$ and $\{f_i\}$ be two sequences of $2$-tensors and functions respectively induced by $h_1$ (see Definition \ref{def:ind} and Definition \ref{def:ind1}). 

We define 
\begin{align*}
g^0(s,t) &= \bar g+s u g_N + t h_1 + \frac{t^2}{2} h_2,\\
g(s,t) &= \bar g+s \zeta + t h_1 + \frac{t^2}{2} h_2, \\
f^0(s,t) &= \bar f+s \frac{n}{2} u + t f_1 + \frac{t^2}{2} f_{2}, \\
f(s,t) &= \bar f+s q + t f_1 + \frac{t^2}{2} f_{2}, 
\end{align*}
where $\zeta:=h-h_1-h_2/2$ and $q:=\chi-f_1-f_2/2$. Moreover, we define 
\begin{align*}
\zeta':= \zeta-ug_N, \quad q':=q-\frac{n}{2}u \quad \text{and} \quad Q:= \frac{1}{2} \tr_{\bar g}(\zeta)-q=\frac{1}{2} \tr_{\bar g}(\zeta')-q'. 
\end{align*}

We also define a center of mass vector $\mathcal B(h,\chi)=(\mathcal B_1(h,\chi),\cdots, \mathcal B_{m-n}(h,\chi))$ by
\begin{align} \label{eq:center}
\mathcal B_i(h,\chi) :=\int \la \partial_{z_i}, \na \lc \frac{1}{2} \tr_{\bar g} (h)-\chi \rc \ra \,\vf=\frac{1}{2}\int z_i \lc \frac{1}{2} \tr_{\bar g} (h)-\chi \rc \,\vf.
\end{align}

Since $h_1,h_2,f_1$ and $f_2$ depend only on $N$ and the integral of $z_i$ vanishes, it is clear from \eqref{eq:center} that
\begin{align*}
\mathcal B_i(h,\chi)=\mathcal B_i(\zeta',q')= \int \la \partial_{z_i}, \na Q \ra \,\vf.
\end{align*} 

By the notation in Lemma \ref{lem:p1}, it is clear that $C_N \la \mathcal B(h,\chi),\mathrm{d}z \ra:=C_N \sum \mathcal B_i(h,\chi) \mathrm{d}z_i$ is the projection of $\na Q$ onto $\KK_0$, where $C_N:=(\int_{\bar M}\,\vf)^{-1}=e^{n/2}/(|N|(4\pi)^{(m-n)/2})$. It follows immediately from Lemma \ref{lem:p1} that
\begin{align} \label{eq:poin}
\|\na Q\|_{L^2} \le C \lc \| \na^2 Q\|_{L^2}+|\mathcal B(h,\chi)| \rc.
\end{align}

Let $\phi^0$ be the pointwise Taylor jet of $\bphi(\{g^0,f^0\})$ at $(0,0)$, and set $\phi=\bphi(\{g,f\})$. Lemma \ref{lem:ephi2} gives, at $t=s=0$,
\begin{align} \label{EX401b}
\phi^0_s=\phi^0_t=\phi_t=0.
\end{align}

For simplicity, we use the notation $[T]_k:=\sum_{0 \le i \le k}|\na^i T|$ for any tensor $T$ on $\bar M$.

\begin{lem} \label{lem:basices1}
Under the assumption \eqref{EX401a}, the following estimates hold.
\begin{enumerate}[label=\textnormal{(\alph{*})}]
\item $\|h_1\|_{C^k}+\|f_1\|_{C^k} \le C\delta$ and $\|h_2\|_{C^k}+\|f_2\|_{C^k} \le C\delta^2$.

\item $[u]_k \le C (1+|z|^2) \|u\|_{L^2} \le C(1+|z|^2) \delta$, where $z$ is the coordinate of $\R^{m-n}$.

\item $\|\zeta\|_{C^k}+\|q \|_{C^k} \le C\delta$.

\item For any $s,t \in [0,1]$, $\|g(s,t)-\bar g\|_{C^k}+\|f(s,t)-\bar f\|_{C^k} \le C\delta$.

\item For any constant $p>1$, with $C$ also depending on $p$, 
\begin{align} 
[\zeta]_k^p \le C \lc \delta^{p-1}[\zeta']_k+[u]^p_k \rc \quad \text{and} \quad [\na q]_{k-1}^p \le C \lc \delta^{p-1} [\na q']_{k-1}+[u]_k^p \rc. \label{EX401ca}
\end{align}
\end{enumerate}
\end{lem}
\begin{proof}
(a) From our definition, $\|h_1\|_{L^2} \le \|\pi_{\isd}(h)\|_{L^2} \le \|h\|_{L^2} \le C\delta$. Since $\|h_1\|_{C^k}$ is comparable to $\|h_1\|_{L^2}$ as $\isd_N$ is finite-dimensional, we obtain the estimate for $\|h_1\|_{C^k}$. The estimates for $\|f_1\|_{C^k}$, $\|h_2\|_{C^k}$ and $\|f_2\|_{C^k}$ can be derived from Proposition \ref{P01} and Proposition \ref{P01X}.

(b) Since $\|u\|_{L^2} \le C\|\pi_{\isd}(h)\|_{L^2} \le C\|h\|_{L^2}$, the estimate follows from \cite[Lemma 6.4]{CM21b}.

(c) From (a) and (b), it is clear that $\|h_1\|_{C^k}+\|h_2\|_{C^k}+\|f_1\|_{C^k}+\|f_2\|_{C^k} \le C\delta$. Therefore, the conclusion follows from the definitions of $\zeta$ and $q$.

(d) This is immediate from (c).

(e) We compute
\begin{align*}
[\zeta]_k^p\le C([\zeta']_k+[u]_k) [\zeta]^{p-1}_k \le C\delta^{p-1} [\zeta']_k+C[u]_k [\zeta]^{p-1}_k \le C\delta^{p-1} [\zeta']_k+C[u]^p_k+\frac{1}{2} [\zeta]_k^p,
\end{align*}
where we have used (c) and Young's inequality. In other words, $[\zeta]_k^p \le C\delta^{p-1} [\zeta']_k+C[u]_k^p$. By the same proof, we obtain the inequality for $[\na q]_{k-1}^p$, and hence the proof of \eqref{EX401ca} is complete.
\end{proof} 
Lemma \ref{lem:basices1} (d) indicates that $g(s,t)$ is indeed a metric. Next, we compute all second derivatives of $\phi^0$ at $t=s=0$.

\begin{prop} \label{prop:2nd}
Evaluated at $t=s=0$, we have
\begin{enumerate}[label=\textnormal{(\roman{*})}]
\item $\displaystyle \phi^0_{tt}=\phi_{tt} \in \isd_N$.

\item $\displaystyle \phi^0_{ss}=-|\na u|^2 g_N-n u \na^2 u-\frac{n}{2} \mathrm{d}u \otimes \mathrm{d}u$. Moreover, $\phi^0_{ss} \perp \isd_N$.

\item $\phi^0_{st}=\dfrac{1}{2} u h_1-\dfrac{1}{2} \na^2(H_1u)+\dfrac{1}{4}\mathrm{d}u \boxtimes \mathrm{d}H_1$, where $H_1:=\tr_{\bar g} (h_1)$ and $a \boxtimes b:= a \otimes b+b \otimes a$ denotes the symmetric tensor product. In addition, $\LL (\phi^0_{st})+\phi^0_{st}=0$. In particular, $\phi^0_{st} \perp \isd$. 
\end{enumerate}
\end{prop} 

\begin{proof}
(i) We consider $\tilde g(t):=g_N+t h_1 + t^2 h_2/2$ and $\phi_N:=\Phi_N(\tilde g(t))$ as $2$-tensors on $N$. Then it follows from Definition \ref{def:ind}, Definition \ref{def:ind1} and \eqref{EX219c} that
\begin{align*} 
\phi_{tt}=(\phi_N)_{tt}=\Phi'_N(h_2)+M_2(h_1)=\pi_{\ker(\Phi'_N)}(M_2(h_1)) \in \isd_N,
\end{align*}
where we have used \eqref{E203} and the fact that $M_2(h_1) \perp \im(\delta_N^*)$ by Lemma \ref{lem:perp1}.

(ii) The detailed calculation for the expression of $\phi^0_{ss}$ can be found in \cite[Proposition 8.1]{CM21b}. For any $\xi \in \isd_N$, it is clear that
\begin{align*}
\int \la \phi^0_{ss},\xi \ra\,\vf=-\int |\na u|^2 \tr_{\bar g}(\xi) \,\vf=0
\end{align*}
since $\int_N \tr_{\bar g}(\xi)=0$ from $\Delta_N(\tr_{\bar g}(\xi))+\tr_{\bar g}(\xi)=0$.

(iii) In this calculation, write $g=g^0(s,t)$ and $f=f^0(s,t)$. All parameter derivatives are evaluated at $t=s=0$ in normal coordinates for $\bar g$ at a fixed point. We first compute
\begin{align}
C_{ij}^k:=& (\Gamma_{ij}^k)_s= \frac{1}{2} g^{kl}\lc (g_s)_{jl,i}+(g_s)_{il,j}-(g_s)_{ij,l} \rc \notag \\
=&\frac{1}{2} g^{kl}\lc (ug_N)_{jl,i}+(ug_N)_{il,j}-(ug_N)_{ij,l} \rc \notag \\
=&\frac{1}{2} g^{kl} \lc u_i (g_N)_{jl}+u_j (g_N)_{il}-u_l (g_N)_{ij} \rc. \label{EX310b}
\end{align}
Similarly, we compute
\begin{align}
D_{ij}^k:= (\Gamma_{ij}^k)_t= \frac{1}{2} g^{kl}\lc (g_t)_{jl,i}+(g_t)_{il,j}-(g_t)_{ij,l}\rc \label{EX310c}
\end{align}
and hence
\begin{align}
(\Gamma_{ij}^k)_{st}=&\frac{1}{2} (g^{kl})_s \lc (g_t)_{jl,i}+(g_t)_{il,j}-(g_t)_{ij,l} \rc-g^{kl}C_{ij}^a (g_t)_{al} \notag \\
=& -\frac{1}{2} ug_N^{kl} \lc (h_1)_{jl,i}+(h_1)_{il,j}-(h_1)_{ij,l} \rc-g^{kl}C_{ij}^a (g_t)_{al},
\label{EX310d}
\end{align}
where we have used the fact that $\partial_s (g_t)_{jl,i}=-C_{ij}^a (g_t)_{al}-C_{il}^a (g_t)_{aj}$ since $g_{st}=0$. From \eqref{EX310b} and \eqref{EX310d} we have
\begin{align}
(\Gamma_{ij}^k)_{st}=-\frac{1}{2} u \lc (h_1)_{jk,i}+(h_1)_{ik,j}-(h_1)_{ij,k} \rc-\frac{1}{2}(u_i (h_1)_{jk}+u_j (h_1)_{ik}).
\label{EX310e}
\end{align}

Since $f_{st}=0$, we obtain
\begin{align}
(f_{ij})_{st}=&-(\Gamma_{ij}^k)_s (f_t)_k-(\Gamma_{ij}^k)_t (f_s)_k-(\Gamma_{ij}^k)_{st} f_k \notag \\
=&-\frac{1}{2} \lc u_i (g_N)_{jk}+u_j (g_N)_{ik}-u_k (g_N)_{ij} \rc(f_1)_k-\frac{n}{4} \lc (h_1)_{jk,i}+(h_1)_{ik,j}-(h_1)_{ij,k} \rc u_k-(\Gamma_{ij}^k)_{st} \bar f_k \notag \\
=&-\frac{1}{2} \lc u_i (f_1)_j+u_j (f_1)_i \rc=-\frac{1}{4} \lc u_i (H_1)_j+u_j (H_1)_i \rc, \label{EX310f}
\end{align}
where we have used the facts that $u$ and $\bar f$ depend only on $\R^{m-n}$ and $f_1=H_1/2$.

Recall the following expression of $\Rc$ under local coordinates (cf. \cite[1.43]{Be87})
\begin{align}\label{EX311a}
\mathrm{Ric}_{ij}=\partial_l \Gamma_{ij}^l-\partial_i \Gamma_{lj}^l+\Gamma_{ij}^a \Gamma_{la}^l-\Gamma_{lj}^a \Gamma_{ia}^l,
\end{align}
and hence we have
\begin{align*}
(\mathrm{Ric}_{ij})_{st}=& \partial_l (\Gamma_{ij}^l)_{st}-\partial_i (\Gamma_{lj}^l)_{st}+C_{ij}^a D_{la}^l+D_{ij}^a C_{la}^l-C_{lj}^a D_{ia}^l-D_{lj}^a C_{ia}^l.
\end{align*}

We compute by \eqref{EX310e} that
\begin{align}
\partial_l (\Gamma_{ij}^l)_{st}=&-\frac{1}{2} u\lc (h_1)_{jl,il}+(h_1)_{il,jl}-(h_1)_{ij,ll} \rc-\frac{1}{2} \lc u_i (h_1)_{jl,l}+u_j (h_1)_{il,l} \rc \notag \\
=& -\frac{1}{2} u\lc R_{lija} (h_1)_{al}+\mathrm{Ric}_{ia}(h_1)_{ja}+R_{ljia} (h_1)_{al}+\mathrm{Ric}_{ja}(h_1)_{ia}-(h_1)_{ij,ll} \rc=-\frac{1}{2} u(h_1)_{ij}, \label{EX311ba}
\end{align}
where we have used $\delta(h_1)=\delta_N(h_1)=0$, $\Delta_N h_1+2\Rm_N(h_1)=0$ and $\Rc_N=g_N/2$.

Similarly, we obtain
\begin{align}
\partial_i (\Gamma_{lj}^l)_{st}=&-\frac{1}{2} \lc u (H_1)_j+H_1 u_j \rc_i =-\frac{1}{2}(uH_1)_{ij}. \label{EX311bb}
\end{align}

Substituting \eqref{EX310b} and \eqref{EX310c} yields
\begin{align}
C_{ij}^a D_{la}^l=&\frac{1}{4} \lc u_i (g_N)_{ja}+u_j (g_N)_{ia}-u_a (g_N)_{ij} \rc (H_1)_a=\frac{1}{4} \lc u_i (H_1)_j +u_j (H_1)_i\rc. \notag \\
D_{ij}^a C_{la}^l =&\frac{n}{4} \lc (h_1)_{ja,i}+(h_1)_{ia,j}-(h_1)_{ij,a} \rc u_a=0.\notag \\
C_{lj}^a D_{ia}^l =&\frac{1}{4} \lc u_l (g_N)_{ja}+u_j (g_N)_{la}-u_a (g_N)_{lj} \rc \lc (h_1)_{al,i}+(h_1)_{il,a}-(h_1)_{ia,l} \rc=\frac{1}{4} u_j (H_1)_i. \notag \\
D_{lj}^a C_{ia}^l=&\frac{1}{4} \lc u_i (g_N)_{al}+u_a (g_N)_{il}-u_l (g_N)_{ia} \rc \lc (h_1)_{ja,l}+(h_1)_{la,j}-(h_1)_{lj,a} \rc=\frac{1}{4} u_i (H_1)_j. \label{EX311bc}
\end{align}

From \eqref{EX311ba}, \eqref{EX311bb} and \eqref{EX311bc}, we obtain
\begin{align}
(\mathrm{Ric}_{ij})_{st}=\frac{1}{2} \lc (u H_1)_{ij}-u(h_1)_{ij} \rc. \label{EX311d}
\end{align}

From \eqref{EX310f} and \eqref{EX311d}, we conclude that
\begin{align}
\phi^0_{st}=\frac{1}{2} u h_1-\frac{1}{2} \na^2(H_1u)+\frac{1}{4} \mathrm{d}u \boxtimes \mathrm{d}H_1. \label{EX3111a}
\end{align}

Next, by direct computation, we have
\begin{align*}
&\LL(u h_1)=(\ddf u) h_1+u (\LL(h_1))=-u h_1, \\
&\LL\lc \na^2(H_1u) \rc=\na^2 \lc \ddf (H_1u)+H_1u \rc=\na^2 \lc (\ddf H_1)u+H_1(\ddf u)+H_1u \rc=-\na^2(H_1u), \\
&\LL(\mathrm{d}u \boxtimes \mathrm{d}H_1)=\ddf(\mathrm{d}u) \boxtimes \mathrm{d}H_1+\mathrm{d}u\boxtimes \ddf(\mathrm{d}H_1)=-\mathrm{d}u \boxtimes \mathrm{d}H_1,
\end{align*}
where we have used \cite[Lemma 2.2, Corollary 2.7]{CM21b}. By these identities, it is clear from \eqref{EX3111a} that $\LL (\phi^0_{st})+\phi^0_{st}=0$.
\end{proof}

Next, we set
\begin{align*}
\alpha:=\|u\|_{L^2}, \quad \gamma=\gamma(z):=\alpha (1+|z|^2) \quad \text{and} \quad \beta:= \|h_1\|_{L^2}. 
\end{align*}
It follows from Lemma \ref{lem:basices1} that
\begin{align*}
\alpha+\beta \le C\delta,\quad \|h_2\|_{C^k}+\|f_2\|_{C^k} \le C \beta^2 \quad \text{and} \quad [u]_k \le C \gamma. 
\end{align*}
We next estimate the higher derivatives of $\phi$. For simplicity, we define $\la \cdot \ra_k:=[\cdot]_k+[\cdot]_{k-1}(1+|z|)$.

\begin{lem} \label{lem:esa}
There exists a constant $C$ such that for any $s,t \in [0,1]$,
\begin{align} 
& [\phi_{sss}]_{k-2}+[\phi_{stt}]_{k-2}+[\phi_{sst}]_{k-2} \notag \\
\le & C(1+|z|) \lc \delta [\zeta']^2_{k}+(\gamma^2+\beta^2)([\zeta']_{k}+[\na q']_{k-2})+\gamma^3+\beta^2 \gamma+\gamma^2 \beta \rc. \label{eq:esa1} \\
& [\phi_{sss}]_{k-2}+[\phi_{stt}]_{k-2}+[\phi_{sst}]_{k-2} \notag \\
\le & C \lc \delta^2\la \zeta'\ra_k+(\beta^2+\gamma^2) [\na q']_{k-2}+(1+|z|)(\gamma^3+\beta^2 \gamma+\gamma^2 \beta) \rc. \label{eq:esa11} \\
& [\phi_{ttt}]_{k-2} \le C(1+|z|) \beta^3. \label{eq:esa2} \\
& [\phi_{ssss}]_{k-2}+[\phi_{ssst}]_{k-2}+[\phi_{sstt}]_{k-2}+[\phi_{sttt}]_{k-2}+[\phi_{tttt}]_{k-2} \notag \\
\le& C\delta(1+|z|) \lc \delta [\zeta']^2_{k}+(\gamma^2+\beta^2)([\zeta']_{k}+[\na q']_{k-2})+\gamma^3+\beta^2 \gamma+\gamma^2 \beta \rc+C(1+|z|) \beta^4. \label{eq:esa3} \\
& [\phi_{ssss}]_{k-2}+[\phi_{ssst}]_{k-2}+[\phi_{sstt}]_{k-2}+[\phi_{sttt}]_{k-2}+[\phi_{tttt}]_{k-2} \notag \\
\le& C\delta \lc \delta^2\la \zeta'\ra_k+(\beta^2+\gamma^2) [\na q']_{k-2}+(1+|z|)(\gamma^3+\beta^2 \gamma+\gamma^2 \beta) \rc+C(1+|z|) \beta^4. \label{eq:esa4} 
\end{align}
\end{lem} 

\begin{proof}
From the definition, $\phi=g/2-\Rc-\na^2 f$, where $f=f(s,t)$ and the underlying metric is $g(s,t)$. It is clear from the local expression of $\Rc$ (cf. \eqref{EX311a}) that
\begin{align}
& [\Rc_{sss}]_{k-2}+[\Rc_{sst}]_{k-2}+[\Rc_{stt}]_{k-2} \notag\\
\le & C [\zeta]^3_k+C \beta^2 [\zeta]_k \notag\\
\le & C([\zeta']_k^2+\gamma^2)[\zeta]_k+C \beta^2 [\zeta']_k+C \beta^2 \gamma \notag\\
\le & C\delta [\zeta']_k^2+C(\gamma^2+\beta^2)[\zeta']_k+C(\beta^2 \gamma+\gamma^3). \label{EZ401a}
\end{align}

By using Lemma \ref{lem:basices1} (e), we estimate alternatively
\begin{align}
& [\Rc_{sss}]_{k-2}+[\Rc_{sst}]_{k-2}+[\Rc_{stt}]_{k-2} \notag\\
\le & C [\zeta]^3_k+C \beta^2 [\zeta]_k \notag\\
\le & C (\delta^2 [\zeta']_k+\gamma^3)+C \beta^2 [\zeta']_k+C \beta^2 \gamma \notag\\
\le & C\delta^2 [\zeta']_k+C(\beta^2 \gamma+\gamma^3). \label{EZ401b}
\end{align}

Since $\na^2 f=\partial_i \partial_j f-\Gamma_{ij}^l \partial_l f$ under the local coordinates, we obtain $[\partial_i \partial_j f_{sss}]_{k-2}=[\partial_i \partial_j f_{sst}]_{k-2}=[\partial_i \partial_j f_{stt}]_{k-2}=0$ and estimate similarly that
\begin{align}
&[(\Gamma_{ij}^l \partial_l f)_{sss}]_{k-2}+[(\Gamma_{ij}^l \partial_l f)_{sst}]_{k-2}+[(\Gamma_{ij}^l \partial_l f)_{stt}]_{k-2} \notag\\
\le & C(1+|z|) ([\zeta]^3_{k-1}+\beta^2[\zeta]_{k-1})+C \beta([\zeta]_{k-1}^2+\beta[\zeta]_{k-1})+C \beta^2[\zeta]_{k-1}+C[\na q]_{k-2}([\zeta]_{k-1}^2+\beta^2) \notag \\
\le & C(1+|z|) \lc \delta [\zeta']^2_{k-1}+(\gamma^2+\beta^2)[\zeta']_{k-1} +\gamma^3+\beta^2 \gamma\rc \notag \\
&+ C\lc (\beta+\delta) [\zeta']_{k-1}^2+\beta^2([\zeta']_{k-1}+[\na q']_{k-2})+\gamma^2[\na q']_{k-2}+\gamma^3+\beta^2 \gamma+\gamma^2 \beta \rc \notag\\
\le & C(1+|z|) \lc \delta [\zeta']^2_{k-1}+(\gamma^2+\beta^2)([\zeta']_{k-1}+[\na q']_{k-2}) +\gamma^3+\beta^2 \gamma+\gamma^2 \beta\rc. \label{EZ401c}
\end{align}

Similarly, by using Lemma \ref{lem:basices1} (e), we also obtain
\begin{align}
&[(\Gamma_{ij}^l \partial_l f)_{sss}]_{k-2}+[(\Gamma_{ij}^l \partial_l f)_{sst}]_{k-2}+[(\Gamma_{ij}^l \partial_l f)_{stt}]_{k-2} \notag\\
\le & C(1+|z|) ([\zeta]^3_{k-1}+\beta^2[\zeta]_{k-1})+C \beta([\zeta]_{k-1}^2+\beta[\zeta]_{k-1})+C \beta^2[\zeta]_{k-1}+C[\na q]_{k-2}([\zeta]_{k-1}^2+\beta^2) \notag \\
\le & C(1+|z|) \lc \delta^2 [\zeta']_{k-1}+(\beta^2+\delta \beta)[\zeta']_{k-1} +\gamma^3+\beta^2 \gamma+\gamma^2 \beta\rc+C(\beta^2+\gamma^2) [\na q']_{k-2} \notag \\
\le & C(1+|z|) \lc \delta^2 [\zeta']_{k-1} +\gamma^3+\beta^2 \gamma+\gamma^2 \beta \rc+C(\beta^2+\gamma^2) [\na q']_{k-2}. \label{EZ401d}
\end{align}

Since $[g_{sss}]_{k-2}=[g_{sst}]_{k-2}=[g_{stt}]_{k-2}=0$, \eqref{eq:esa1} follows from \eqref{EZ401a} and \eqref{EZ401c}, and \eqref{eq:esa11} follows from \eqref{EZ401b} and \eqref{EZ401d}.

The remaining estimates follow by the product rule, using $g_t,f_t=O(\beta)$ and $g_{tt},f_{tt}=O(\beta^2)$.
\end{proof}

\begin{lem} \label{lem:esa1}
There exists a constant $C$ such that at $s=t=0$, the following estimates hold.
\begin{enumerate}[label=\textnormal{(\roman{*})}]
\item $[\phi_{st}-\phi^0_{st}]_{k-2}\le C \beta ( \la \zeta' \ra_k+[\na q']_{k-2}).$

\item $[\phi_{ss}-\phi_{ss}^0]_{k-2} \le C( \la \zeta' \ra^2_k+[\na q']^2_{k-2})+C\gamma(\la\zeta'\ra_k+[\na q']_{k-2}).$

\item $[\phi_{ss}-\phi_{ss}^0]_{k-2} \le C(\delta+\gamma)( \la \zeta' \ra_k+[\na q']_{k-2})$.
\end{enumerate}
\end{lem} 

\begin{proof}
From our definition, we have at $s=t=0$ 
\begin{align*}
\phi_{st}-\phi^0_{st}=& \bphi^{(2)}\lc \{\zeta,q\},\{h_1,f_1\} \rc-\bphi^{(2)} \lc \{ug_N,nu/2\},\{h_1,f_1\} \rc= \bphi^{(2)}\lc \{\zeta',q'\},\{h_1,f_1\} \rc.
\end{align*}
As in the proof of Lemma \ref{lem:esa}, we obtain
\begin{align}
& [\bphi^{(2)}\lc \{\zeta',q'\},\{h_1,f_1\} \rc]_{k-2} \notag \\
\le & C \beta [\zeta']_k+C(1+|z|)\beta [\zeta']_{k-1}+C \beta([\zeta']_{k-1}+[\na q']_{k-2}) \le C\beta ( \la \zeta' \ra_k+[\na q']_{k-2}). \label{E403X1a}
\end{align}
From \eqref{E403X1a}, the proof of (i) is complete.

Similarly, we have
\begin{align*}
\phi_{ss}-\phi^0_{ss}=\bphi^{(2)}\lc \{\zeta',q'\},\{\zeta',q'\} \rc+2\bphi^{(2)}\lc \{\zeta',q'\},\{ug_N,nu/2\} \rc.
\end{align*}

It can be estimated as \eqref{E403X1a} that
\begin{align*}
[\bphi^{(2)}\lc \{\zeta',q'\},\{ug_N,nu/2\} \rc ]_{k-2} \le C \gamma ( \la \zeta' \ra_k+[\na q']_{k-2}). 
\end{align*}

One can estimate similarly that
\begin{align*}
& [\bphi^{(2)}\lc \{\zeta',q'\},\{\zeta',q'\} \rc]_{k-2} \\
\le & C[\zeta']_k^2+C(1+|z|) [\zeta']_{k-1}^2+C [\zeta']_{k-1} [\na q']_{k-2} \le C( \la \zeta' \ra^2_k+[\na q']^2_{k-2})
\end{align*}
and
\begin{align*}
& [\bphi^{(2)}\lc \{\zeta',q'\},\{\zeta',q'\} \rc]_{k-2} \\
\le & C[\zeta']_k^2+C(1+|z|) [\zeta']_{k-1}^2+C [\zeta']_{k-1} [\na q']_{k-2} \le C(\delta+\gamma) ( \la \zeta' \ra_k+[\na q']_{k-2}).
\end{align*}

Therefore, conclusions (ii) and (iii) are proved.
\end{proof}

\begin{prop} \label{prop:tay2}
There exists a constant $C$ such that 
\begin{align} 
\|\zeta'\|_{H^{k}}+\|\na q'\|_{H^{k-1}} \le C \lc \|\phi(1,1)\|_{H^{k-2}}+\|\dbf h\|_{H^{k-1}}+|\mathcal B(h,\chi)|+ \alpha^2+ \alpha \beta+\beta^3 \rc. \label{E402}
\end{align}
\end{prop} 
\begin{proof}
Fix $\ep\in(0,1/4)$ and put $\rho(z):=(1+|z|^2)^{1/2}$. We apply the Taylor expansion for $\phi$ and obtain
\begin{align}
&[\phi(1,1) -\phi_s-\phi_t- \frac{1}{2}(\phi_{ss}+2\phi_{st}+\phi_{tt})]_{k-2} \notag \\
\le & C \lc \delta^2\la \zeta'\ra_k+(\beta^2+\gamma^2) [\na q']_{k-2}+(1+|z|)(\gamma^3+\beta^2 \gamma+\gamma^2 \beta+\beta^3) \rc, \label{E402a}
\end{align}
where we have used \eqref{eq:esa11} and \eqref{eq:esa2} and the derivatives are evaluated at $s=t=0$. Combined with Lemma \ref{lem:esa1} (i) (iii) and Proposition \ref{prop:2nd} (ii) (iii), we have
\begin{align}
&[\phi(1,1) - \phi_s-\frac{1}{2}\phi_{tt}]_{k-2} \notag\\
\le & C \lc \delta^2\la \zeta'\ra_k+(\beta^2+\gamma^2) [\na q']_{k-2}+(1+|z|)(\gamma^3+\beta^2 \gamma+\gamma^2 \beta+\beta^3) \rc \notag \\
&+ C(\beta+\delta+\gamma)(\la \zeta' \ra_k+[\na q']_{k-2})+C(\gamma^2+\gamma \beta), \label{E402bb}
\end{align}
where we have used $\phi_t=0$ from \eqref{EX401b}.

From Proposition \ref{prop:2nd}, $\phi_{tt}=\phi^0_{tt} \in \isd_N \subset \isd$. Moreover, we obtain from Lemma \ref{lem:ephi2} that
\begin{align}
\phi_s=\frac{1}{2}\LL(\zeta)+\dbf^* \dbf \zeta+\na^2 Q=\frac{1}{2}\LL(\zeta')+\dbf^* \dbf h+\na^2 Q.
\label{E402c}
\end{align}

Therefore, we obtain from \eqref{E402bb} and \eqref{E402c} that
\begin{align}
&\|\LL(\zeta'/2)+\na^2 Q+\phi_{tt}/2 \|_{H^{k-2}}^2 \notag \\
\le& C\|\phi(1,1)\|^2_{H^{k-2}}+C\int [\dbf^*\dbf h]^2_{k-2}+(1+|z|)^2( \gamma^3+\beta^2 \gamma+\gamma^2 \beta+\beta^3)^2+(\gamma^2+\gamma \beta)^2 \, \vf \notag \\
&+C\int \delta^2 (\la \zeta' \ra^2_k+[\na q' ]_{k-2}^2)+\gamma^2 \la \zeta'\ra^2_k+(\gamma+\gamma^2)^2[\na q']^2_{k-2} \, \vf \notag \\
\le &C\|\phi(1,1)\|_{H^{k-2}}^2+C \|\dbf h\|_{H^{k-1}}^2+C ( \alpha^2+\beta^3+\alpha\beta+\alpha^3+\alpha^2\beta+\alpha \beta^2)^2 \notag \\
&+C \delta^2 \lc \|\zeta'\|_{H^{k}}^2+\|\na q'\|_{H^{k-2}}^2 \rc+C\alpha^2 \lc \|\rho^3\zeta'\|_{H^{k}}^2+\|\rho^4\na q'\|_{H^{k-2}}^2 \rc \notag\\
\le &C\|\phi(1,1)\|_{H^{k-2}}^2+C \|\dbf h\|_{H^{k-1}}^2+C ( \alpha^2+\beta^3+\alpha\beta)^2 \notag\\
&+C \delta^2 \lc \|\zeta'\|_{H^{k}}^2+\|\na q'\|_{H^{k-2}}^2 \rc +C\alpha^2 \lc \|\zeta'\|^{2-2\ep}_{H^{k}}+\|\na q'\|^{2-2\ep}_{H^{k-2}}\rc=: \mathfrak X^2,
\label{E402d}
\end{align}
where we have used Lemma \ref{lem:esf} and \cite[Lemma 9.9]{CM21b}. Since $\phi_{tt}$ is perpendicular to $\LL(\zeta')$ and $\na^2 Q$, it is clear from \eqref{E402d} that
\begin{align}
\|\phi_{tt} \|_{H^{k-2}} \le C\|\phi_{tt} \|_{L^2} \le C \mathfrak X,
\label{E402e}
\end{align}
where the first inequality holds since $\isd$ is finite-dimensional. From Proposition \ref{P301}, there exists a $1$-form $w$ satisfying $\PP w=\dbf h=\dbf \zeta'$ and 
\begin{align}
\|w\|_{H^{k+1}} \le C \|\dbf h\|_{H^{k-1}}.
\label{E402f}
\end{align}
Moreover, it follows from Lemma \ref{lem:pre} that $\LL(\zeta'-\dbf^*w) \in \ker(\dbf)$. Since $\na^2 Q \in \ker(\dbf)^{\perp}$, we conclude from Lemma \ref{lem:hperp} that
\begin{align*}
\|\LL(\zeta') \|_{H^{k-2}} \le& \|\LL(\zeta'-\dbf^*w) \|_{H^{k-2}}+\|\LL(\dbf^*w)\|_{H^{k-2}} \\
\le& C\|\LL(\zeta')+2\na^2 Q \|_{H^{k-2}}+C \|\dbf h\|_{H^{k-1}} \le C\mathfrak X,
\end{align*}
where we have used \eqref{E402d}, \eqref{E402e} and \eqref{E402f}. From our construction, $\pi_{\isd}(\zeta')=0$. By Theorem \ref{thm:est1} we obtain
\begin{align} \label{E402fa}
\|\zeta'\|_{H^{k}} \le& C\lc \|\LL(\zeta') \|_{H^{k-2}}+\|\dbf h\|_{L^2} \rc \le C\mathfrak X
\end{align}
since $\dbf \zeta'=\dbf h$. Similarly, we estimate
\begin{align}
\|\na^2 Q \|_{H^{k-2}} \le & \|\LL(\zeta'/2)+\na^2 Q \|_{H^{k-2}}+\|\LL(\zeta'/2) \|_{H^{k-2}} \le C\mathfrak X. \label{E402ga}
\end{align}

Therefore, it follows from \eqref{eq:poin} that
\begin{align} \label{E402gb}
\|\na Q \|_{L^2} \le C\| \na^2 Q\|_{L^2}+C|\mathcal B(h,\chi)| \le C(|\mathcal B(h,\chi)|+\mathfrak X).
\end{align}
Combining \eqref{E402fa}, \eqref{E402ga}, \eqref{E402gb} and the definition of $Q$, we obtain
\begin{align} \label{E403gc}
\|\na q'\|_{H^{k-1}} \le C\lc \|\zeta'\|_{H^{k}} +|\mathcal B(h,\chi)|+\mathfrak X \rc \le C\lc |\mathcal B(h,\chi)|+\mathfrak X \rc.
\end{align}

Consequently, it follows from \eqref{E402fa} and \eqref{E403gc} that
\begin{align} \label{E403gd}
\|\zeta'\|_{H^{k}}+\|\na q'\|_{H^{k-1}} \le C\lc |\mathcal B(h,\chi)|+\mathfrak X \rc.
\end{align}

Recall that
\begin{align}
\mathfrak X \le & C \lc \|\phi(1,1)\|_{H^{k-2}}+ \|\dbf h\|_{H^{k-1}} \rc+C (\alpha^2+\beta^3+\alpha \beta)+C\delta \lc \|\zeta'\|_{H^{k}}+\|\na q'\|_{H^{k-2}} \rc \notag \\
&+C\alpha \lc \|\zeta'\|^{1-\ep}_{H^{k}}+\|\na q'\|^{1-\ep}_{H^{k-2}}\rc. \label{E403b}
\end{align}

For any $\eta>0$, Young's inequality gives
\begin{align} 
&\alpha \lc \|\zeta'\|^{1-\ep}_{H^{k}}+\|\na q'\|^{1-\ep}_{H^{k-2}} \rc \le \eta \lc \|\zeta'\|_{H^{k}}+\|\na q'\|_{H^{k-2}} \rc+C_{\eta} \alpha^{\frac 1 \ep}. \label{E403bc}
\end{align}

Choosing $\eta$ sufficiently small and combining \eqref{E403gd}, \eqref{E403b} and \eqref{E403bc}, we obtain
\begin{align}
&\|\zeta'\|_{H^{k}}+\|\na q'\|_{H^{k-1}} \notag \\
\le & C \lc \|\phi(1,1)\|_{H^{k-2}}+ \|\dbf h\|_{H^{k-1}}+|\mathcal B(h,\chi)| \rc+C (\alpha^2+\beta^3+\alpha \beta+\alpha^{\frac 1 \ep}) \notag \\
&+(C\delta+1/2) \lc \|\zeta'\|_{H^{k}}+\|\na q'\|_{H^{k-2}} \rc. \label{E403Xa}
\end{align}

Since $\alpha^{1/\ep}\le\alpha^2$ for small $\delta$, absorbing the last line of \eqref{E403Xa} proves \eqref{E402}.
\end{proof}

Now, we prove the main result of the section, which is the rigidity inequality of mixed orders for generalized cylinders.

\begin{thm} \label{thm:sta1}
There is a constant $C>0$ such that, for every $\ep\in(0,1/2)$, there are positive constants $\delta_{\ep}$ and $C_{\ep}$ with the following property. If $\|h\|_{C^2}+\|\chi\|_{C^2}\le\delta_{\ep}$, then
\begin{align*}
\alpha^2+\beta^3 \le& C\lc \|(1+|z|)^2\phi(1,1)\|_{L^1}+\beta (\|\phi(1,1)\|_{L^2}+\|\dbf h\|_{H^{1}}+|\mathcal B(h,\chi)|)\rc \notag \\
&+C_{\ep} \lc \|\phi(1,1)\|^{2-\ep}_{L^2}+\|\dbf h\|^{2-\ep}_{H^{1}}+|\mathcal B(h,\chi)|^{2-\ep} \rc.
\end{align*}
\end{thm} 
\begin{proof}
Fix $\ep\in(0,1/2)$. We apply the Taylor expansion for $\phi$ and obtain
\begin{align*}
& \left|\phi(1,1) -\phi_s-\phi_t- \frac{1}{2}(\phi_{ss}+2\phi_{st}+\phi_{tt})-\frac{1}{6}(\phi_{ttt}+3\phi_{stt}+3\phi_{sst}+\phi_{sss}) \right| \notag \\
\le & C\delta(1+|z|) \lc \delta [\zeta']^2_{2}+(\gamma^2+\beta^2)([\zeta']_{2}+|\na q'|)+\gamma^3+\beta^2 \gamma+\gamma^2 \beta \rc+C(1+|z|) \beta^4, 
\end{align*}
where we have used \eqref{eq:esa3} and the derivatives are evaluated at $s=t=0$. Combined with \eqref{eq:esa1}, it follows that
\begin{align}
& \left| \phi(1,1) -\phi_s-\phi_t- \frac{1}{2}(\phi_{ss}+2\phi_{st}+\phi_{tt})-\frac{1}{6}\phi_{ttt}\right| \notag \\
\le & C(1+|z|) \lc \delta [\zeta']^2_{2}+(\gamma^2+\beta^2)([\zeta']_{2}+|\na q'|)+\gamma^3+\beta^2 \gamma+\gamma^2 \beta+ \beta^4 \rc. \label{E404b}
\end{align}

From Lemma \ref{lem:esa1} (i) (ii), one can further derive from \eqref{E404b} that
\begin{align}
& \left| \phi(1,1) -\phi_s- \frac{1}{2}(\phi^0_{ss}+2\phi^0_{st}+\phi_{tt})-\frac{1}{6}\phi_{ttt}\right| \notag \\
\le & C \lc (1+|z|)(\gamma^2+\beta^2)+ \beta+\gamma \rc ( \la \zeta' \ra_2+|\na q'|)+C( \la \zeta' \ra^2_2+|\na q'|^2)+C\delta(1+|z|)[\zeta']^2_{2} \notag \\
&+C(1+|z|) \lc\gamma^3+\beta^2 \gamma+\gamma^2 \beta+ \beta^4 \rc. \label{E404c}
\end{align}

Recall from \eqref{E402c} and Proposition \ref{prop:2nd} (ii) (iii) that $\phi_s \perp \isd$, $\phi^0_{st} \perp \isd$ and $\phi^0_{ss} \perp \isd_N$. Along the compact metric path $g_N+t h_1+t^2h_2/2$, replacing the minimizing potential by its second-order Taylor polynomial changes the third variation only by a Hessian, which is orthogonal to $\isd_N$. Thus
\[
\pi_{\isd_N}(\phi_{tt}+\phi_{ttt}/3)
=\pi_{\isd_N}\bigl(M_2(h_1)+M_3(h_1,h_2)/3\bigr).
\]
The order-$3$ obstruction and \eqref{E205a} therefore give
\begin{align} \label{E404d}
\|\pi_{\isd_N}(\phi_{tt}+\phi_{ttt}/3)\|_{L^2} \ge c_0 \|h_1\|^3_{L^2}=c_0 \beta^3
\end{align}
for some constant $c_0>0$. When $\beta>0$, define $\xi:=\pi_{\isd_N}(\phi_{tt}+\phi_{ttt}/3)/\|\pi_{\isd_N}(\phi_{tt}+\phi_{ttt}/3)\|_{L^2}$. Since $\isd_N$ is finite-dimensional, $\|\xi\|_{C^0}$ is bounded. By taking the inner product of the expression on the left side of \eqref{E404c} with $\xi$ and integrating, we obtain from \eqref{E404d} that
\begin{align*} 
\beta^3 \le& C \lc \|\phi(1,1)\|_{L^1}+\|\lc (1+|z|)^5 \alpha^2+(1+|z|)\beta^2+\beta+\alpha(1+|z|)^2\rc (\la \zeta' \ra_2+|\na q'|)\|_{L^1} \rc \notag \\
&+C \|\la \zeta' \ra^2_2+|\na q'|^2+\delta(1+|z|)[\zeta']^2_{2} \|_{L^1}+C \lc \alpha^3+\beta^2 \alpha+\alpha^2 \beta+ \beta^4 \rc \notag \\
\le& C \lc \|\phi(1,1)\|_{L^1}+(\alpha+\alpha^2+\beta) \|(\la \zeta' \ra_2+|\na q'|)\|_{L^2} \rc \notag\\
&+C \|(\la \zeta' \ra_2+|\na q'|)\|^2_{L^2}+C_{\ep}\delta \| [\zeta']_2 \|^{2-\ep}_{L^2}+C \lc \alpha^3+\beta^2 \alpha+\alpha^2 \beta+ \beta^4 \rc, 
\end{align*}
where we have used Cauchy–Schwarz inequality and \cite[Lemma 9.9]{CM21b} for the last inequality. It follows from Proposition \ref{prop:tay2} that
\begin{align} 
\beta^3 \le& C\|\phi(1,1)\|_{L^1}+C(\alpha+\alpha^2+\beta)\lc \|\phi(1,1)\|_{L^2}+\|\dbf h\|_{H^{1}}+|\mathcal B(h,\chi)|+ \alpha^2+ \alpha \beta+\beta^3 \rc \notag \\
&+C\lc \|\phi(1,1)\|^2_{L^2}+\|\dbf h\|^2_{H^{1}}+|\mathcal B(h,\chi)|^2+ \alpha^4+ \alpha^2 \beta^2+\beta^6 \rc \notag\\
&+C_{\ep}\delta \lc \|\phi(1,1)\|^{2-\ep}_{L^2}+\|\dbf h\|^{2-\ep}_{H^{1}}+|\mathcal B(h,\chi)|^{2-\ep}+ \alpha^{4-2\ep}+ \alpha^{2-\ep} \beta^{2-\ep}+\beta^{6-3\ep} \rc \notag\\
&+C \lc \alpha^3+\beta^2 \alpha+\alpha^2 \beta+ \beta^4 \rc, \label{eq:imp1}
\end{align}
where we have used Young's inequality.

On the other hand, if we set $v:=\Delta_{g_E}|\na u|^2-|\na u|^2$, then it follows from \cite[Lemma 6.4]{CM21b} that
\begin{align*}
\int \la \phi^0_{ss}, vg_N \ra\,\vf=-n\int v|\na u|^2 \,\vf \ge c \|u\|^4_{L^2}
\end{align*}
and $|v| \le C(1+|z|)^2 \alpha^2$. In addition, we claim that $\phi_{tt}+\phi_{ttt}/3 \perp \KK_1$ since $\phi_{tt}+\phi_{ttt}/3$ depends only on $N$ and the integral of $u$ vanishes. For $\alpha>0$, take the inner product of the expression on the left side of \eqref{E404c} with $vg_N/\alpha^2$ and integrate to obtain
\begin{align*} 
\alpha^2\le& C\|(1+|z|)^2\phi(1,1)\|_{L^1}\notag\\
&+C\left\|\bigl((1+|z|)^7\alpha^2+(1+|z|)^3\beta^2+\beta(1+|z|)^2
 +\alpha(1+|z|)^4\bigr)(\la\zeta'\ra_2+[\nabla q']_0)\right\|_{L^1}\notag\\
&+C \|(1+|z|)^2(\la \zeta' \ra^2_2+[\na q']^2_{0})+\delta(1+|z|)^3[\zeta']^2_{2} \|_{L^1}+C \lc \alpha^3+\beta^2 \alpha+\alpha^2 \beta+ \beta^4 \rc \notag \\
\le& C \lc \|(1+|z|)^2\phi(1,1)\|_{L^1}+(\alpha+\alpha^2+\beta) \|\la \zeta' \ra_2+[\na q']_0\|_{L^2} \rc \notag\\
&+C_{\ep} \|\la \zeta' \ra_2+[\na q']_0\|^{2-\ep}_{L^2}+C_{\ep}\delta \| [\zeta']_2 \|^{2-\ep}_{L^2}+C \lc \alpha^3+\beta^2 \alpha+\alpha^2 \beta+ \beta^4 \rc,
\end{align*}

As in \eqref{eq:imp1}, we conclude that
\begin{align} 
\alpha^2 \le& C\|(1+|z|)^2\phi(1,1)\|_{L^1}+C(\alpha+\alpha^2+\beta)\lc \|\phi(1,1)\|_{L^2}+\|\dbf h\|_{H^{1}}+|\mathcal B(h,\chi)|+ \alpha^2+ \alpha \beta+\beta^3 \rc \notag \\
&+C_{\ep}\lc \|\phi(1,1)\|^{2-\ep}_{L^2}+\|\dbf h\|^{2-\ep}_{H^{1}}+|\mathcal B(h,\chi)|^{2-\ep}+ \alpha^{4-2\ep}+ \alpha^{2-\ep} \beta^{2-\ep}+\beta^{6-3\ep} \rc \notag\\
&+C \lc \alpha^3+\beta^2 \alpha+\alpha^2 \beta+ \beta^4 \rc. \label{eq:imp2}
\end{align}

The cases $\beta=0$ in \eqref{eq:imp1} and $\alpha=0$ in \eqref{eq:imp2} are immediate. The error quantities are $O(\delta_{\ep})$, so their squares are bounded by their $(2-\ep)$-powers. Combining these estimates, Young's inequality bounds the error terms with coefficient $\alpha+\alpha^2$ by $\eta\alpha^2$ plus the $2-\ep$ powers in the conclusion. The terms with coefficient $\beta$ give the linear error term in the conclusion. Since $0<\ep<1/2$, all remaining terms in $\alpha$ and $\beta$ are $o(\alpha^2+\beta^3)$. Choosing $\eta$ and then $\delta_{\ep}$ sufficiently small completes the proof.
\end{proof}

We conclude this section by deriving the rigidity inequality for general Ricci shrinkers with bounded curvature and $\isd=0$.

\begin{thm} 
Suppose $( M, g , f)$ is a non-flat Ricci shrinker with bounded curvature and $\isd=0$. Then for any integer $k \ge 2$, there exist constants $ \delta$ and $C$ depending on $k$ and the Ricci shrinker such that for any $(h,\chi) \in C^k(S^2( M)) \times C^k( M)$ with $\|h\|_{C^k}+\|\chi\|_{C^k} \le \delta$,
\begin{align*}
\|h\|_{H^{k}}+\|\na \chi\|_{H^{k-1}} \le C \lc \|\bphi(\{ g+h, f+\chi\})\|_{H^{k-2}}+\|\df h\|_{H^{k-1}}\rc.
\end{align*}
\end{thm} 

\begin{proof}
We apply the linear splitting argument of Proposition \ref{prop:tay2} with $\isd=0$. Define $\delta:=\|h\|_{C^k}+\|\chi\|_{C^k}$ and
\begin{align*} 
g(t):= g+th,\quad f(t):= f+t \chi, \quad Q:=\frac{1}{2} \tr_{ g}(h)-\chi \quad \text{and} \quad \phi:=\bphi(\{g(t),f(t)\}).
\end{align*}
In this proof, set $\la h\ra_k:=[h]_k+\sqrt{1+f}\,[h]_{k-1}$. The Taylor expansion at $t=0$ gives
\begin{align}\label{E406b}
[\phi(1) -\phi_t]_{k-2} \le C\delta (\la h \ra_k+[\na \chi ]_{k-1}).
\end{align}
The first-variation formula in Lemma \ref{lem:ephi2} gives
\begin{align}
\phi_t=\frac{1}{2}\LL(h)+\df^* \df h+\na^2 Q.
\label{E406c}
\end{align}
Here $\KK_0=0$ by Proposition \ref{prop:iddd0} in the non-compact case, while in the compact case every parallel one-form is exact by the shrinker equation and hence vanishes. Lemma \ref{lem:p1} and the definition of $Q$ imply
\[
\|\na\chi\|_{H^{k-1}}
\le C\bigl(\|h\|_{H^{k}}+\|\na^2Q\|_{H^{k-2}}\bigr).
\]
Put $X:=\|h\|_{H^{k}}+\|\na^2Q\|_{H^{k-2}}$. Applying the linear splitting estimate from Proposition \ref{prop:tay2} to \eqref{E406c}, then using \eqref{E406b}, Lemma \ref{lem:esf} and the preceding bound, gives
\[
X\le C\bigl(\|\phi(1)\|_{H^{k-2}}+\|\df h\|_{H^{k-1}}\bigr)+C\delta X.
\]
Absorbing the last term for sufficiently small $\delta$ proves the claimed estimate.
\end{proof}

\section{Proof of the main theorem}

The case $m=n$ follows from compact rigidity in Section 2 and smooth
compactness. The compact quotient case is treated by the
equivariant argument below.  We henceforth assume
$d=m-n\ge1$ in the contraction and extension arguments.

To prove the rigidity of the Ricci shrinker, we apply the method of contraction and extension established in \cite{CIM15} and \cite{CM21b}.

Given a Ricci shrinker $(\bar M,\bar g,\bar f)=(N^n\times\R^{m-n},g_N\times g_E,|z|^2/4+n/2)$ satisfying \eqref{EX301}, set
\[
 \bar D_L=\{\bar b\le L\},\qquad \bar b=2\sqrt{\bar f}.
\]
Fix a point $p\in M$ with $f(p)\le n/2+1$, retained throughout the
iteration. Every comparison embedding below is smooth on a neighborhood
of its closed domain and is required to satisfy
\begin{equation}\label{eq:embedding-mark}
 p\in\Phi_L(\operatorname{int}\bar D_{10n})
 \quad\text{or}\quad
 p\in\Psi_L(\operatorname{int}\bar D_{10n}),
\end{equation}
respectively.
\begin{itemize}
\item[$(\star_L)$] There is a smooth embedding
$\Phi_L:\bar D_L\to M$ such that
\[
 [\bar g-\Phi_L^*g]_5+[\bar f-\Phi_L^*f]_5
 \le e^{\bar f/4-L^2/16}\quad\text{on }\bar D_L.
\]
\item[$(\dagger_L)$] There is a smooth embedding
$\Psi_L:\bar D_L\to M$ such that
\[
 |\bar g-\Psi_L^*g|+|\bar f-\Psi_L^*f|
 \le e^{-a_\dagger L^2}\quad\text{on }\bar D_L,
\]
and, for each $k\ge1$, the $C^k$-norms of these differences are bounded
by $D_kL^k$.
\end{itemize}
Here $a_\dagger>0$ is a fixed sufficiently small constant, chosen in the
proof of Theorem~\ref{thm:502}; the sequence $\{D_k\}$ depends only on
the model.

\begin{lem}\label{lem:weighted-gauge-iteration}
Let $(\bar M,\bar g,\bar f)=(N^n\times\mathbb{R}^d,
g_N\times g_E,n/2+|z|^2/4)$, and put $\bar b=2\sqrt{\bar f}$.
Fix $a>0$ and an integer $J$ such that $Ja>1/8$.
Suppose that $h^{(0)}$ and $\chi^{(0)}$ are smooth, supported in
$\{\bar b\le L\}$, and
\[
 \|h^{(0)}\|_{C^{J+5,1/2}}+\|\chi^{(0)}\|_{C^{J+5,1/2}}
 \le C_0 L^{A_0}e^{-aL^2}.
\]
For all sufficiently large $L$, there is a diffeomorphism $\vartheta$,
equal to the identity on $\{\bar b\ge L-1/2\}$ and preserving
$\{\bar b<L-1/2\}$, such that, if
\[
 h^{(J)}=\vartheta^*(\bar g+h^{(0)})-\bar g,
 \qquad
 \chi^{(J)}=\vartheta^*(\bar f+\chi^{(0)})-\bar f,
\]
then
\begin{align}
 \|h^{(J)}\|_{C^{J+5,1/2}}+\|\chi^{(J)}\|_{C^{J+5,1/2}}
 &\le C L^A e^{-aL^2},\label{eq:weighted-gauge-size}\\
 \|\delta_{\bar f}h^{(J)}\|_{H^{1}}
       +|\mathcal B(h^{(J)},\chi^{(J)})|
 &\le C L^A e^{-aL^2-(L-1)^2/8}.
 \label{eq:weighted-gauge-result}
\end{align}
The constants depend on the fixed cylinder and $a,J,C_0,A_0$.
If bounds of the initial form are available in every $C^{k,1/2}$ norm,
then the first estimate holds in every such norm, with constants
depending on $k$ and the corresponding initial bounds.
\end{lem}

\begin{proof}
Set
\[
 V=\int_{\bar M}\,\mathrm{d}V_{\bar f}
   =|N|e^{-n/2}(4\pi)^{d/2},\qquad \epsilon=e^{-aL^2}.
\]
Fix a cutoff $\rho=\rho_L$ which equals one on
$\{\bar b\le L-1\}$ and is supported in $\{\bar b<L-1/2\}$.
Its covariant derivatives of each fixed order are uniformly bounded
for large $L$.  The Gaussian tail estimate gives, for every fixed $q$,
\begin{equation}\label{eq:gauge-tail}
 \int_{\{\bar b\ge L-1\}}(1+\bar b)^q\,\mathrm{d}V_{\bar f}
 \le C_q L^{A_q}e^{-(L-1)^2/4}.
\end{equation}

Given $(h^{(j)},\chi^{(j)})$, set
$F_j=\delta_{\bar f}h^{(j)}$ and
$B_j=\mathcal B(h^{(j)},\chi^{(j)})$.
Proposition~\ref{P301} gives the unique one-form $u_j$, perpendicular
to $\mathcal K_{\mathcal P}$, satisfying
\[
 2\mathcal P u_j=F_j.
\]
Define
\begin{equation}\label{eq:gauge-fields}
 v_j=u_j+\frac{2}{V}\sum_{i=1}^d(B_j)_i\,\mathrm{d}z_i,
 \qquad w_j=\rho v_j.
\end{equation}
Translations are Killing fields, so $2\mathcal P v_j=F_j$.  Moreover,
since $u_j$ is perpendicular to the parallel forms,
\begin{equation}\label{eq:gauge-center-cancellation}
 \int_{\bar M}\langle v_j,\mathrm{d}z_i\rangle\,\mathrm{d}V_{\bar f}=2(B_j)_i.
\end{equation}
Let $\varphi_j$ be the time-one flow of $w_j^\sharp$ and set
\begin{equation}\label{eq:gauge-update}
 h^{(j+1)}=\varphi_j^*(\bar g+h^{(j)})-\bar g,
 \qquad
 \chi^{(j+1)}=\varphi_j^*(\bar f+\chi^{(j)})-\bar f.
\end{equation}
Every flow is the identity outside $\{\bar b<L-1/2\}$ and preserves
that set.  In particular, the tensors in~\eqref{eq:gauge-update}
remain supported in $\{\bar b\le L\}$.

Estimate~\eqref{ey302a} in the proof of Theorem~\ref{thm:diff2}, applied to
$2\mathcal P u_j=\delta_{\bar f}h^{(j)}$, gives, for every fixed $k\ge3$,
\begin{equation}\label{eq:gauge-local-linear}
 \|v_j\|_{C^{k+1,1/2}(\{\bar b\le L-1/2\})}
 \le C_kL^{A_k}
  \bigl(\|h^{(j)}\|_{C^{k,1/2}}+|B_j|\bigr).
\end{equation}
We use~\eqref{eq:gauge-local-linear} at $k=J+5$.
Also $|B_j|\le C\bigl(\|h^{(j)}\|_{C^1}
+\|\chi^{(j)}\|_{C^1}\bigr)$, by its integral definition.
The usual differential equations for the derivatives of a flow imply
\[
 \|h^{(j+1)}\|_{C^{J+5,1/2}}+
 \|\chi^{(j+1)}\|_{C^{J+5,1/2}}
 \le C L^A\bigl(\|h^{(j)}\|_{C^{J+5,1/2}}
                  +\|\chi^{(j)}\|_{C^{J+5,1/2}}\bigr),
\]
provided $L\|w_j\|_{C^{J+6,1/2}}$ is sufficiently small.
The factor polynomial in $L$ allows for $|\nabla\bar f|\le CL$
on the support of $w_j$.  Starting from the hypothesis and applying
this estimate successively, for the fixed number $J$ of steps, gives
\begin{equation}\label{eq:gauge-uniform-small}
 \|h^{(j)}\|_{C^{J+5,1/2}}+\|\chi^{(j)}\|_{C^{J+5,1/2}}
 +\|w_j\|_{C^{J+6,1/2}}
 +\|v_j\|_{C^{J+6,1/2}(\{\bar b\le L-1/2\})}
 \le C L^A\epsilon.
\end{equation}
Proposition~\ref{P301} and multiplication by $\rho$
give, for every integer $s\ge0$ under consideration,
\begin{equation}\label{eq:gauge-weighted-linear}
 \|v_j\|_{H^{s+2}}+\|w_j\|_{H^{s+2}}
 \le C_s\bigl(\|F_j\|_{H^{s}}+|B_j|\bigr).
\end{equation}
These constants are independent of $L$.

If $h,w$ satisfy the
bounds in~\eqref{eq:gauge-uniform-small}, with $w$ supported where
$\bar b<L-1/2$, and $\varphi_t$ is its flow, then
\begin{equation}\label{eq:gauge-metric-taylor}
 \varphi_1^*(\bar g+h)-\bar g
 =h-2\delta_{\bar f}^*w+L_{w^\sharp}h+R,
 \qquad
 \|R\|_{H^{s+1}}
 \le C_s(1+\|h\|_{C^{s+3}})
       \|w\|_{C^{s+3}}\|w\|_{H^{s+3}}.
\end{equation}
Taylor's formula expresses
$R=\int_0^1(1-t)\varphi_t^*
\bigl(L_{w^\sharp}^2(\bar g+h)\bigr)\,\mathrm{d}t$.
The product estimates give the stated bound before pullback. Moreover,
\[
 |\bar f\circ\varphi_t-\bar f|\le CL\|w\|_{C^0},
 \qquad |\log\operatorname{Jac}\varphi_t|\le C\|w\|_{C^1}.
\]
These estimates and the flow equations bound pullback on $H^{s+1}$
uniformly in $L$ and $0\le t\le1$, proving~\eqref{eq:gauge-metric-taylor}.
The same product argument gives
\[
 \|L_{w^\sharp}h\|_{H^{s+1}}
 \le C_s\|h\|_{C^{s+2}}\|w\|_{H^{s+2}}.
\]

Apply $\delta_{\bar f}$ to~\eqref{eq:gauge-metric-taylor}.
Using $2\mathcal Pv_j=F_j$, we obtain the exact decomposition
\begin{equation}\label{eq:gauge-divergence-decomposition}
 F_{j+1}=(1-\rho)F_j-2[\mathcal P,\rho]v_j
       +\delta_{\bar f}\bigl(L_{w_j^\sharp}h^{(j)}+R_j\bigr).
\end{equation}
The first two terms are supported in $\{L-1\le\bar b\le L\}$.
Their derivatives through any of the fixed orders in use are bounded
by $CL^A\epsilon$, by~\eqref{eq:gauge-uniform-small} and the
formula for $\mathcal P$.  Therefore~\eqref{eq:gauge-tail} gives
\begin{equation}\label{eq:gauge-linear-tail}
 \|(1-\rho)F_j-2[\mathcal P,\rho]v_j\|_{H^{s}}
 \le C_sL^A\epsilon e^{-(L-1)^2/8}.
\end{equation}
On tensors supported in $\{\bar b\le L\}$, the map
$\delta_{\bar f}:H^{s+1}\longrightarrow H^{s}$ has norm at most
$C_sL$.  Equations~\eqref{eq:gauge-metric-taylor},
\eqref{eq:gauge-uniform-small}, and
\eqref{eq:gauge-weighted-linear} consequently imply
\begin{equation}\label{eq:gauge-nonlinear}
 \|\delta_{\bar f}(L_{w_j^\sharp}h^{(j)}+R_j)\|_{H^{s}}
 \le C_sL^A\epsilon
       \bigl(\|F_j\|_{H^{s+1}}+|B_j|\bigr).
\end{equation}

We next estimate the center in the weighted norm.  Put
$Q=\operatorname{tr}_{\bar g}h/2-\chi$.
The first variation along $w$ is
\[
 Q'(0)=\delta_{\bar f}w
       +\langle\operatorname{sym}\nabla w,h\rangle
       +\langle w,\nabla Q\rangle.
\]
Since
$\mathcal B_i(h,\chi)=\frac12\int z_iQ\,\mathrm{d}V_{\bar f}$,
integration by parts gives
\[
 \frac12\int z_i\delta_{\bar f}w\,\mathrm{d}V_{\bar f}
     =-\frac12\int\langle \mathrm{d}z_i,w\rangle\,\mathrm{d}V_{\bar f}.
\]
The last two terms in $Q'(0)$ have combined $L^2$ norm at most
$C(\|h\|_{C^1}+\|\chi\|_{C^1})\|w\|_{H^{1}}$.
The second variation, estimated by the same weighted Taylor argument
as above, has $L^2$ norm at most
\[
 CL(1+\|h\|_{C^2}+\|\chi\|_{C^2})
          \|w\|_{C^2}\|w\|_{H^{2}}.
\]
Since each $z_i$ belongs to $L^2(\mathrm{d}V_{\bar f})$, we obtain
\begin{align*}
 |B_{j+1}|
 &\le \frac12\left|\int (1-\rho)
                   (\langle v_j,\mathrm{d}z_i\rangle)_{i=1}^d\,\mathrm{d}V_{\bar f}\right|
       +CL^A\epsilon\|w_j\|_{H^{2}}\\
 &\le CL^A\epsilon e^{-(L-1)^2/8}
       +CL^A\epsilon\bigl(\|F_j\|_{L^2}+|B_j|\bigr).
\end{align*}
We used~\eqref{eq:gauge-center-cancellation}; the omitted part of
$v_j$ is estimated by Cauchy--Schwarz,
\[
 \left|\int(1-\rho)\langle v_j,\mathrm{d}z_i\rangle\,\mathrm{d}V_{\bar f}\right|
 \le \|v_j\|_{L^2}
      \left(\int_{\{\bar b\ge L-1\}}\,\mathrm{d}V_{\bar f}\right)^{1/2},
\]
and $\|v_j\|_{L^2}\le CL^A\epsilon$ follows from
\eqref{eq:gauge-weighted-linear} and the unweighted bounds.

Finally, define the descending orders and errors
\[
 s_j=J+1-j,\qquad
 S_j=\|F_j\|_{H^{s_j}}+|B_j|,\qquad 0\le j\le J.
\]
Combining~\eqref{eq:gauge-divergence-decomposition}--
\eqref{eq:gauge-nonlinear} with the center estimate gives
\begin{equation}\label{eq:gauge-recurrence}
 S_{j+1}\le CL^A\epsilon S_j
             +CL^A\epsilon e^{-(L-1)^2/8},
 \qquad S_0\le CL^A\epsilon.
\end{equation}
Iterating $J$ times and increasing the fixed polynomial factor gives
\[
 S_J\le CL^A\bigl(\epsilon^{J+1}
                    +\epsilon e^{-(L-1)^2/8}\bigr).
\]
The hypothesis $Ja>1/8$ absorbs the first term into the second for
large $L$.  Since $s_J=1$, this proves
\eqref{eq:weighted-gauge-result}.  The composition
$\vartheta=\varphi_0\circ\varphi_1\circ\cdots\circ\varphi_{J-1}$
has all the stated support and preservation properties.  This proves
the lemma.
\end{proof}

\begin{thm}\label{thm:501}
There exists $L_1>0$ such that if $(\dagger_L)$ holds and $L>L_1$, then $(\star_{L-2})$ holds.
\end{thm}
\begin{proof}
Let $a_\dagger>0$ be the fixed exponent in $(\dagger_L)$, so that its zeroth-order error is at most $e^{-a_\dagger L^2}$. All constants below may depend on $a_\dagger$, the fixed cylinder, and the sequence $\{D_k\}$. Fix
\[
 0<a<\min\{a_\dagger,1/64\},
\]
and choose an integer $J$ with $Ja>1/8$.

Let $\Psi_L:\bar D_L\to M$ be the embedding in
$(\dagger_L)$.  Choose a smooth cutoff $\eta$ which equals one on
$\bar D_{L-1/2}$ and has compact support in
$\operatorname{int}\bar D_L$. Set
\[
 h^{(0)}=\eta(\Psi_L^*g-\bar g),\qquad
 \chi^{(0)}=\eta(\Psi_L^*f-\bar f).
\]
Interpolation on unit balls, using the bounds $D_kL^k$ at sufficiently high but fixed orders, gives for each fixed $k$
\begin{equation}\label{E501a}
 \|h^{(0)}\|_{C^{k,1/2}}+\|\chi^{(0)}\|_{C^{k,1/2}}
 \le C_kL^{A_k}e^{-aL^2}.
\end{equation}

Apply Lemma~\ref{lem:weighted-gauge-iteration}. Write $h_0=h^{(J)}$, $\chi_0=\chi^{(J)}$, and $\Phi=\Psi_L\circ\vartheta$. Since $\vartheta$ is the identity where $\eta\ne1$, we have
\[
 h_0=\eta(\Phi^*g-\bar g),\qquad
 \chi_0=\eta(\Phi^*f-\bar f).
\]
Thus, for each fixed $k$,
\begin{align}
 \|h_0\|_{C^{k,1/2}}+\|\chi_0\|_{C^{k,1/2}}
 &\le C_kL^{A_k}e^{-aL^2},\label{E503a}\\
 \|\delta_{\bar f}h_0\|_{H^{1}}
       +|\mathcal B(h_0,\chi_0)|
 &\le CL^Ae^{-aL^2-(L-1)^2/8}.
 \label{E503b}
\end{align}
The pair $(\bar g+h_0,\bar f+\chi_0)$ is an exact normalized Ricci shrinker on $\{\bar b<L-1/2\}$.

Decompose, as in Section~4,
\begin{equation}\label{E503c}
 h_0=ug_N+h_1+\frac12h_2+\zeta',\qquad
 \chi_0=\frac n2u+f_1+\frac12f_2+q',
\end{equation}
where $\pi_{\isd}(h_0)=ug_N+h_1$, $h_1\in\isd_N$, and $h_2,f_1,f_2$ are the induced terms. Set
\[
 \alpha=\|u\|_{L^2},\qquad \beta=\|h_1\|_{L^2},\qquad
 \phi=\bphi(\{\bar g+h_0,\bar f+\chi_0\}).
\]
The support of $\phi$ is contained in $\{\bar b\ge L-1/2\}$. The local formula for $\bphi$, \eqref{E503a}, and the Gaussian tail bound yield
\begin{align}
 \|(1+|z|)^2\phi\|_{L^1}
 &\le CL^Ae^{-aL^2-(L-1)^2/4},\notag\\
 \|\phi\|_{L^2}
 &\le CL^Ae^{-aL^2-(L-1)^2/8}.
 \label{E504b}
\end{align}
Increasing $L$ absorbs the polynomial factors and the terms linear in $L$ in the exponents. Consequently, if
\[
 D=\|\phi\|_{L^2}+\|\delta_{\bar f}h_0\|_{H^{1}}
                         +|\mathcal B(h_0,\chi_0)|,
\]
then
\begin{equation}\label{E504c}
 D\le Ce^{-(1/8+a/2)L^2},\qquad
 \|(1+|z|)^2\phi\|_{L^1}\le Ce^{-(1/4+a/2)L^2}.
\end{equation}
Choose $\ep>0$ so small that
$(2-\ep)(1/8+a/2)\ge1/4+a/2$.
Theorem~\ref{thm:sta1} now gives
\begin{equation}\label{E504a}
 \alpha^2+\beta^3
 \le C\left(e^{-(1/4+a/2)L^2}
                    +\beta e^{-(1/8+a/2)L^2}\right).
\end{equation}
Young's inequality absorbs half of $\beta^3$ from the last term. Since $a<1/64$, the first exponential is no larger than $e^{-(3/16+3a/4)L^2}$, and hence
\begin{equation}\label{E504d}
 \alpha\le Ce^{-(3/32+3a/8)L^2},\qquad
 \beta\le Ce^{-(1/16+a/4)L^2}.
\end{equation}
Proposition~\ref{prop:tay2}, together with \eqref{E504c} and \eqref{E504d}, implies
\begin{align}
 \|\zeta'\|_{H^{2}}+\|\nabla q'\|_{H^{1}}
 &\le C\left(D+\alpha^2+\alpha\beta+\beta^3\right)\le Ce^{-(1/8+a/2)L^2}.
 \label{E504e}
\end{align}

If $\bar b(x)\le L-1$, then on $B(x,2L^{-1})$ the weight differs from $e^{-\bar f(x)}$ by a bounded factor. Thus the unweighted local $L^2$ norms of $\zeta'$ and $\nabla q'$ are at most
\[
 Ce^{\bar f(x)/2-(1/8+a/2)L^2}.
\]
For every fixed higher order, their local $C^k$ norms are bounded by $C_kL^{A_k}e^{-aL^2}$, using \eqref{E503a}, \eqref{E503c}, \eqref{E504d}, the polynomial form of $u$, and the finite-dimensionality of $\isd_N$. Interpolation at scale $L^{-1}$, with exponent $\theta$ on the local
$L^2$ term satisfying $1-\theta<a/2$, and absorption of the polynomial
factors give
\begin{equation}\label{E504g}
 [\zeta']_2(x)+[\nabla q']_1(x)
 \le Ce^{\bar f(x)/2-(1/8+3a/8)L^2}
 \qquad (\bar b(x)\le L-1).
\end{equation}

Subtracting the normalization identities for the two shrinkers and
using \eqref{E503c} and \eqref{E504g} gives
\begin{equation}\label{E504h}
 |q'(x)|\le CL^A\left(\alpha+\beta+
                  e^{\bar f(x)/2-(1/8+3a/8)L^2}\right).
\end{equation}
Equations~\eqref{E503c}, \eqref{E504d}, \eqref{E504g}, and \eqref{E504h} consequently give, after absorbing polynomial factors,
\begin{align}
 |h_0(x)|+|\chi_0(x)|
 &\le C\left(e^{-(1/16+a/8)L^2}
                +e^{\bar f(x)/2-(1/8+a/4)L^2}\right)\notag\\
 &\le Ce^{\bar f(x)/4-(1/16+a/8)L^2},
 \qquad \bar b(x)\le L-1.
 \label{E504i}
\end{align}
Finally interpolate \eqref{E504i} with the higher derivative bounds
\eqref{E503a} on balls of radius comparable to $L^{-1}$, taking the
exponent $\theta$ on \eqref{E504i} with $1-\theta<a$. Absorbing the
constants and scale factors gives
\begin{equation}\label{E505b}
 [h_0]_5+[\chi_0]_5\le e^{\bar f/4-L^2/16}
 \quad\hbox{on }\{\bar b\le L-3/2\}.
\end{equation}
Restrict $\Phi=\Psi_L\circ\vartheta$ to the exact model domain
$\bar D_{L-2}$.  Estimate \eqref{E505b} is stronger than the estimate
required for $(\star_{L-2})$ there.  To check the marking, let
$q=\Psi_L^{-1}(p)$.  The potential estimate in $(\dagger_L)$ gives
$\bar b(q)\le2\sqrt{n/2+2}<10n-1$ for large $L$.  The displacement
of $\vartheta^{-1}$ is at most $CL^Ae^{-aL^2}<1/2$ by the gauge
construction.  Since $\bar b$ is $1$-Lipschitz,
$\vartheta^{-1}(q)\in\operatorname{int}\bar D_{10n}$.
This proves $(\star_{L-2})$ with the same marked point.
\end{proof}

\begin{thm}\label{thm:502}
There exist positive constants $a_0$ and $\theta$ such that, for each
fixed $a_\dagger\in(0,a_0]$, there is $L_2=L_2(a_\dagger)>0$
with the following property: if $(\star_L)$ holds with $L\ge L_2$,
then $(\dagger_{(1+\theta)L})$ holds with exponent $a_\dagger$.
\end{thm}

\begin{proof}
Write $d=m-n$.  The constants in this proof depend only on the model
shrinker and on the fixed cutoff parameters chosen below.  Powers of
$L$ are denoted by $L^A$; the exponent can increase at different
occurrences and with the number of derivatives under consideration.

\textbf{Step 1: curvature bounds and radial identification.}
Work first inside the compact image
$\Phi_L(\bar D_{L-1/2})$.  On this domain the errors in
$(\star_L)$ are at most $Ce^{-cL}$.  The equation
$b\circ\Phi_L=L-1$ has a unique radial graph near
$\{\bar b=L-1\}$: the radial derivative of $\Phi_L^*f$ is
$\bar b/2+O(e^{-cL})$ on a fixed-width neighborhood of that level.
Denote its image by $\Sigma$ and the compact domain bounded by
$\Sigma$ and containing $\Phi_L(N\times\{0\})$ by $U_0$.
Thus $\partial U_0=\Sigma$, $b=L-1$ on $\Sigma$, and
$p\in\operatorname{int}U_0$.

Let $g(t)=(1-t)(\psi^t)^*g$ be the self-similar Ricci flow
generated by $\nabla f/(1-t)$. The nearby radial graph $b=L-3/4$
bounds a compact region inside the given chart containing a fixed
neighborhood of $\Sigma$. Backward trajectories remain in this region,
since $-\nabla f$ points inward along its boundary. Thus the chart
estimates give uniformly controlled geometry on a fixed backward
time interval. Choose $r_0>0$ sufficiently small that balls of radius $r_0$
about these worldlines satisfy the local entropy hypothesis
of \cite[Theorem 24]{LW20}. Choose $\epsilon>0$ and set
$\tau=(\epsilon r_0)^2$ so that $-\tau$ lies in this interval and
$2\tau$ is smaller than the uniform pseudolocality lifetime. Applying that theorem
at time $-\tau$ gives
\[
 |\operatorname{Rm}_{g(t)}|(x)
 \le(t+\tau)^{-1}\le\tau^{-1}=(\epsilon r_0)^{-2},
 \qquad 0\le t\le(\epsilon r_0)^2,\qquad x\in\Sigma.
\]
The theorem and metric distortion give the same uniform control
on fixed smaller parabolic neighborhoods of these worldlines.
Along these trajectories,
\[
 \frac{\mathrm{d}}{\mathrm{d}t}f(\psi^t(x))
 =\frac{f(\psi^t(x))-\mathrm{R}(\psi^t(x))}{1-t}.
\]
The scalar curvature is uniformly bounded, so
\begin{equation}\label{E506d}
 \frac{(L-1)^2/4-Ct}{1-t}
 \le f(\psi^t(x))
 \le\frac{(L-1)^2}{4(1-t)}.
\end{equation}
Fix a time strictly smaller than $(\epsilon r_0)^2$.
For a fixed sufficiently small $\theta_1>0$ and large $L$, every
trajectory from $\Sigma$ crosses $b=(1+\theta_1)L$ before that
time.  On this collar $|\nabla b|^2=1-4\mathrm{R}/b^2$ is bounded below,
so the hitting times are smooth.  Strict increase of $b$ and
uniqueness of the gradient-flow trajectories show that the collar is
embedded and meets $U_0$ only at its initial boundary.

The union of $U_0$ and this collar is a compact regular domain
$U_{\mathrm{out}}$, with boundary on $b=(1+\theta_1)L$.
The backward and forward bounds supply a time buffer at the initial
level. After shrinking the spatial neighborhoods and $\theta_1$,
local Shi estimates give
\begin{equation}\label{E506a}
 |\nabla^k\operatorname{Rm}|\le C_k
 \quad\text{on }U_{\mathrm{out}},\qquad k\ge0.
\end{equation}
The same backward-flow argument gives these bounds on the inner
domain.

Let $V=\nabla b/|\nabla b|^2$ and
$\bar V=\nabla\bar b/|\nabla\bar b|^2$.
Starting with $\Phi_L$, identify $\{\bar b=L-1\}$ with $\Sigma$
by following the $V$-flow for time $L-1-b(\Phi_L(x))$.
Extend this level map radially by the $\bar V$-flow on the source
and the $V$-flow on the target, and denote the resulting map by $F$.
On $L-2\le\bar b\le L-1$, both $F$ and $\Phi_L$ are
$C^4$-close, with error $Ce^{-cL}$.  Interpolate there by setting
\[
 Y(x)=\exp_x^{\bar g,-1}\bigl(\Phi_L^{-1}(F(x))\bigr),
 \qquad
 \Psi_1(x)=\Phi_L\bigl(\exp_x^{\bar g}(\chi(\bar b(x))Y(x))\bigr),
\]
where $\chi=0$ for $\bar b\le L-2$ and $\chi=1$ for
$\bar b\ge L-3/2$.  Use $\Phi_L$ on the inside and the radial
map on the outside.  The interpolation is $C^1$-close to the
identity after composition with $\Phi_L^{-1}$ and is radial near
its outer edge.  It therefore gives a smooth embedding, agreeing with
$\Phi_L$ on $\bar D_{L-2}$.

For any sufficiently small fixed $\theta_2\le\theta_1$, this gives
an embedding
\[
 \Psi_1:\bar D_{(1+\theta_2)L}\longrightarrow U_{\mathrm{out}}.
\]
We have
\begin{align}
 [\Psi_1^*g-\bar g]_4+[\Psi_1^*f-\bar f]_4
 &\le e^{\bar f/4-L^2/16}
 &&\text{on }\bar D_{L-2},\label{EA512b}\\
 [\Psi_1^*g-\bar g]_4+[\Psi_1^*f-\bar f]_4
 &\le Ce^{-cL}
 &&\text{on }\{L-2\le\bar b\le L-1\},\label{EA512ba}\\
 [\Psi_1^*g-\bar g]_2&\le C\theta_2+Ce^{-cL},
 \qquad \Psi_1^*f=\bar f
 &&\text{on }\{L-1\le\bar b\le(1+\theta_2)L\}.
 \label{EA512bb}
\end{align}
For the last metric estimate, write the radial metric as
$a\,\mathrm{d}r^2+\gamma_r$, where $r=b$.  The shrinker equation gives
$\partial_r\gamma_r=(2a/r)(\gamma_r-2\mathrm{Ric}|_{T\Sigma_r})$.
Differentiating this radial evolution through order two and using
\eqref{E506a} gives radial rates $O(L^{-1})$ in $C^2$.
Integration over an interval of length $O(\theta_2L)$, together with
$a=(1-4\mathrm{R}/r^2)^{-1}$, proves \eqref{EA512bb}.

\textbf{Step 2: the glued geometry.}
Set $\theta_3=\theta_2/2$.  Choose a cutoff $\rho$ which is one on
$(-\infty,(1+\theta_3)L]$, zero on
$[(1+\theta_2)L,\infty)$, and satisfies
$|\rho^{(k)}|\le C_k(\theta_2 L)^{-k}$.
Define on $\bar M$
\begin{equation}\label{EA515a}
 g'=\rho(\bar b)\Psi_1^*g+(1-\rho(\bar b))\bar g,
 \qquad
 f'=\rho(\bar b)\Psi_1^*f+(1-\rho(\bar b))\bar f.
\end{equation}
In Steps 2--5 we work on the complete glued manifold
$(\bar M,g',f')$ and set $b'=2\sqrt{f'}$.  Via $\Psi_1$, the pair
agrees with the original shrinker before the inner edge of the gluing
annulus; beyond its outer edge it equals the model.  Moreover,
$f'=\bar f$ throughout the outer region $\bar b\ge L-1$.

We verify the uniform bounds needed for global growth estimates.
On the gluing annulus write $r=\bar b=b'$.  Because the outer
identification is radial, the two metrics in \eqref{EA515a} have the
form
\[
 g_\ell=a_\ell\,\mathrm{d}r^2+\gamma_{\ell,r},\quad
 f_\ell=r^2/4,\quad \ell=0,1,
\]
with no mixed radial--tangential terms.  If $S_\ell=\mathrm{R}(g_\ell)$, then
\[
 a_\ell=(1-4S_\ell/r^2)^{-1},\qquad
 \partial_r\gamma_{\ell,r}
 =\frac{2a_\ell}{r}
   (\gamma_{\ell,r}-2\mathrm{Ric}(g_\ell)|_{T\Sigma_r}).
\]
Consequently, by \eqref{E506a},
\[
 |a_\ell-1|\le CL^{-2},\quad
 |\partial_ra_\ell|+|\mathrm{d}_\Sigma a_\ell|\le CL^{-2},\quad
 |\partial_r\gamma_{\ell,r}|\le CL^{-1}.
\]
The bound $|\gamma_{1,r}-\gamma_{0,r}|\le C\theta_2+Ce^{-cL}$
shows that their cutoff interpolation
$g'=a\,\mathrm{d}r^2+\gamma_r$ satisfies the same estimates, with constants
independent of $L$.  These estimates give $|\nabla^2f'|\le C$, and the $C^2$ metric bound
gives bounded curvature. In particular, for
$E=\mathrm{Ric}(g')+\nabla^2f'-g'/2$, we have
\begin{equation}\label{EA515axx}
 |\operatorname{Rm}(g')|+|E|\le C,
 \qquad
 |m/2-\Delta f'|+|f'-|\nabla f'|^2|\le C.
\end{equation}
For the last assertion, directly compute
\[
 f'-|\nabla f'|^2=\frac{r^2}{4}(1-a^{-1}),\qquad
 \Delta f'=\frac1{2a}
 +\frac r{4a}\operatorname{tr}_{\gamma_r}\partial_r\gamma_r
 -\frac r{4a^2}\partial_ra.
\]
Off the annulus these are shrinker identities with bounded scalar
curvature.  Proposition~\ref{prop:ap1} therefore applies globally
with $\theta=0$ and a fixed $\Lambda$.

\textbf{Step 3: almost-splitting eigenfunctions.}
Let $0<\mu_1\le\mu_2\le\cdots$ be the positive eigenvalues of
$-\Delta_{f'}$ and choose $L^2(\mathrm{d}V_{f'})$-orthonormal eigenfunctions
$v_1,\ldots,v_d$. The spectrum is discrete.
Cut off the constant function and the model coordinate functions
$z_1,\ldots,z_d$ at $b'=L$.  Their Gram matrix differs from the
model Gram matrix by $Ce^{-L^2/16}$, and the same is true for their
Dirichlet-energy matrix.  Indeed, on $\bar b\le L-2$ this follows
by integrating \eqref{EA512b} against the Gaussian weight; on the
remaining cutoff collar its weighted volume is
$O(L^A e^{-(L-2)^2/4})$.  The min--max principle gives
\begin{equation}\label{E507e}
 \sum_{i=1}^d\mu_i\le d/2+Ce^{-L^2/16}.
\end{equation}
In particular these eigenvalues are uniformly bounded.

Cutoff integration and the one-form commutation identity give
\[
 \|\mathrm{d}v_i\|_{L^2}^2=\mu_i,
 \qquad
 \Delta_{f'}\mathrm{d}v_i=(\tfrac12-\mu_i)\mathrm{d}v_i+E(\mathrm{d}v_i).
\]
Thus $\langle\Delta_{f'}\mathrm{d}v_i,\mathrm{d}v_i\rangle\ge-C|\mathrm{d}v_i|^2$ and
$\mathrm{d}v_i,\Delta_{f'}\mathrm{d}v_i\in L^2$. Unique continuation verifies the nonvanishing assumption of
Proposition~\ref{prop:ap1}. Proposition~\ref{prop:ap1} and
\eqref{EA515axx} give polynomial growth for $I_{\mathrm{d}v_i}$, with initial
value bounded by $C\mu_i$ using coarea on a fixed annulus.  Hence
\begin{equation}\label{EA516f}
 \int_{\{b'\ge(1+\theta_3)L\}}|\mathrm{d}v_i|^2\,\mathrm{d}V_{f'}
 \le CL^A e^{-(1+\theta_3)^2L^2/4}\mu_i.
\end{equation}
The integrated Bochner formula, justified first with cutoffs, yields
\[
 \|\nabla^2v_i\|_{L^2}^2
 =(\mu_i-\tfrac12)\mu_i-\int E(\mathrm{d}v_i,\mathrm{d}v_i)\,\mathrm{d}V_{f'}.
\]
The defect $E$ is supported in the gluing annulus.  Consequently
\eqref{EA516f} first gives
$\mu_i\ge1/2-CL^A e^{-(1+\theta_3)^2L^2/4}$; combining this with
\eqref{E507e} proves
\begin{equation}\label{E507a}
 |\mu_i-\tfrac12|+\|\nabla^2v_i\|_{L^2}^2
 \le Ce^{-L^2/16}.
\end{equation}

\textbf{Step 4: Hessian estimates inside the genuine shrinker.}
Proposition~\ref{prop:ap1} applied to $v_i$ gives polynomial growth
of $I_{v_i}$ on the whole glued manifold.  On the region where the
metric is the original shrinker, local elliptic estimates on balls of
radius $cL^{-1}$ therefore give
\begin{equation}\label{E508d}
 |\nabla^j v_i|\le C_jL^{A_j}
 \quad\text{on }\{b'\le(1+\theta_3)L-1\},
\end{equation}
for every fixed $j$.  This scale keeps the drift coefficients bounded
after rescaling, and \eqref{E506a} controls their higher derivatives.

The Hessian $H_i=\nabla^2v_i$ satisfies
\[
 \Delta_{f'}H_i=(1-\mu_i)H_i-2\operatorname{Rm}(H_i)
\]
on $\{b'<(1+\theta_3)L\}$, where $E=0$.
Choose a fixed regular radius $r_*$ larger than the threshold in
Lemma~\ref{lem:finite-frequency} for the uniform bounds below. The homogeneous Hessian
equation and \eqref{E507a} give
\[
 I_{H_i}(r_*)\le Ce^{-L^2/16}
\]
by interior elliptic estimates on a fixed neighborhood of this level.
Set $R_*=(1+\theta_3)L-1$.  Equation \eqref{E508d}, with $j=2,3$,
and the polynomial area bound for $\{b'=R_*\}$ give
\[
 D_{H_i}(R_*)_+
 \le R_*^{2-m}\int_{b'=R_*}|H_i|\,|\nabla H_i|
 \le CL^A.
\]
Here $I_{H_i}$ and $D_{H_i}$ are the quantities defined in
Appendix~\ref{app:A}, with dimension $m$; the area bound follows
from the uniform metric comparison in \eqref{EA512bb}.
On this exact-shrinker annulus, $R_1=R_2=\mathrm{R}(g')$
are uniformly bounded by \eqref{E506a}, and
$\langle\Delta_{f'}H_i,H_i\rangle\ge-C|H_i|^2$.
Thus Lemma~\ref{lem:finite-frequency}, with a fixed admissible $\delta$, implies
\begin{equation}\label{E508e}
 I_{H_i}(r)\le CL^A\left(e^{-L^2/16}
             +e^{-(R_*^2-r^2)/6}\right),
 \qquad r_*\le r\le R_*.
\end{equation}

Set $\theta_4=\theta_3/2$.  Since $|\nabla b'|\ge1/2$ on this
annulus, the coarea formula gives
\[
 \int_{s<b'<t}|H_i|^2\,\mathrm{d}V_{g'}
 \le4\int_s^t r^{m-1}I_{H_i}(r)\,\mathrm{d}r,
 \qquad r_*\le s<t\le R_*.
\]
Apply this bound on shells of width comparable to $L^{-1}$ and
use interior elliptic estimates for the homogeneous Hessian
equation on balls of the same scale.  On
$\{b'\le(1+\theta_4)L-1\}$ outside the fixed core, these balls lie
in $\{b'\le(1+3\theta_3/4)L\}$ for large $L$.  Together with
\eqref{E507a} on the fixed core, this yields
\[
 \|H_i\|_{C^k(\{b'\le(1+\theta_4)L-1\})}
 \le C_kL^{A_k}\left(e^{-L^2/32}
       +e^{-[R_*^2-(1+3\theta_3/4)^2L^2]/12}\right).
\]
For large $L$,
\[
 R_*^2-(1+3\theta_3/4)^2L^2\ge\frac{\theta_3}{4}L^2.
\]
Thus, with $a_H=\theta_3/64$, we obtain
\begin{equation}\label{E508f}
 \|\nabla^2v_i\|_{C^k(\{b'\le(1+\theta_4)L-1\})}
 \le C_k L^{A_k}e^{-a_HL^2},\qquad k\ge0.
\end{equation}

Normalize the functions by a bounded constant linear transformation
so that their gradients are almost orthonormal.  More explicitly,
choose a fixed large $s$ and put
$Q_{ij}=\int_{b'<s}\langle\nabla v_i,\nabla v_j\rangle\,\mathrm{d}V_{f'}$.
Polynomial growth and the Gaussian tail, together with
$\int\langle\nabla v_i,\nabla v_j\rangle\,\mathrm{d}V_{f'}
=\mu_i\delta_{ij}$, imply
$|Q-\tfrac12\mathrm{Id}|<1/10$ if $s$ and then $L$ are large.
Choose a symmetric invertible matrix $A$ with
$AQA^T=\mathrm{Id}\int_{b'<s}\,\mathrm{d}V_{f'}$, and set $u_i=A_{ij}v_j$.
The matrices $A,A^{-1}$ are uniformly bounded.
Using \eqref{E508f}, integrate the gradient of
$\langle\nabla u_i,\nabla u_j\rangle$ first in the fixed core and
then along the radial paths of length $O(L)$.  This gives
\begin{equation}\label{E508}
 |\langle\nabla u_i,\nabla u_j\rangle-\delta_{ij}|
 +\|\nabla^2u_i\|_{C^k}
 +|2\langle\nabla u_i,\nabla f'\rangle-u_i|
 \le C_kL^{A_k}e^{-a_HL^2}
\end{equation}
on $\{b'\le(1+\theta_4)L-1\}$, for every fixed $k$.
For the last term use
$2\langle\nabla v_i,\nabla f'\rangle-v_i
=2\Delta v_i+(2\mu_i-1)v_i$ before applying $A$.
All derivatives of this last error also satisfy estimates of the same
form, by \eqref{E508f}, the eigenfunction equation, and the curvature
derivative bounds.  The functions and their first derivatives are
polynomially bounded by \eqref{E508d}.

\textbf{Step 5: a central cross-section and a product chart.}
Write $U=(u_1,\ldots,u_d)$ and
$G_{ij}=\langle\nabla u_i,\nabla u_j\rangle$.
On a fixed neighborhood of $N\times\{0\}$, the original metric and
potential differ from the model by $Ce^{-L^2/16}$ in $C^4$.
Together with \eqref{E508}, this implies
\[
 \|U-Oz\|_{C^4}\le CL^Ae^{-a_HL^2}
\]
for some orthogonal matrix $O$.
Indeed, Poincar\'e and elliptic estimates on $N$, together with the
Hessian estimate, show that $U$ is close to an affine function of $z$.
The approximate Euler identity controls its constant term, and
$G\simeq\mathrm{Id}$ makes its linear part close to an orthogonal matrix.
The implicit function theorem gives a unique component of $\{U=0\}$
in this fixed neighborhood, which is a small graph over
$N\times\{0\}$.  Denote this central component by $N_0$.
Its induced metric $g_{N_0}$ has uniformly bounded curvature and
covariant curvature derivatives, by the Gauss equations,
\eqref{E506a}, and \eqref{E508}.

Define the exact dual fields
\[
 Y_i=\sum_j(G^{-1})_{ij}\nabla u_j.
\]
They satisfy $Y_i u_j=\delta_{ij}$.
Equation \eqref{E508} gives, for every fixed $k$,
\[
 |Y_i|\le C,\qquad
 |\nabla^kY_i|\le C_kL^{A_k}e^{-a_HL^2}\quad(k\ge1),
 \qquad
 |Y_if'-u_i/2|\le CL^Ae^{-a_HL^2}.
\]
For the last estimate, multiplying the approximate Euler identity
in \eqref{E508} by $G^{-1}$ introduces only a polynomial factor.

Let $\theta_5=\theta_4/2$ and $\theta_6=\theta_4/4$.
For $z\in\R^d$, put $Y_z=\sum_i z_iY_i$ and denote its flow by
$\varphi_{Y_z}^t$.  Define
\[
 F(y,z)=\varphi_{Y_z}^1(y),\qquad y\in N_0,\quad
 |z|\le(1+\theta_5)L.
\]
Along this flow one has $U(\varphi_{Y_z}^t(y))=tz$ exactly.
Integrating
\[
 \frac{\mathrm{d}}{\mathrm{d}t}\bigl(f'(\varphi_{Y_z}^t(y))-|tz|^2/4\bigr)
 =\sum_i z_i\bigl(Y_if'-u_i/2\bigr)(\varphi_{Y_z}^t(y))
\]
gives
\begin{equation}\label{E509}
 \left|f'(F(y,z))-n/2-|z|^2/4\right|
 \le CL^Ae^{-a_HL^2}.
\end{equation}
This also proves existence up to time one: the same estimate at an
intermediate time confines the trajectory to a compact subset of
$\{b'<(1+\theta_4)L-1\}$, where the fields are defined, since
$|tz|\le|z|$ and $\theta_5<\theta_4$.

The identity $U\circ F=z$ separates the $z$-directions, while the
restriction of $F$ at fixed $z$ is the restriction of a flow
diffeomorphism to $N_0$.  Thus $F$ is a local diffeomorphism.
It is injective: equality of two images first gives equality of $z$,
and applying $\varphi_{Y_z}^{-1}$ then gives equality of the initial
points in $N_0$.  Hence $F$ is an embedding on every closed domain
strictly inside its range of definition.

The image also contains the fixed core needed for the marking.
On each fixed model neighborhood, $U$ is $C^4$-close to $Oz$ and
$N_0$ is a small central graph.  If $x$ belongs to a fixed inner
model region, put $z=U(x)$ and follow $Y_z$ backward for time one.
The trajectory stays in a larger fixed region, and its endpoint
has $U=0$, hence belongs to $N_0$.  Thus $x=F(y,z)$.
This applies in particular to $x=\Psi_1^{-1}(p)$.

The variation formula $L_{Y_i}g'=2\operatorname{sym}
\nabla Y_i$ and the preceding derivative estimates show that
\begin{equation}\label{E510b}
 \|F^*g'-(g_{N_0}\times g_E)\|_{C^3}
 \le CL^Ae^{-a_HL^2}.
\end{equation}
Differentiating the parameter-dependent flow, the difference between
the $z_i$-coordinate field and $Y_i$ is controlled by the transported
commutators $[Y_z,Y_i]$. These and their fixed derivatives are
exponentially small. The same argument gives the higher-order bounds.

Finally choose a diffeomorphism $\varphi:N\to N_0$ putting the
induced metric in the slice
$\delta_{g_N}(\varphi^*g_{N_0}-g_N)=0$.
The graph construction and Theorem~\ref{TH202} give
\begin{equation}\label{E510d}
 \|\varphi^*g_{N_0}-g_N\|_{C^{2,1/2}}
 \le CL^Ae^{-a_HL^2}.
\end{equation}
The intrinsic curvature derivative bounds for $g_{N_0}$, the slice
condition, and elliptic regularity for the Ricci tensor give
\begin{equation}\label{E510e}
 \|\varphi^*g_{N_0}-g_N\|_{C^k}\le C_k,
 \qquad k\ge3.
\end{equation}
Together with the flow estimates these give the required
$C^k$ bounds for the product chart: for each fixed $k$, the factors
$L^{A_k}e^{-a_HL^2}$ have a bound independent of $L$.
For the potential the bounds follow
from the shrinker equation, its derivatives, and
$|\nabla f'|=O(L)$ on the region under consideration.

\textbf{Step 6: return to the original shrinker.}
Define
\[
 \Psi=\Psi_1\circ F\circ(\varphi\times\mathrm{Id})
 \quad\text{on }\bar D_{(1+\theta_6)L}.
\]
By \eqref{E509}, this composition defines an embedding on
$\bar D_{(1+\theta_6)L}$.
Equations \eqref{E509}, \eqref{E510b}, and \eqref{E510d} give
\[
 |\Psi^*g-\bar g|+|\Psi^*f-\bar f|
 \le CL^Ae^{-c_*L^2}
\]
with $c_*=a_H/2$, after absorbing the fixed polynomial factors by
increasing $L$.  Their $C^k$ norms are bounded by $C_kL^k$
for every $k\ge1$, by \eqref{E510e} and the preceding flow and
curvature estimates.
Choose
\[
 a_0=\frac{a_H}{8(1+\theta_6)^2},\qquad\theta=\theta_6.
\]
For any fixed $a_\dagger\in(0,a_0]$ and sufficiently large $L$,
$CL^Ae^{-c_*L^2}\le e^{-a_\dagger((1+\theta)L)^2}$.
The fixed-core argument in Step 5 shows that $p$ lies in the image.
Applying \eqref{E509} to its preimage and using $f(p)\le n/2+1$
gives $\bar b(\Psi^{-1}(p))\le 10n-1$ for large $L$.
Thus \eqref{eq:embedding-mark} holds, and the displayed estimates
prove $(\dagger_{(1+\theta)L})$ with the same marked point.
\end{proof}

\begin{thm}\label{thm:rigid1}
Let $(N^n,g_N)$ be a closed Einstein manifold satisfying $\Rc(g_N)=g_N/2$ and \eqref{EX301} and let $(\bar M,\bar g,\bar f)$ be a Ricci shrinker with $ \bar M=N^n \times \R^{m-n}$, $\bar g=g_N \times g_E$ and $\bar f=|z|^2/4+n/2$. Fix $\bar p\in N\times\{0\}$. Then any $m$-dimensional Ricci shrinker $(M,g,f,p)$ which is close to $(\bar M,\bar g,\bar f,\bar p)$ in the pointed Gromov--Hausdorff sense must be isometric to $(\bar M,\bar g)$.
\end{thm}

\begin{proof}
The cases $m\le3$ follow from the classification; hence assume $m\ge4$.
Suppose the conclusion fails.  Then there is a sequence of Ricci
shrinkers $(M_i^m,g_i,f_i,p_i)$ converging to
$(\bar M,\bar g,\bar f,\bar p)$ in the pointed Gromov--Hausdorff sense,
with no $(M_i,g_i)$ isometric to $(\bar M,\bar g)$.
The smooth limit has Euclidean $m$-dimensional tangent spaces, so
\cite[Proposition 5.8]{LLW21} gives a uniform lower entropy bound
for the sequence.  By \cite[Theorem 1.1]{LLW21}, the convergence
is therefore pointed smooth Cheeger--Gromov convergence.

Fix $a_\dagger\in(0,a_0]$ and choose an initial radius $L$
larger than the corresponding thresholds in
Theorems~\ref{thm:501} and \ref{thm:502}, and larger than $4/\theta$.
For all sufficiently large $i$, smooth convergence supplies
$(\star_L)$ and a marked point $p_i$ in the image of the model core,
with $f_i(p_i)\le n/2+1$.  Alternating extension and contraction,
while keeping this point fixed, gives embeddings satisfying
$(\star_{L_j})$, where
\[
 L_{j+1}=(1+\theta)L_j-2,\qquad L_j\longrightarrow\infty.
\]

Fix such an $i$, and denote these embeddings by $\Phi_j$.
On $\bar D_{L_j/2}$ their pullback metrics are uniformly comparable
to $\bar g$, with comparison constants tending to one.
Any path from $p_i$ leaving $\Phi_j(\bar D_{L_j/2})$ meets its
boundary.  The segment before its first boundary crossing pulls back
to a path whose model $\bar b$-coordinate increases from at most
$10n$ to $L_j/2$.  Since $\bar b$ is $1$-Lipschitz, metric comparison
gives constants $c,C>0$, independent of $j$, such that
\begin{equation}\label{eq:image-ball}
 B_{g_i}(p_i,cL_j-C)\subset\Phi_j(\bar D_{L_j/2}).
\end{equation}
On each fixed model domain the pullback metrics and potentials
converge to $\bar g$ and $\bar f$ in $C^5$.

The preimages $q_j=\Phi_j^{-1}(p_i)$ stay in the compact set
$\bar D_{10n}$.  Images of each fixed model domain therefore stay
in a bounded ball about $p_i$, which is compact by completeness.
The metric bounds and the coordinate formula for the pullback
connection yield, after taking a diagonal subsequence, a locally
$C^4$ limit $\Phi_\infty:\bar M\to M_i$ with
$\Phi_\infty^*g_i=\bar g$ and $\Phi_\infty^*f_i=\bar f$.
The resulting local isometry is smooth.

We check that it is a global isometry.  For fixed $x,y\in\bar M$,
a minimizing $g_i$-geodesic between $\Phi_j(x)$ and $\Phi_j(y)$
stays in a fixed ball about $p_i$.  For large $j$, this ball lies
in $\Phi_j(\bar D_{L_j/2})$ by \eqref{eq:image-ball}.
Pulling back the geodesic and using the metric comparison on that
whole domain gives
\[
 d_{g_i}(\Phi_j(x),\Phi_j(y))
 \ge (1-o(1))d_{\bar g}(x,y).
\]
The reverse inequality follows by pushing forward a fixed model
minimizing geodesic.  Passing to the limit proves that
$\Phi_\infty$ preserves distances, and in particular is injective.
Its image is open by the local isometry property and closed by
completeness of the model and distance preservation. Thus
$\Phi_\infty$ is surjective.
This contradicts the choice of the sequence.
\end{proof}

Next, we consider a finite quotient
$(\hat M,\hat g,\hat f)=(\bar M,\bar g,\bar f)/\Gamma$,
where $\Gamma$ is a finite group of potential-preserving model
isometries acting freely on $\bar M$.  Write
$\pi:\bar M\to\hat M$ for the quotient map and set
\[
 \hat D_L=\{\hat b\le L\}=\bar D_L/\Gamma,
 \qquad \hat b=2\sqrt{\hat f}.
\]
Every element of $\Gamma$ has the form
\[
 \gamma(y,z)=(a_\gamma y,A_\gamma z),
 \qquad a_\gamma\in\operatorname{Isom}(N,g_N),
 \quad A_\gamma\in\mathrm O(m-n).
\]
The compact factor has positive Ricci curvature, so the product splitting
is preserved by every model isometry.  Potential preservation removes
the translational part on the Euclidean factor.

\begin{thm}\label{thm:rigid2}
Any $m$-dimensional Ricci shrinker $(M,g,f,p)$ which is close to
$(\hat M,\hat g,\hat f,\hat p)$ in the pointed Gromov--Hausdorff
sense must be isometric to $(\hat M,\hat g)$.
\end{thm}

As in the proof of Theorem~\ref{thm:rigid1}, convergence to the smooth
quotient is smooth.
When $m=n$, this gives a diffeomorphism from the closed manifold
$\hat M$ onto the nearby shrinker.  Pull the metric back and
lift it to $N$.  The normalized slice map of
Theorem~\ref{TH202} can be chosen $\Gamma$-equivariantly:
the slice equation and the orthogonal complement of the Killing fields
are invariant, and its small solution is unique.  Compact rigidity
on $N$ identifies the lifted metric with $g_N$ in this slice, and
the equivariant identification descends to the quotient.  We henceforth
assume $d=m-n\ge1$.

Use the same marked-point convention and core radius $10n$ as above,
with $p\in M$ fixed throughout the iteration.  We formulate the
comparison conditions on the exact quotient model domains:
\begin{itemize}
\item[$(\star_L')$] There is a smooth embedding
$\hat\Phi_L:\hat D_L\to M$, whose image of
$\operatorname{int}\hat D_{10n}$ contains $p$, such that
\[
 [\hat\Phi_L^*g-\hat g]_5+
 [\hat\Phi_L^*f-\hat f]_5
 \le e^{\hat f/4-L^2/16}.
\]
\item[$(\dagger_L')$] There is a smooth embedding
$\hat\Psi_L:\hat D_L\to M$, with the same marked-point
condition, such that
\[
 |\hat\Psi_L^*g-\hat g|+
 |\hat\Psi_L^*f-\hat f|\le e^{-a_\dagger L^2},
\]
and the $C^k$ errors are bounded by $D_k'L^k$ for every $k\ge1$.
\end{itemize}

\begin{thm}\label{thm:501X}
There exists $L_1'>0$ such that if $(\dagger_L')$ holds and
$L>L_1'$, then $(\star_{L-2}')$ holds.
\end{thm}

\begin{proof}
Pull back the metric and potential by
$\hat\Psi_L\circ\pi$ to $\bar D_L$.  They are
$\Gamma$-invariant.  Apply the proof of Theorem~\ref{thm:501}
to this lifted geometry.  Its radial cutoffs are invariant, and the
weighted divergence operator, its orthogonal kernel projection, and
its inverse on the orthogonal complement commute with $\Gamma$.
Consequently every gauge vector field is invariant.  This also holds
for the constant center correction, since its coefficients transform
by the Euclidean representation $A_\gamma$ and the input data are
invariant.  The resulting gauge diffeomorphism therefore commutes
with $\Gamma$.

Restrict the corrected map to the exact domain $\bar D_{L-2}$.
The estimates of Theorem~\ref{thm:501} hold there, and the map
descends to an embedding of $\hat D_{L-2}$ into $M$.
The same interior displacement estimate preserves the marked point.
This proves $(\star_{L-2}')$.
\end{proof}

\begin{thm}\label{thm:502X}
There exist positive constants $a_0'\le a_0$ and $\theta'$, depending
on the fixed quotient, such that for each fixed
$a_\dagger\in(0,a_0']$ there is $L_2'=L_2'(a_\dagger)>0$
with the following property: if $(\star_L')$ holds with $L\ge L_2'$,
then $(\dagger_{(1+\theta')L}')$ holds with exponent $a_\dagger$.
\end{thm}

\begin{proof}
We carry out the construction of Theorem~\ref{thm:502} while
preserving the fixed deck action.  Starting with
$\hat\Phi_L(\hat D_L)$, take the interior potential level used
in Step~1 of that proof and extend its compact interior by the
outward gradient flow.  The curvature estimates hold on this
constructed domain, which we denote by $\mathcal U_L$:
\begin{equation}\label{E512a}
 |\nabla^k\operatorname{Rm}_g|\le C_k
 \quad\text{on }\mathcal U_L,\qquad k\ge0.
\end{equation}
The radial construction lifts the covering
$\hat\Phi_L\circ\pi$ over this collar.  Its lifted radial
identification and the interpolation on the inner collar commute
with $\Gamma$.  Thus the glued complete metric and potential
$(g',f')$ on $\bar M$ are $\Gamma$-invariant and satisfy all
bounds in Steps~2--4 of Theorem~\ref{thm:502}.

Put $d=m-n$.  Choose $\eta>0$ so that the model spectrum of
$-\Delta_{\bar f}$ in
$(1/2-2\eta,1/2+2\eta)$ consists precisely of its
$d$ Euclidean coordinate modes.  Such a choice is possible because
$\lambda_1(-\Delta_N)>1/2$, by the Lichnerowicz estimate, and
the next Euclidean eigenvalue is $1$.
The glued metric is uniformly $(1+C\theta_2)$-comparable to
$\bar g$ and $\|f'-\bar f\|_{C^0}=o(1)$ as $L\to\infty$.
The min--max comparison of the weighted Dirichlet forms therefore
shows, after decreasing $\theta_2$ if necessary, that the spectral projection
\[
 P=\mathbf 1_{(1/2-\eta,\,1/2+\eta)}(-\Delta_{f'})
\]
has rank $d$, with a uniform spectral gap separating its range
from the other eigenvalues.

Let $Z=(Z_1,\ldots,Z_d)$ be the vector of model coordinate
functions multiplied by the radial cutoff used in Step~3.
Since $\Gamma$ preserves $g'$, $f'$ and the cutoff,
\[
 Z(\gamma x)=A_\gamma Z(x),\qquad
 (PZ)(\gamma x)=A_\gamma(PZ)(x).
\]
The same weighted estimates as in Step~3 give
\[
 \|(-\Delta_{f'}-1/2)Z_i\|_{L^2}
 \le CL^A e^{-L^2/32}.
\]
The spectral gap implies
$\|PZ_i-Z_i\|_{L^2}\le CL^A e^{-L^2/32}$.
Interior elliptic estimates give the corresponding small error on
fixed core neighborhoods.  In particular the components of
$V=PZ$ form a basis of the low spectral space.

Apply Steps~3--4 to an orthonormal eigenbasis and transfer the
estimates to the components of $PZ$, which are bounded linear
combinations of that basis.  All estimates, including the Hessian estimate and
the approximate Euler identity in \eqref{E508}, are preserved.
Normalize $V$ using its gradient Gram matrix on the fixed invariant
region $\{b'<s\}$, as in Step~4.  That matrix commutes with
all $A_\gamma$, so its positive square root does as well.
The resulting vector $U=(u_1,\ldots,u_d)$ satisfies
\begin{equation}\label{E512c}
 U(\gamma x)=A_\gamma U(x)
\end{equation}
and the estimates \eqref{E508}.  On the fixed core $U$ is
exponentially close to $z$.

The unique small central graph $N_0=\{U=0\}$ is therefore
$\Gamma$-invariant, and its graph parametrization by $N$ is
equivariant.  Uniqueness of the normalized slice map makes
$\varphi:N\to N_0$ equivariant.
Let $F$ be the product map constructed in Step~5 from these functions.
Equation \eqref{E512c} implies
\[
 \gamma_*Y_z=Y_{A_\gamma z},\qquad
 F(\gamma y,A_\gamma z)=\gamma F(y,z).
\]
Together with the equivariance of $\varphi$ and the radial
identification, this proves that the new embedding on the covering
space descends to an embedding of the exact domain
$\hat D_{(1+\theta')L}$ into $M$.

Choose $a_0'\le\min\{a_0,c_*/(2(1+\theta')^2)\}$.
For each fixed $a_\dagger\in(0,a_0']$ and all sufficiently large $L$,
the estimates of Theorem~\ref{thm:502} give
\begin{equation}\label{E517c}
 |\hat\Psi^*g-\hat g|+|\hat\Psi^*f-\hat f|
 \le CL^A e^{-c_*L^2}
 \le e^{-a_\dagger((1+\theta')L)^2},
\end{equation}
and the $C^k$ errors are bounded by $D_k'L^k$.
The image contains the marked point, whose preimage lies in the
fixed core by the potential comparison.  This proves the theorem.
\end{proof}

\begin{proof}[Proof of Theorem~\ref{thm:rigid2}]
Suppose a sequence of complete shrinkers $(M_i,g_i,f_i,p_i)$ converges
in the pointed Gromov--Hausdorff sense to
$(\hat M,\hat g,\hat f,\hat p)$, with no $(M_i,g_i)$ isometric to
$(\hat M,\hat g)$.
Fix $a_\dagger\in(0,a_0']$ and
$L_0>\max\{L_1',L_2',4/\theta',20n\}$.
As in the proof of Theorem~\ref{thm:rigid1}, smooth convergence
supplies $(\star_{L_0}')$ with a marked point $p_i$ for all
sufficiently large $i$.
Fix such an $i$ and iterate Theorems~\ref{thm:501X}
and~\ref{thm:502X} to obtain $(\star_{L_k}')$, where
$L_{k+1}=(1+\theta')L_k-2\to\infty$.
The boundary-crossing and limiting-isometry arguments in the proof
of Theorem~\ref{thm:rigid1} apply to the quotient model and give a
global isometry between $(\hat M,\hat g)$ and $(M_i,g_i)$,
a contradiction.
\end{proof}

Theorems~\ref{thm:rigid1} and \ref{thm:rigid2}, together with
Gaussian rigidity, give the following eigenvalue gap.

\begin{cor}
Let $m\ge4$ and let $(M^m,g,f)$ be a Ricci shrinker with
$\mmu(g,1)\ge-A$.  Write
$0=\lambda_0<\lambda_1\le\lambda_2\le\cdots$ for the weighted
eigenvalues, counted with multiplicity and defined by
$\Delta_fu+\lambda u=0$.
There exists $\ep=\ep(m,A)>0$ such that, if
\begin{align*}
\lambda_{m-3} \le \frac{1}{2} +\ep,
\end{align*}
then $(M,g)$ is isometric to one of the following spaces with its
standard metric:
\begin{align*}
S^3/\Gamma \times \R^{m-3}, \quad S^2 \times \R^{m-2}, \quad (S^2 \times \R)/\Z_2 \times \R^{m-3}, \quad \R^m.
\end{align*}
Here $\Gamma\le\mathrm O(4)$ is a finite subgroup acting freely on
$S^3$, and the $\Z_2$ action on $S^2\times\R$ is generated by
either $(x,t)\mapsto(-x,t)$ or $(x,t)\mapsto(-x,-t)$.
\end{cor} 

\begin{proof}
Suppose the conclusion fails.  The weighted spectral gap is
$\lambda_1\ge1/2$, so there is a sequence of Ricci shrinkers
$(M_i^m,g_i,f_i)$ outside the stated list, with
$\mmu(g_i,1)\ge-A$ and $\lambda_{i,m-3}\to1/2$.

After passing to a subsequence, \cite[Theorem 1.1]{LLW21} gives
\begin{align}
(M_i, g_i, f_i,p_i) \longright{pointed-\hat{C}^{\infty}-Cheeger-Gromov} (M_{\infty}, g_{\infty}, f_{\infty},p_{\infty}), 
\label{E520a} 
\end{align}
where $p_i$ is a minimum point of $f_i$, and $(M_{\infty}, g_{\infty},f_{\infty}, p_{\infty})$ is a singular Ricci shrinker with singularities of codimension at least $4$. Moreover, the convergence \eqref{E520a} is smooth away from the singularities of $M_{\infty}$.

By \cite[Theorem 1.2]{LLW21}, the finite Gaussian measures, normalized to probabilities,
converge in the pointed measured Gromov--Hausdorff sense.
The shrinker equation and \cite[Theorem 7.2]{GMS15} give the
$\operatorname{RCD}(1/2,\infty)$ condition on the limit, and
\cite[Theorem 7.8]{GMS15} gives spectral convergence.
For $1\le j\le m-3$ we have
$1/2\le\lambda_{i,j}\le\lambda_{i,m-3}$, so the first $m-3$
positive eigenvalues of the limit all equal $1/2$.
By \cite[Theorem 1.1]{GKKO20}, the limit therefore splits
isometrically as $Y\times\R^{m-3}$, with the product Gaussian
measure on the Euclidean factor.

The splitting and $\dim_{\mathcal H}\mathcal S\le m-4$ force
$\mathcal S=\varnothing$. Thus $Y$ is a complete three-dimensional
Ricci shrinker. The three-dimensional
classification gives precisely the spaces in the statement.

For a non-Gaussian limit, Theorems~\ref{thm:rigid1} and
\ref{thm:rigid2} identify $(M_i,g_i)$ with that limit for all
sufficiently large $i$.  For a Gaussian limit, the same conclusion
follows from Gaussian rigidity \cite{Yo09,Yo12,LW20}.
Both cases contradict the choice of the sequence.
\end{proof}

\section{Uniqueness of the tangent flow}

As another application of the rigidity of $(\bar M, \bar g, \bar f)=(N^n \times \R^{m-n},g_N \times g_E, |z|^2/4+n/2)$ and the related techniques, we prove the uniqueness of the tangent flow if one tangent flow is $(\bar M, \bar g, \bar f)$. In the case of the mean curvature flow, the uniqueness of the tangent flow at each cylindrical singular point is proved in \cite{CM15,CM19}. Moreover, under a type-I curvature bound, the corresponding uniqueness of tangent flow in the setting of the Ricci flow is also proved; see \cite[Theorem 10.8]{CM21b}. 

In this section, we mainly focus on compact Ricci flows and their tangent flows. Throughout this section, we assume $(M^m,g(t))_{t \in [-T,0)}$ is an $m$-dimensional Ricci flow on a closed manifold $M$ such that $0$ is the singular time. We first recall the following definition of a tangent flow at the singular time, which is based on the theory of $\IF$-convergence established by Bamler \cite{Bam20a, Bam20b, Bam20c}. We refer readers to \cite[Definition 3.1]{Bam20a} and \cite[Definitions 3.13, 3.42, 5.1, 5.7]{Bam20b} for the precise definitions of the related concepts.

\begin{defn}\label{def:tangent}
Let $(\mu_t)_{t \in [-T,0)}$ be a conjugate heat flow on $(M^m,g(t))_{t \in [-T,0)}$ such that for any $t \in [-T,0)$,
\begin{align}\label{E601aa}
\emph{\Var}_t(\mu_t) \le H_m |t|,
\end{align}
where $H_m:=(m-1)\pi^2/2+4$. A tangent flow at $(\mu_t)$ is a metric soliton that is the $\IF$-limit of a sequence of metric flow pairs $(\XX^i, (\mu^i_t)_{t \in I})$, where $\XX^i$ is a metric flow over $I:=(-2,-1/2]$ induced by $(M,|t_i|^{-1} g(t|t_i|))$ and $\mu^i_t=\mu_{t|t_i|}$ for a sequence $t_i \nearrow 0$.
\end{defn}

Notice that for any sequence $t_i \nearrow 0$, if we define the metric flow pair $(\XX^i, (\mu^i_t)_{t \in I})$ as in the above definition, it follows from the $\IF$-compactness theorem (\cite[Theorems 7.4 and 7.6]{Bam20b}) that by taking a subsequence, 
\begin{equation} \label{Fconv}
(\XX^i, (\mu^i_t)_{t \in I}) \xrightarrow[i \to \infty]{\quad \IF, \CF \quad} (\XX^\infty, (\mu^\infty_t)_{t \in I}),
\end{equation}
for a correspondence $\CF$ (\cite[Definition 5.4]{Bam20b}) between the metric flows $\XX^i$, $i \in \mathbb N \cup \{\infty\}$. The limiting metric flow pair $(\XX^\infty, (\mu^\infty_t)_{t \in I})$ is future continuous (\cite[Definition 4.7]{Bam20b}) and $H_m$-concentrated (\cite[Definition 3.21]{Bam20b}) over $I$.

Since all metric flows $\XX^i$ are induced from a given compact Ricci flow $(M,g(t))$, the corresponding Nash entropies of $\mu^i_t$ are uniformly bounded below (cf. \cite[Section 2.3]{Bam20c}). Therefore, it follows from \cite[Theorems 2.3, 2.4 and 2.14]{Bam20c} that there exists a decomposition
\begin{equation*}
\XX^{\infty}=\mathcal R \sqcup \mathcal S, 
\end{equation*}
such that $\mathcal R$ is given by a Ricci flow spacetime $(\RR,g_t^{\infty})_{t \in I}$ in the sense of \cite[Definition 9.1]{Bam20b} and $\text{dim}_{\mathcal M^*}(\mathcal S) \le m-2$, where $\text{dim}_{\mathcal M^*}$ denotes the $*$-Minkowski dimension in \cite[Definition 3.31]{Bam20b}. 

Moreover, it follows from \cite[Theorem 2.6]{Bam20c} that the limiting metric flow pair $(\XX^\infty, (\mu^\infty_t)_{t \in I})$ must be a metric soliton (\cite[Definition 3.42]{Bam20b}). Roughly speaking, it means that all the time-slices $(\XX^{\infty}_t,|t|^{-\frac 1 2}d^{\infty}_t,\mu^{\infty}_t)$ are isometric as metric measure spaces. In addition, it follows from \cite[Addendum 2.7] {Bam20c} that on the regular part $\RR_t$ of the time-slice $\XX^{\infty}_t$,
\begin{align*} 
\Rc(g_t^{\infty})+\na^2_{g_t^{\infty}} f_t^{\infty}=\frac{g_t^{\infty}}{2|t|},
\end{align*}
where $f_t^{\infty}$ is determined by $\mu^{\infty}_t=(4\pi|t|)^{-\frac{m}{2}}e^{-f_t^{\infty}} \,\mathrm{d}V_{g_t^{\infty}}$. Furthermore, the restriction of $d^{\infty}_t$ to $\RR_t$ agrees with the intrinsic length metric of $g^{\infty}_t$, and $\XX^{\infty}_t$ is the metric completion of $\RR_t$ (cf. \cite[Theorem 2.4]{Bam20c}). The singular set $\MS_t$ has Minkowski dimension at most $m-4$ by \cite[Theorem 2.18]{Bam20c}.

Since $\XX^\infty$ is continuous (see \cite[Definition 4.7]{Bam20b}), we may assume that the convergence \eqref{Fconv} is uniform over $J:=\{-1\}$, by \cite[Theorem 7.6]{Bam20b}.

We first record a forward-continuation lemma.

\begin{lem}\label{lem:soliton-forward-continuation}
Let $(\XX,(\mu_t)_{t\in I})$ be a non-collapsed metric soliton obtained as an $\IF$-limit of smooth Ricci flows, with smooth convergence on its regular part. Suppose $p_i$ at time $-1$ converges to $p\in\RR_{-1}$ and, for some $r,T_-,T_+,K,v_0>0$ with $T_+<1/4$, the approximating flows satisfy
\[
 P^\circ(p_i;r,-T_-,T_+)\text{ is unscathed},\qquad
 |\Rm|\le K,\qquad |B(p_i,r)|\ge v_0.
\]
Assume the closed time interval lies in the interior of $I$. Then
$P^\circ(p;r,-T_-,T_+)$ is unscathed in $\RR$ and inherits the curvature bound. In particular, every smaller closed parabolic neighborhood has compact closure in $\RR$.
\end{lem}
\begin{proof}
Write $\tau=-t$ and
\[
 \mathrm{d}\mu_t=v_t\,\mathrm{d}V_{g_t},\qquad
 v_t=(4\pi\tau)^{-m/2}e^{-F_t-W},
\]
where $W$ is the entropy constant. By \cite[Theorem 2.18]{Bam20c}, the normalized potential satisfies on $\RR$
\begin{equation}\label{E6:soliton-identities}
 F_t=\tau(\mathrm{R}_t+|\nabla F_t|^2),\qquad
 \partial_tF_t=|\nabla F_t|^2,\qquad \mathrm{R}_t\ge0.
\end{equation}
Here $\partial_t$ is the time vector field of the Ricci flow spacetime.

By \cite[Theorem 9.24]{Bam20b}, there is $T_*\in(0,T_+]$ such that $P^\circ(p;r,-T_-,T_*)$ is unscathed and has the stated curvature bound. If $T_*<T_+$, the same theorem implies that $v_t$ tends uniformly to zero on every smaller transported initial ball as $t\nearrow-1+T_*$. On the Ricci-flow worldline $\gamma$ through $p$, however, \eqref{E6:soliton-identities} gives
\[
 \frac{\mathrm{d}}{\mathrm{d}t}\bigl((-t)F_t(\gamma(t))\bigr)
 =-(-t)\mathrm{R}_t(\gamma(t))\le0.
\]
Consequently, for $-1\le t<-1+T_*$,
\begin{equation}\label{E6:density-lower}
 F_t(\gamma(t))\le\frac{F_{-1}(p)}{-t},\qquad
 v_t(\gamma(t))\ge(4\pi)^{-m/2}
 \exp\left(-W-\frac{F_{-1}(p)}{1-T_+}\right)>0.
\end{equation}
This contradicts the vanishing assertion. Hence $T_*=T_+$. The last assertion follows by smooth transport of compact smaller initial balls.
\end{proof}

For the following extension proposition assume $m>n$. Let $\bar D_L=\{\bar b\le L\}$ and use the fixed core radius $10n$ from Section~5. For a singular metric soliton, fix $p\in\RR_{-1}$ with $f(p)\le n/2+1$, where $f=F_{-1}$ is normalized by \eqref{E101}. The conditions $(\star_L'')$ and $(\dagger_L'')$ are the embedding conditions $(\star_L)$ and $(\dagger_L)$ of Section~5, with target $\RR_{-1}$: the embeddings are defined on the exact source $\bar D_L$, their images contain $p$ in the image of $\operatorname{int}\bar D_{10n}$, and the same pullback estimates hold.

\begin{prop}\label{prop:singular-embedding-extension}
There are constants $L_*>0$ and $\theta>0$, depending only on the model and the fixed exponent $a_\dagger$, with the following property. If a non-collapsed metric soliton is an $\IF$-limit of compact Ricci flows and satisfies $(\star_L'')$ for $L\ge L_*$, then it satisfies $(\dagger_{(1+\theta)L}'')$. Every target domain used in the extension has compact closure in the regular part.
\end{prop}
\begin{proof}
Let $\Phi_L:\bar D_L\to\RR_{-1}$ be the given embedding. The pullback estimates on the fixed-width annulus about $\bar b=L-1$ show that the level
\[
 \{f\circ\Phi_L=(L-1)^2/4\}\cap\bar D_{L-1/2}
\]
is a unique radial graph over $\{\bar b=L-1\}$. Let $\Sigma$ be its image under $\Phi_L$, and let $U$ be the image of the compact region it bounds. The point $p$ lies in $\operatorname{int}U$, and $\partial U=\Sigma$.

There is a fixed sufficiently small $r_0>0$ such that every point of $\Sigma$ has a ball of radius $r_0$ contained in $\Phi_L(\operatorname{int}\bar D_{L-1/2})$, with uniformly bounded curvature and a uniform lower volume bound. This follows from metric comparison and the distance to the image boundary. Smooth convergence on this compact regular neighborhood transfers the estimates to the compact approximating flows. The backward pseudolocality theorem \cite[Theorem 2.47]{Bam20c}, together with forward pseudolocality as stated immediately after that theorem, gives uniformly bounded curvature on two-sided parabolic neighborhoods of a fixed size. Lemma~\ref{lem:soliton-forward-continuation} transfers their full forward lifetime to the limit. A finite cover of $\Sigma$ therefore supplies Ricci-flow worldlines from every point of $\Sigma$ on $[-1,-1+\delta]$, with a fixed $\delta>0$ independent of $L$. We decrease $\delta$ so that all these tubes have compact closure in $\RR$. Local Shi estimates give uniform curvature derivative bounds on smaller tubes.

Use the canonical homotheties from \cite[Theorem 2.18]{Bam20c} to identify these worldlines with curves $\psi^s$ in the time-$(-1)$ slice. They satisfy
\[
 \frac{\mathrm{d}}{\mathrm{d}s}\psi^s=\frac{\nabla f}{1-s},\qquad
 \frac{\mathrm{d}}{\mathrm{d}s}\bigl((1-s)f(\psi^s(x))\bigr)
 =-\mathrm{R}(\psi^s(x)).
\]
The curvature bound on the tubes implies $0\le \mathrm{R}(\psi^s(x))\le C$. Thus
\begin{equation}\label{E6:regular-collar}
 \frac{(L-1)^2/4-Cs}{1-s}
 \le f(\psi^s(x))
 \le\frac{(L-1)^2}{4(1-s)},\qquad x\in\Sigma.
\end{equation}
Choose $0<\delta_0<\delta$ and $\theta_1>0$ with
$(1+\theta_1)^2<(1-\delta_0)^{-1}$. For large $L$, each trajectory crosses $b=(1+\theta_1)L$ before $\delta_0$. On this collar,
\[
 |\nabla b|^2=1-\frac{4\mathrm{R}}{b^2}\ge\frac12.
\]
The hitting times are smooth, and uniqueness of the gradient flow and strict increase of $b$ show that the trajectories do not intersect. Together with $U$, they form a compact embedded regular domain with boundary on $b=(1+\theta_1)L$.

The radial interpolation and Steps~2--6 of Theorem~\ref{thm:502} apply in this constructed region and give an embedding on $\bar D_{(1+\theta)L}$ satisfying $(\dagger_{(1+\theta)L}'')$. The core coverage and potential comparison retain the marked point $p$.
\end{proof}

With these preparations, we prove the main result of this section.

\begin{thm}\label{thm:601}
Given a conjugate heat flow $(\mu_t)_{t \in [-T,0)}$ satisfying \eqref{E601aa} on $(M,g(t))_{t \in [-T,0)}$, if a tangent flow at $(\mu_t)$ is $(\bar M, \bar g,\bar f)$ as a metric soliton, then any tangent flow at $(\mu_t)$ is.
\end{thm}

\begin{proof}
If $m=n$, the tangent flow in the hypothesis is smooth and closed, and the conclusion follows from \cite{CMZ21a,SW15}. We therefore assume $m>n$.

For $s\in[-T/2,0)$, let $(\XX^s,(\mu^s_t)_{t\in I})$ be the metric flow pair induced by
\[
 (M,|s|^{-1}g(|s|t)),\qquad \mu^s_t=\mu_{|s|t}.
\]
By assumption, there is a sequence $t_i\nearrow0$ such that
\begin{align} \label{E602a}
\lim_{i \to \infty} d_{\IF}^J \lc (\XX^{t_i},(\mu^{t_i}_t)_{t \in I}), (\XX^{\infty},(\mu^{\infty}_t)_{t \in I})\rc=0,
\end{align}
where $(\XX^{\infty},(\mu^{\infty}_t)_{t \in I})$ is the metric soliton such that $\XX^{\infty}_{-1}=(\bar M, \bar g)$ and $\mu^{\infty}_{-1}=(4\pi)^{-\frac{m}{2}}e^{-\bar f-c}\,\mathrm{d}V_{\bar g}$. For the definition of the $\IF$-distance, see \cite[Definition 5.7]{Bam20b}. Here, $c=\boldsymbol{\mu}(\bar g,1)$ so that $\mu^{\infty}_{-1}$ is a probability measure (cf. \eqref{EX201b}). In particular, $\XX^{\infty}$ is smooth everywhere.

Suppose the conclusion fails. Then there exists another sequence $t_i' \nearrow 0$ such that
\begin{align*} 
\lim_{i \to \infty} d_{\IF}^J \lc (\XX^{t_i'},(\mu^{t_i'}_t)_{t \in I}), (\XX^{\infty,\prime},(\mu^{\infty,\prime}_t)_{t \in I})\rc=0,
\end{align*}
and for some constant $\ep_0>0$,
\begin{align*} 
d_{\IF}^J \lc (\XX^{\infty},(\mu^{\infty}_t)_{t \in I}), (\XX^{\infty,\prime},(\mu^{\infty,\prime}_t)_{t \in I})\rc \ge 2\ep_0.
\end{align*}

Without loss of generality, we assume for any $i \ge 1$,
\begin{align*}
d_{\IF}^J \lc (\XX^{t_i},(\mu^{t_i}_t)_{t \in I}), (\XX^{t'_i},(\mu^{t'_i}_t)_{t \in I})\rc \ge \ep_0.
\end{align*}

We define $\ep_j=2^{-j} \ep_0$. For any fixed $j \ge 1$, we choose by the continuity of the Ricci flow a time $s_{i,j}$ in between $t_i$ and $t_i'$ such that
\begin{align} \label{E602cc}
d_{\IF}^J \lc (\XX^{t_i},(\mu^{t_i}_t)_{t \in I}), (\XX^{s_{i,j}},(\mu^{s_{i,j}}_t)_{t \in I})\rc=\ep_j.
\end{align}
As discussed above, for any fixed $j \ge 1$, $(\XX^{s_{i,j}},(\mu^{s_{i,j}}_t)_{t \in I})$ $\IF$-subconverges to a metric soliton. By taking a diagonal sequence, we may assume that for any fixed $j \ge 1$, 
\begin{align} \label{E602d}
\lim_{i \to \infty} d_{\IF}^J \lc (\XX^{s_{i,j}},(\mu^{s_{i,j}}_t)_{t \in I}), (\XX^{\infty,j},(\mu^{\infty,j}_t)_{t \in I})\rc=0,
\end{align}
where $(\XX^{\infty,j},(\mu^{\infty,j}_t)_{t \in I})$ is a metric soliton. In particular, it follows from \eqref{E602cc} and \eqref{E602d} that
\begin{align*}
d_{\IF}^J \lc (\XX^{\infty,j},(\mu^{\infty,j}_t)_{t \in I}), (\XX^{\infty},(\mu^{\infty}_t)_{t \in I})\rc=\ep_j \xrightarrow{j \to \infty} 0.
\end{align*}

Next, we set $d_j^i=|s_{i,j}|^{-\frac 1 2}d_{g(s_{i,j})}$ and denote the distance functions on $\XX^{\infty,j}_{-1}$ and $\XX^{\infty}_{-1}$ by $d^{\infty}_j$ and $d^{\infty}$ respectively.

Fix $H=2H_m$ and choose an $H$-center $z_j$ of $\mu_{-1}^{\infty,j}$ for each $j$.

\textbf{Claim}: For every $D>0$, there is $J_D$ such that, for $j\ge J_D$, the ball $B_{d_j^{\infty}}(z_j,D)$ contains no singular point and has uniformly bounded curvature.

\emph{Proof of Claim}: Otherwise, after taking a subsequence, there are $x_j\in B_{d_j^{\infty}}(z_j,D)$ which are singular, or are regular with $|\Rm|(x_j)\to\infty$. In either case, the limiting regularity radii at these points tend to zero, where the radius is zero at singular points; see \cite[Lemma 15.16 and Corollary 15.47]{Bam20c}.

Choose $H$-centers $z_{i,j}$ of $\mu_{-1}^{s_{i,j}}$ and points $x_{i,j}$ strictly converging to $x_j$ in a correspondence for \eqref{E602d}; see \cite[Definition 6.12]{Bam20b}. Chebyshev's inequality at the $H$-centers and the $W_1$-convergence of the measures give a constant $D'$ independent of $i,j$ such that, for sufficiently large $i$,
\begin{equation}\label{E603b}
 d_j^i(x_{i,j},z_{i,j})\le D'.
\end{equation}
For a diagonal sequence with convergence errors tending to zero, smooth convergence holds on fixed model domains; see \cite[Theorem 9.21]{Bam20b} and \cite[Theorem 2.5]{Bam20c}. Choose a compact model core whose image carries more than two thirds of the approximating measure. The $H$-center estimate places more than two thirds in a fixed-radius ball about $z_{i,j}$, so $z_{i,j}$ remains a bounded distance from the core. Metric comparison and the first boundary crossing give a length bound $cR-C$ for paths from the core leaving the image of $\bar D_R$. Taking $R$ large therefore places the points in \eqref{E603b} in a fixed compact model set $K\subset\bar M$.

Let $\ep_{\mathrm{reg}}>0$ be the dimensional constant in \cite[Theorem 15.45(b)]{Bam20c}. By compactness and \cite[Theorem 15.45(c)]{Bam20c}, choose $\tau_0>0$ with $[-1-\tau_0,-1]\subset I$ such that
\[
 \NN_x(\tau_0)>-\ep_{\mathrm{reg}}/4\qquad(x\in K).
\]
Then choose $i_j\ge j$, using entropy continuity \cite[Theorem 2.10]{Bam20c}, so that the $\IF$-convergence errors tend to zero, \eqref{E603b} holds, and
\begin{equation}\label{E603bb}
 |\NN_{x_j}(\tau_0)-\NN_{x_{i_j,j}}(\tau_0)|<\ep_{\mathrm{reg}}/4.
\end{equation}
Set $s_j=s_{i_j,j}$ and $x'_j=x_{i_j,j}$. Then
\begin{equation}\label{E603c}
 (\XX^{s_j},(\mu_t^{s_j})_{t\in I})
 \longrightarrow(\XX^{\infty},(\mu_t^{\infty})_{t\in I})
 \quad\text{in }d_{\IF}^J.
\end{equation}
After taking a subsequence, $x'_j$ converges to $x\in K$. Entropy continuity gives $\NN_{x'_j}(\tau_0)>-\ep_{\mathrm{reg}}/2$ for large $j$, so \eqref{E603bb} implies $\NN_{x_j}(\tau_0)>-3\ep_{\mathrm{reg}}/4$. The limiting $\ep$-regularity theorem \cite[Theorem 15.45(b) and Corollary 15.47]{Bam20c} now makes every $x_j$ regular and bounds its regularity radius below by $\ep_{\mathrm{reg}}\sqrt{\tau_0}$. This excludes both singular points and unbounded curvature, and proves the claim.

After passing to a subsequence, the Claim and the metric-soliton structure give
\begin{align} \label{E604a}
\lc \RR_j,g_j, f_j, z_j \rc \longright{pointed-{C}^{\infty}-Cheeger-Gromov} (\bar M, \bar g, \bar f,z_\infty),
\end{align}
where $z_\infty\in\bar M$, $\mu^{\infty,j}_{-1}=(4\pi)^{-\frac{m}{2}} e^{-f_j-c_j}\,\mathrm{d}V_{g_j}$ and the constant $c_j$ is chosen so that \eqref{E101} holds, and $g_j$ is the metric on the regular part $\RR_j$ of $\XX^{\infty,j}_{-1}$. Notice that $(\RR_j,g_j, f_j)$ satisfies the Ricci shrinker equation \eqref{E100} with normalization \eqref{E101}. 

Fix a sufficiently large radius $L_0$, exceeding the thresholds for Theorem~\ref{thm:501}, Proposition~\ref{prop:singular-embedding-extension}, and $4/\theta$. The convergence \eqref{E604a} provides, for every sufficiently large $j$, an initial embedding on $\bar D_{L_0}$ satisfying $(\star_{L_0}'')$. Choose the marked point $p_j$ as the image of a point of $N\times\{0\}$; then $f_j(p_j)\le n/2+1$.

We now fix one such $j$. The proof of Theorem~\ref{thm:501} uses cutoffs on the source model and an exact shrinker equation only in the given image. It therefore gives
\[
 (\dagger_L'')\ \Longrightarrow\ (\star_{L-2}'')
\]
on $\RR_j$, without requiring its completeness. Proposition~\ref{prop:singular-embedding-extension} gives the reverse improvement with an enlarged radius. Iterating both steps, always with the same marked point $p_j$, yields embeddings
\[
 \Phi_k:\bar D_{L_k}\longrightarrow\RR_j,
 \qquad L_{k+1}=(1+\theta)L_k-2,
\]
satisfying $(\star_{L_k}'')$. Since $L_0>2/\theta$, the radii $L_k$ tend to infinity.

On $\bar D_{L_k/2}$, the estimate in $(\star_{L_k}'')$ implies
\[
 \frac12\bar g\le\Phi_k^*g_j\le2\bar g
\]
for all sufficiently large $k$, and $p_j\in\Phi_k(\bar D_{10n})$. Any path in $\RR_j$ from $p_j$ leaving $\Phi_k(\bar D_{L_k/2})$ first meets its boundary. Since $\bar b$ is $1$-Lipschitz, the pulled-back initial segment and the metric comparison give a length lower bound $cL_k-C$, with $c,C>0$ independent of $k$. The metric $d_j^{\infty}$ agrees on $\RR_j$ with its intrinsic length metric, so the intersection of its ball of this radius with $\RR_j$ lies in the compact image. Since $\RR_j$ is dense in its completion $\XX_{-1}^{\infty,j}$ and the compact image is closed there, the inclusion extends to the whole ball:
\begin{equation}\label{E6:exhaustion}
 B_{d_j^{\infty}}(p_j,cL_k-C)
 \subset\Phi_k(\bar D_{L_k/2})\subset\RR_j.
\end{equation}
Every point of the completed slice lies in one of these balls. Hence
$\XX_{-1}^{\infty,j}=\RR_j$, and its metric is complete.

This argument holds for every sufficiently large $j$. The embeddings $\Phi_k$ and the completeness just proved allow us to apply the limiting-isometry argument in the proof of Theorem~\ref{thm:rigid1}. Thus these time slices are isometric to $(\bar M,\bar g)$. The normalized soliton potentials agree after a translation of the Euclidean factor, and the probability measures therefore agree. The metric-soliton structure \cite[Theorem 2.18]{Bam20c} identifies the metric flow pairs with $(\XX^{\infty},(\mu_t^{\infty})_{t\in I})$. This contradicts their positive $d_{\IF}^J$-distance $\ep_j$ from that model and proves the theorem.
\end{proof}

\begin{rem}
The same proof applies to finite smooth quotients of $(\bar M,\bar g,\bar f)$, using the quotient-model construction and Theorem~\ref{thm:rigid2}. It also applies to complete non-compact Ricci flows with bounded curvature on compact time intervals and bounded Nash entropy, using the complete-approximant extension in the proof of Theorem~\ref{thm:602} and the corresponding non-collapsed compactness and regularity theory; see \cite{Bam21,LW23a}.
\end{rem}

Similarly, one can consider a complete ancient solution to the Ricci flow $(M^m,g(t))_{t \in (-\infty,0]}$ with bounded curvature on compact time-intervals and bounded entropy. More precisely, for any $T>0$, $\sup_{M \times [-T,0]} |\Rm|<\infty$. Moreover, we have
\begin{align} \label{E604xxa}
\inf_{t \le 0,\tau>0}\mmu(g(t),\tau) >-\infty.
\end{align}

Notice that \eqref{E604xxa} is equivalent to a uniform lower bound for the pointed Nash entropy (cf. \cite{CMZ21b}). We fix a point $(\bar p, \bar t) \in M \times (-\infty,0]$ and set $(\mu_t)_{t \in (-\infty,\bar t]}$ to be the conjugate heat flow based at $(\bar p, \bar t)$. Then for any sequence $\tau_i \to +\infty$, by taking a subsequence,
\begin{equation*}
(\XX^i, (\mu^i_t)_{t \in I}) \xrightarrow[i \to \infty]{\quad \IF, \CF \quad} (\XX^\infty, (\mu^\infty_t)_{t \in I}),
\end{equation*}
where $\XX^i$ is a metric flow over $I=(-2,-1/2]$ induced by $(M,\tau_i^{-1}g(\tau_i t+\bar t))$ and $\mu^i_t=\mu_{\tau_i t+\bar t}$. The metric soliton $(\XX^\infty,(\mu^\infty_t)_{t\in I})$ is called a tangent flow at infinity. It is independent of the base point $(\bar p,\bar t)$ for a fixed sequence of scales along which the limit exists; see \cite[Theorem 2.5]{CMZ21a}. We obtain the following result; see also \cite[Theorem 1.3]{BCDMZ22} for the $4$-dimensional steady soliton singularity models.
\begin{thm}\label{thm:602}
Let $(M^m,g(t))_{t \in (-\infty,0]}$ be a complete ancient solution to the Ricci flow with bounded curvature on compact time-intervals and bounded entropy. If a tangent flow at infinity is $(\bar M, \bar g,\bar f)$ as a metric soliton, then any tangent flow at infinity is.
\end{thm}
\begin{proof}
If $m=n$, smooth convergence to the closed model gives an embedding of the whole model into $M$. Its image is open and compact, so $M$ is closed. The uniqueness result for closed ancient Ricci flows in \cite{CMZ21a} applies. We therefore assume $m>n$.

Repeat the proof of Theorem~\ref{thm:601} with the scale parameter $\tau\to\infty$. Proposition~\ref{prop:singular-embedding-extension} also applies to complete approximating flows with bounded curvature on compact time intervals when their non-collapsed $\IF$-limit has the preceding soliton structure and smooth convergence on the regular part, by \cite[Theorem A.1 and Remark A.3]{Top10} and Lemma~\ref{lem:soliton-forward-continuation}. This supplies the extension step for the complete rescaled flows.
\end{proof}

\newpage

\appendixpage
\addappheadtotoc
\appendix
\section{Growth of eigentensors on perturbed Ricci shrinkers} 
\label{app:A} 
In this appendix, we slightly generalize the growth control of eigentensors on Ricci shrinkers as established by Colding and Minicozzi \cite[Section 4]{CM21b}.

Let $(M^n,g,\mathrm{d}V_f=e^{-f}\mathrm{d}V_g)$ be a smooth metric measure space, where $(M^n,g)$ is an $n$-dimensional complete Riemannian manifold, and $f$ is a smooth, proper, and positive real-valued function on $M$. Define $b=2\sqrt f$. Consider a smooth tensor field $T$ on $M$ that does not vanish outside any compact set and satisfies the inequality
\begin{align} \label{apeq:100}
\la \Delta_f T,T \ra \ge -\lambda |T|^2,
\end{align}
where $\lambda$ is a positive constant. 

For each positive regular value $r$ of $b$, define
\begin{align*}
I(r):=r^{1-n} \int_{b=r} |T|^2 |\na b|,\quad D(r):=\frac{r^{2-n}}{2} \int_{b=r} \la \na |T|^2,\frac{\na b}{|\na b|} \ra \quad \text{and} \quad U(r)=\frac{D(r)}{I(r)}
\end{align*}
where $U(r)$ is defined when $I(r)>0$ and the integrals are computed using the corresponding Hausdorff measures, with the volume forms omitted for simplicity.

Additionally, define
\begin{align}
R_1:=\frac{n}{2}-\Delta f \quad \text{and} \quad R_2:= f-|\na f|^2. \label{apeq:101a}
\end{align}
Notice that for Ricci shrinkers, both $R_1$ and $R_2$ represent the scalar curvature. We further assume the following growth condition:
\begin{align} \label{apeq:102a}
|R_1|+|R_2| \le \theta b^2+\Lambda,
\end{align}
where $\theta$ and $\Lambda$ are nonnegative constants.

The following growth control of $I(r)$ is analogous to \cite[Theorem 4.1]{CM21b}.
\begin{prop} \label{prop:ap1}
Let $\lambda>0$, $\Lambda \ge 0$, and $\theta \in [0,10^{-3})$. Suppose that \eqref{apeq:102a} holds, and that $T$ satisfies \eqref{apeq:100}, with $T, \Delta_f T \in L^2$ with respect to $\mathrm{d}V_f$. Set $R_*=100(\sqrt{\Lambda}+\sqrt{\lambda}+\sqrt n)$. For any $r_2>r_1>R_*$, we have
\begin{align} 
I(r_2) \le& I(r_1) \lc \frac{r_2}{r_1} \rc^{8\lambda+8\Lambda+n} e^{4\theta(r_2^2-r_1^2)}. \label{apeq:102}
\end{align}
\end{prop}
\begin{proof}
First, $T,\Delta_fT\in L^2$ imply $\nabla T\in L^2$ by the cutoff integration-by-parts argument in Lemma \ref{lem:ibp}.

From \eqref{apeq:101a} and \eqref{apeq:102a}, we have
\begin{align}\label{EAX:001}
||\na b|^2-1|=\left|\frac{|\na f|^2}{f}-1 \right|=\frac{|R_2|}{f} \le 4\theta+\frac{\Lambda}{f}.
\end{align}
By assumption, if $b \ge 10 \sqrt{\Lambda}$, it follows from \eqref{EAX:001} that
\begin{align}\label{EAX:002}
\frac{1}{2} \le |\na b| \le 2.
\end{align}
In particular, all sufficiently large levels of $b$ are smooth compact hypersurfaces. Suppose $I(r)=0$ for some $r>R_*$. Then $T$ vanishes on $\{b=r\}$. Completing the square for compactly supported cutoffs $\chi T$ gives
\[
\int_{b>r}|\nabla(\chi T)|^2e^{-f}
\ge\int_{b>r}\left(\frac14|\nabla f|^2-\frac12\Delta f\right)\chi^2|T|^2e^{-f}.
\]
By \eqref{apeq:101a} and \eqref{apeq:102a},
\[
\frac14|\nabla f|^2-\frac12\Delta f
=\frac{b^2}{16}-\frac n4+\frac{R_1}{2}-\frac{R_2}{4}
\ge\left(\frac1{16}-\frac\theta2\right)b^2-\frac n4-\frac\Lambda2
\ge\frac{b^2}{32}-\frac n4-\frac\Lambda2.
\]
Letting the cutoffs tend to $1$, using $T,\nabla T\in L^2$ and Fatou's lemma for this potential bounded below, yields
\begin{align}\label{apeq:exterior_dirichlet}
\int_{b>r}|\nabla T|^2e^{-f}
\ge\left(\frac{r^2}{32}-\frac n4-\frac\Lambda2\right)
\int_{b>r}|T|^2e^{-f}.
\end{align}
Integration of \eqref{apeq:100} over $\{b>r\}$, with cutoffs and the zero boundary trace, gives
\[
\int_{b>r}|\nabla T|^2e^{-f}\le\lambda\int_{b>r}|T|^2e^{-f}.
\]
Since $r>R_*$ implies $r^2/32-n/4-\Lambda/2>\lambda$, estimate \eqref{apeq:exterior_dirichlet} forces $T=0$ on $\{b>r\}$, contrary to the assumption. Thus $I(r)>0$ for $r>R_*$.

Additionally, using \eqref{apeq:101a} and direct computation, we obtain the following relations:
\begin{align}
b\Delta b=& n-|\na b|^2-2R_1, \label{apeq:103a} \\
|\na b|^2 =& 1-4b^{-2}R_2. \label{apeq:103b}
\end{align}

From the definition of $I(r)$, we compute:
\begin{align}
I(r)=& \int_{b<r} \textrm{div}(|T|^2 b^{1-n} \na b) \notag \\
=&\int_{b<r} b^{1-n} \left \{ \la \na|T|^2,\na b \ra+|T|^2 \lc \Delta b-\frac{(n-1)|\na b|^2}{b} \rc \right\} \notag \\
=&\int_{b<r} b^{1-n} \left \{ \la \na|T|^2,\na b \ra+\frac{|T|^2}{b^3} (4nR_2-2b^2R_1) \right\} \label{EA:001}
\end{align}
where the last equality holds because, by \eqref{apeq:103a} and \eqref{apeq:103b},
\begin{align*}
b\Delta b-(n-1)|\na b|^2=n(1-|\na b|^2)-2R_1=b^{-2}(4nR_2-2b^2R_1).
\end{align*}

Using \eqref{EA:001} and the coarea formula, we find:
\begin{align*}
I'(r)=r^{1-n} \int_{b=r} \la \na |T|^2, \frac{\na b}{|\na b|} \ra+r^{-2-n} \int_{b=r} \frac{|T|^2}{|\na b|} (4nR_2-2b^2R_1)
\end{align*}
and consequently:
\begin{align}
r(\log I(r))'=2U(r)-\frac{2r^{1-n}}{I(r)} \int_{b=r } \frac{|T|^2}{|\na b|} (R_1-2nr^{-2}R_2 ).
\label{EA:002a}
\end{align}

If $r \ge 10 \sqrt{\Lambda}$, it follows from \eqref{apeq:102a}, \eqref{EAX:002} and \eqref{EA:002a} that
\begin{align} \label{EA:002b}
r(\log I(r))' \le 2U(r)+8\theta r^2+8\Lambda+n.
\end{align}

On the other hand, from
\begin{align*}
D(r)=\frac{r^{2-n}e^{\frac{r^2}{4}}}{2} \int_{b<r} (\Delta_f |T|^2) e^{-f}
\end{align*}
and the coarea formula, we compute:
\begin{align}
D'(r)=&\frac{2-n}{r} D(r)+\frac{r}{2} D(r)+\frac{r^{2-n}}{2} \int_{b=r} \frac{\Delta_f |T|^2}{|\na b|} \notag \\
\ge & \frac{2-n}{r} D(r)+\frac{r}{2} D(r)+r^{2-n} \int_{b=r} \frac{|\na T|^2-\lambda |T|^2}{|\na b|}, \label{EA:003}
\end{align}
where we have used \eqref{apeq:100} for the last inequality.

If $D(r)>0$, it follows from \eqref{EA:003} that
\begin{align}
r (\log D(r))' \ge & 2-n+\frac{r^2}{2}+\frac{r^{3-n}}{D(r)} \int_{b=r} \frac{|\na T|^2}{|\na b|}-\frac{\lambda r^{3-n}}{D(r)} \int_{b=r} \frac{|T|^2}{|\na b|} \notag \\
=& 2-n+\frac{r^2}{2}-\frac{\lambda r^2}{U(r)}+\frac{r^{3-n}}{D(r)} \int_{b=r} \frac{|\na T|^2}{|\na b|}-\frac{4\lambda r^{1-n}}{D(r)} \int_{b=r} \frac{|T|^2R_2}{|\na b|}, \label{EA:003a}
\end{align}
where the last equality holds because, by \eqref{apeq:103b},
\begin{align*}
\frac{|T|^2 b^2}{|\na b|}=|T|^2 b^2|\na b|+4\frac{|T|^2R_2}{|\na b|}.
\end{align*}

By Cauchy-Schwarz inequality, we have
\begin{align*}
D^2(r)=\lc \frac{r^{2-n}}{2} \int_{b=r} \la \na |T|^2,\frac{\na b}{|\na b|} \ra \rc^2 \le I(r) r^{3-n} \int_{b=r} \frac{|\na T|^2}{|\na b|}.
\end{align*}

From \eqref{EA:003a}, if $D(r)>0$, it follows that
\begin{align}
r (\log D(r))' \ge 2-n+\frac{r^2}{2}+U(r)-\frac{\lambda r^2}{U(r)}-\frac{4 \lambda r^{1-n}}{D(r)} \int_{b=r} \frac{|T|^2R_2}{|\na b|}. \label{EA:003b}
\end{align}
Combining \eqref{EA:002a} and \eqref{EA:003b}, we obtain that if $U(r)>0$,
\begin{align}
r (\log U(r))' \ge& 2-n+\frac{r^2}{2}-U(r)-\frac{\lambda r^2}{U(r)} \notag \\
&+\frac{2r^{1-n}}{I(r)} \int_{b=r} \frac{|T|^2}{|\na b|} \left \{ R_1-\lc 2nr^{-2}+\frac{2\lambda}{U(r)} \rc R_2 \right \}. \label{EA:004}
\end{align}

If $U(r)>4 \lambda$ and $r >10\sqrt{\Lambda}+10 \sqrt{n}$, then from the assumptions on $R_1$ and $R_2$,
\begin{align*}
\left|R_1-\lc 2nr^{-2}+\frac{2\lambda}{U(r)} \rc R_2 \right| \le |R_1|+|R_2| \le \theta r^2+\Lambda.
\end{align*}
Substituting this into \eqref{EA:004}, we find:
\begin{align}\label{EA:005}
(\log U(r))' \ge& \frac{2-n-U(r)}{r}+r \lc \frac{1}{2}-\frac{\lambda}{U(r)} \rc-8\theta r-8\frac{\Lambda}{r} \notag \\
 \ge & \frac{2-n-U(r)-8\Lambda}{r}+ (1/4-8\theta) r. 
\end{align}

\textbf{Claim}: If $U(r_0)>4 \lambda$ for some $r_0 >100(\sqrt{\Lambda}+\sqrt{\lambda}+\sqrt{n})$, then
\begin{align*}
U(r) \ge r^2/5
\end{align*}
for sufficiently large $r$.

\emph{Proof of the Claim}: At a radius $r>R_*$ with $U(r)=4\lambda$, the lower bound in \eqref{EA:005}, by continuity from $U>4\lambda$, is strictly positive. Thus, once $U(r_0)>4\lambda$, it cannot subsequently cross down through $4\lambda$. Whenever $4\lambda<U(r)\le r^2/5$, the same estimate gives
\begin{align*}
(\log U(r))'
\ge (2-n-8\Lambda)r^{-1}+(1/4-8\theta-1/5)r
\ge\frac r{100}.
\end{align*}
It follows by integration that $U$ cannot stay below $r^2/5$ for all large $r$. Moreover, at a crossing with $U(r)=r^2/5$, this bound is greater than $(\log(r^2/5))'=2/r$, since $r>R_*$. Therefore $U$ cannot cross back below $r^2/5$, proving the claim.

From \eqref{EA:003b} and \eqref{apeq:102a}, we have:
\begin{align} \label{EA:007}
r (\log D(r))' \ge 2-n+\frac{r^2}{2}+U(r)-\frac{\lambda r^2}{U(r)}-\frac{16 \lambda}{U(r)} (\theta r^2+\Lambda)
\end{align}

Using the claim and \eqref{EA:007}, if $U(r) >4 \lambda$ for some $r \ge 100(\sqrt{\Lambda}+\sqrt{\lambda}+\sqrt{n})$, then
\begin{align*}
r (\log D(r))' \ge \frac{3}{5} r^2
\end{align*}
for sufficiently large $r$. For any $t>s$ sufficiently large, this implies
\begin{align}
D(t) \ge  D(s)e^{\frac{3(t^2-s^2)}{10}}. \label{EA:008}
\end{align}
From \eqref{EA:008}, we conclude that
\begin{align} \label{EA:009}
2\int_{b \le t}(|\na T|^2+\la \Delta_f T,T \ra) e^{-f}=\int_{b \le t} \Delta_f |T|^2 e^{-f}=2e^{-\frac{t^2}{4}}t^{n-2}D(t) \to \infty \quad \text{as } t \to \infty.
\end{align}

This contradicts the integrability established at the start of the proof. Therefore, $U(r)\le4\lambda$ for every $r>R_*$. Combining this with \eqref{EA:002b}, the conclusion \eqref{apeq:102} follows immediately.
\end{proof}

\begin{rem} \label{rem:inp}
Let $(M^n,\bar g,\bar f)$ be a Ricci shrinker with normalization \eqref{E101} and bounded Ricci curvature. Suppose that a smooth symmetric $(0,2)$-tensor $h$ and a smooth function $\psi$ satisfy
\begin{align*}
\|h\|_{C^1}+\|\psi\|_{C^2} \le \theta,
\end{align*}
where the norms are taken with respect to $\bar g$ and $\theta<1/2$ is sufficiently small. Set
\[
g:=\bar g+h,\qquad f:=\bar f+\psi+1,\qquad b:=2\sqrt f.
\]
Then $g$ is complete and $f$ is positive and proper. Direct estimates give
\begin{align*}
| n/2-\Delta f | \le C\theta b+C \quad \text{and} \quad |f-|\na f|^2| \le C\theta b^2+C,
\end{align*}
where $C$ depends only on $n$ and the Ricci curvature bound of $\bar g$. Thus the geometric hypotheses of Proposition \ref{prop:ap1} hold for sufficiently small $\theta$.
\end{rem}

\begin{lem}\label{lem:finite-frequency}
Let $\lambda>0$, $\Lambda\ge0$, and $0<\delta<\min\{1,\lambda\}$.
There is $r_0=r_0(n,\lambda,\Lambda,\delta)>0$ with the following property.
Let $g$ be a smooth Riemannian metric and $f$ a smooth positive function
on a neighborhood of a compact annulus $\{r_1\le b\le R\}$, where
$b=2\sqrt f$, $R>r_1\ge r_0$, and the boundary consists of the two
indicated levels. Suppose that
\[
 |R_1|+|R_2|\le\Lambda.
\]
If a smooth tensor $T$ satisfies
$\langle\Delta_fT,T\rangle\ge-\lambda|T|^2$ on the annulus, then
\begin{equation}\label{apeq:finite-frequency}
 I(r)\le\left(\frac r{r_1}\right)^p I(r_1)
       +\frac{2}{p}D(R)_+e^{-(R^2-r^2)/6},
 \qquad r_1\le r\le R,
\end{equation}
where $p=4\lambda+2\delta+8\Lambda$ and $D(R)_+=\max\{D(R),0\}$.
The same assertion holds for sections of a metric vector bundle equipped
with a metric connection, with $\Delta_f$ the corresponding weighted
connection Laplacian.
\end{lem}

\begin{proof}
The identity $|\nabla b|^2=1-4R_2/b^2$ gives
$|\nabla b|^2\ge1/2$ after increasing $r_0$, so all levels are regular.
Applying the divergence theorem between two levels, as in the
computations leading to \eqref{EA:002a} and \eqref{EA:003}, gives
\begin{align*}
 rI'&=2D-2r^{1-n}\int_{b=r}\frac{|T|^2}{|\nabla b|}
                   \left(R_1-\frac{2n}{r^2}R_2\right),\\
 rD'&\ge\left(2-n+\frac{r^2}{2}\right)D
       +r^{3-n}\int_{b=r}\frac{|\nabla T|^2-\lambda|T|^2}{|\nabla b|}.
\end{align*}
Using $r^{1-n}\int_{b=r}|T|^2/|\nabla b|\le2I$ and $r^2\ge2n$,
the first identity yields $rI'\le2D+4\Lambda I$.
Cauchy--Schwarz gives
\[
 D^2\le I r^{3-n}\int_{b=r}\frac{|\nabla T|^2}{|\nabla b|}.
\]
Together with the identity
\[
 r^{3-n}\int_{b=r}\frac{|T|^2}{|\nabla b|}
 =r^2I+4r^{1-n}\int_{b=r}\frac{R_2|T|^2}{|\nabla b|},
\]
the second inequality gives, wherever $I>0$,
\[
 rD'\ge\left(2-n+\frac{r^2}{2}\right)D
       +\frac{D^2}{I}-\lambda r^2I-8\lambda\Lambda I.
\]

Set $q=2\lambda+\delta$ and $K=D-qI$.
If $K(r)>0$, then $I(r)>0$: otherwise $T$ vanishes on that level
and hence $D(r)=0$. Substituting $D=K+qI$ into the last inequality
and subtracting $q$ times the inequality for $rI'$ gives
\begin{align*}
 rK'\ge{}&\left(\frac{r^2}{2}+2-n\right)K+\frac{K^2}{I}\\
 &+\left[\frac{\delta r^2}{2}+q(2-n)-q^2
                    -8\lambda\Lambda-4q\Lambda\right]I
 \ge\frac{r^2}{3}K
\end{align*}
on $\{K>0\}$, after increasing $r_0$.
Thus $e^{-r^2/6}K(r)$ is nondecreasing on each positive component
of $K$, which cannot end at a zero before $R$. Since $K(R)\le D(R)$,
\[
 K(r)_+\le D(R)_+e^{-(R^2-r^2)/6},\qquad r_1\le r\le R.
\]
As $p=2q+8\Lambda$, it follows that
$rI'\le pI+2D(R)_+e^{-(R^2-r^2)/6}$.
Set $J(r)=I(r)-(2/p)D(R)_+e^{-(R^2-r^2)/6}$.
\[
 rJ'-pJ\le\left(4-\frac{2r^2}{3p}\right)
 D(R)_+e^{-(R^2-r^2)/6}\le0
\]
after taking $r_0^2\ge6p$. Integrating gives
$J(r)\le(r/r_1)^pJ(r_1)$, which implies
\eqref{apeq:finite-frequency}.
\end{proof}

    \vskip10pt
    
Yu Li, Institute of Geometry and Physics, University of Science and Technology of China, No. 96 Jinzhai Road, Hefei, Anhui Province, 230026, China; Hefei National Laboratory, No. 5099 West Wangjiang Road, Hefei, Anhui Province, 230088, China; E-mail: yuli21@ustc.edu.cn. \\
    
    Wenjia Zhang, School of Mathematical Sciences, University of Science and Technology of China, No. 96 Jinzhai Road, Hefei, Anhui Province, 230026, China; \url{wj12345678@mail.ustc.edu.cn}.\\

    \end{document}